%% file: TRB_revision.tex
\documentclass[english,11pt]{elsarticle} 

\usepackage[T1]{fontenc}
\usepackage{float}
\usepackage{amsmath}
\usepackage{amsthm}
\usepackage{amssymb}
\usepackage{amsfonts}
\usepackage{bbm}
\usepackage{tabulary}
\usepackage{booktabs}
\usepackage{longtable}
\usepackage{graphicx}
\usepackage{epstopdf}
\usepackage[pdftex,colorlinks,linkcolor={blue},citecolor={blue}]{hyperref}
\usepackage{algorithm,algorithmic}
\newcommand{\algorithmicbreak}{\textbf{break}}
\newcommand{\BREAK}{\STATE \algorithmicbreak}
\usepackage{subcaption}
\usepackage{cases}

 \usepackage{pgfplots}
  \pgfplotsset{compat=newest}
  \usetikzlibrary{plotmarks}
  \usetikzlibrary{arrows.meta}
  \usepgfplotslibrary{patchplots}
  \usepackage{grffile}
  \usepackage{amsmath}
   \usepackage{caption}
   \usepackage{tablefootnote}

\usepackage{nomencl}

\input{format}

\makenomenclature 

\usepackage{etoolbox}
\renewcommand\nomgroup[1]{%
  \item[\bfseries
  \ifstrequal{#1}{A}{Decision variables of the TNC platform}{%
  \ifstrequal{#1}{B}{Decision variables of the transit agency}{%
  \ifstrequal{#1}{C}{Endogenous variables}{%
  \ifstrequal{#1}{D}{Auxiliary variables in the solution method}{%
  \ifstrequal{#1}{E}{Exogenous parameters}{}}}}}%
]}
     
\begin{document}
\begin{frontmatter}
 \title{\textbf{Regulating For-Hire Autonomous Vehicles for An Equitable Multimodal Transportation Network}}

\author[1staddress]{Jing Gao}
\ead{jgaoax@connect.ust.hk}
\author[1staddress,2ndaddress]{Sen Li}
\ead{cesli@ust.hk}

\address[1staddress]{Department of Civil and Environmental Engineering, The Hong Kong University of Science and Technology, Hong Kong, China}
\address[2ndaddress]{Intelligent Transportation Thrust, Systems Hub, The Hong Kong University of Science and Technology (Guangzhou), Guangzhou, China}

\begin{abstract}
This paper assesses the equity impacts of for-hire autonomous vehicles (AVs) and investigates regulatory policies that promote spatial and social equity in future autonomous mobility ecosystems. To this end, we consider a multimodal transportation network, where a ride-hailing platform operates a fleet of AVs to offer mobility-on-demand services in competition with a public transit agency that offers transit services on a transportation network. A game-theoretic model is developed to characterize the intimate interactions between the ride-hailing platform, the transit agency, and multiclass passengers with distinct income levels. An algorithm is proposed to compute the Nash equilibrium of the game and conduct an ex-post evaluation of the performance of the obtained solution. Based on the proposed framework, we evaluate the spatial and social equity in transport accessibility using the Theil index, and find that although the proliferation of for-hire AVs in the ride-hailing network improves overall accessibility, the benefits are not fairly distributed among distinct locations or population groups, implying that the deployment of AVs will enlarge the existing spatial and social inequity gaps in the transportation network if no regulatory intervention is in place. To address this concern, we investigate two regulatory policies that can improve transport equity: (a) a minimum service-level requirement on ride-hailing services, which improves the spatial equity in the transport network; (b) a subsidy on transit services by taxing ride-hailing services, which promotes the use of public transit and improves the spatial and social equity of the transport network. We show that the minimum service-level requirement entails a trade-off: as a higher minimum service level is imposed, the spatial inequity reduces, but the social inequity will be exacerbated. On the other hand, subsidies on transit services improve accessibility for low-income households in underserved areas. In certain regimes, the subsidy increases public transit ridership and simultaneously bridges spatial and social inequity gaps. These results are validated through realistic numerical studies for San Francisco.
\end{abstract}

\begin{keyword}
transportation network company, ride-hailing platform, autonomous vehicles, public transit, regulatory policies, transport equity.
\end{keyword}

\end{frontmatter}

\section{Introduction}
Access to transportation is a crucial factor in determining quality of life. In highly populated locations, public transit can provide mobility services on a large scale, connecting individuals to jobs, educational opportunities, nutritious food, and medical visits \cite{sanchez2004transit}. In low-density locations, however, public transit agencies find it financially difficult to maintain an acceptable level of service, resulting in restricted coverage or transit gaps \cite{borjesson2020rural}. In these underserved locations, rich inhabitants can still use their own vehicles to go around, but disadvantaged groups (e.g., low-income individuals) are more likely to rely on public transportation and have fewer options to meet their fundamental travel needs. 
This transit equity gap hurts us all: when our neighbors are isolated and struggling, it undermines the local economy and weakens the social cohesion of the city.

However, transportation is on the verge of dramatic upheaval.  Autonomous vehicle (AV) technology is being tested and deployed in numerous cities around the world. While level-5 fully autonomous driving could be decades away, ride-hailing services with level-4 AVs, offered by transportation network companies (TNCs) in urban and suburban areas, are expected to arrive much earlier \cite{Waymo2022SF}. Automation will bring radical changes to TNC business models: (a) compared to human drivers who selfishly reposition themselves to pursue individual earning opportunities, AVs can be centrally dispatched to achieve system-level objectives \cite{zhang2016control}; (b) compared to human drivers who share a significant portion of the trip fare, AVs can remove the costs associated with human drivers and reduce the trip fare substantially \cite{chen2016operations}. These favorable factors have profound implications on transport equity: with better control of their fleets and lower operational costs, TNCs can improve mobility services in underserved areas and enhance transportation access for disadvantaged populations while preserving financial sustainability.

Despite its potential to mitigate inequities, the actual equity impact of AV deployment in the TNC market remains uncertain. In general, transport equity embraces the concept of fair or equal distribution of impacts (either costs or benefits) across different groups of network users \cite{litman2002evaluating}. This distribution can be viewed from a vertical perspective, referred to as vertical/social equity, which requires that the allocation of benefits and costs favors disadvantaged groups, e.g., people that are disadvantaged in social-economic status (e.g., income, education, age, gender, ethnicity, etc.) should be favored by transport policies. On the other hand, the horizontal perspective, referred to as horizontal/spatial equity, involves analyzing the distribution of impacts across various spatial locations or trip movements and requires that people in different regions of the city should have access to the same quality of mobility services. Unlike public transportation, AV technology is being led by multi-billion dollar companies in the private sector. The primary goal of these companies is to maximize profits, which is naturally contradicted by the equity objectives of the social planner. As a result, the deployment of AV through TNC platforms may negatively affect spatial equity
and social equity
within the multimodal transportation network. On the spatial side, for-profit ride-hailing companies (e.g., Uber, Waymo, Baidu, etc.) tend to focus their services in regions with high demand \cite{hughes2016transportation}. These high-demand areas are often congested communities in the city core, which already have robust public transit services. The geographic concentration of TNC services will further enlarge the accessibility gap between wellserved and underserved areas, which exacerbates spatial inequity. On the social side, as for-hire AVs flock into these areas, they will compete with public transit, degrade transit service qualities, and disproportionately affect low-income individuals that are more transit-dependent \cite{hall2018uber,erhardt2022transportation}. This will enlarge the accessibility gap between passengers with distinct socioeconomic status (e.g., level of income),  which exacerbates social inequity. These issues can be further reinforced as AV technology enables TNCs to have lower costs and stronger control over their fleet. Without regulatory interventions, market forces may drive us to a transport equity catastrophe: public transit ridership will decline, traffic congestion will increase, people who are already transportation-disadvantaged will be left further behind, and the benefits of AVs will be inaccessible to those who need them most.

This paper aims to assess the equity impacts of for-hire AVs in the ride-hailing network and investigate regulatory policies that ensure AV deployment can benefit underserved areas and disadvantaged groups. We consider a TNC platform that deploys a fleet of AVs to offer mobility-on-demand services to maximize its profit, a public transit agency that offers transit services to maximize transit ridership under a fixed budget, and a group of multiclass passengers with different income levels who can choose among ride-hailing, public transit, and bundled services that combines ride-hailing and public transit to reach their destinations. To maintain an equitable multimodal transportation network, it is crucial to investigate: (a) whether AV deployment will exacerbate spatial and social inequity in an unregulated environment; (b) whether it is necessary to intervene with AV regulations; and (c) how to design regulatory policies to mitigate the negative impacts of AVs and improve transport equity. To this end, we will develop a game-theoretic network equilibrium model to examine how travelers with heterogeneous socioeconomic attributes make mode choices in the multimodal transportation system, how the TNC platform and the transit agency interact with each other to reach their own objectives, and how these interactions are affected by the imposed regulations. The major contribution of this paper is summarized below:
\begin{itemize}
    \item We develop a game-theoretic model to characterize the strategic interactions between the TNC platform and the public transit agency over the transport network. The model captures the essential economic aspects of the multimodal transportation network, including the spatial prices of the TNC platform, the service fare and frequencies of public transit, the spatial distribution of idle AVs, waiting times for ride-hailing and transit services, the boarding/alighting station choices of public transit riders, the disutility/generalized costs of distinct mobility modes, and the modal choice of multiclass passengers with distinct income levels. We use the model to evaluate the spatial and social equity in transport access among distinct population groups. {\em To our best knowledge, this is the first work that studies spatial and social equity for multimodal transportation networks considering the strategic interaction between TNC platforms and public transit agencies.}
    \item We solve the game problem via best response methods, and theoretically characterize the solution to the game through an approximate Nash equilibrium and propose a computational framework to rigorously bound the error of this approximation, despite that the optimization problem of each player in the proposed game-theoretic model is highly non-convex. In particular, for TNC's profit-maximization subproblem, we use primal decomposition to compute the globally optimal solution to a relaxed reformulation of the problem, which establishes a tight upper bound to evaluate the performance of TNC's decisions.
    For the transit ridership maximization subproblem, we adopt a grid-based search algorithm to compute its globally optimal solution.
    The developed upper bound enables us to assess the quality of the derived candidate Nash equilibrium solution by showing that it is at least as good as an $\epsilon$-Nash equilibrium \cite{daskalakis2006note, li2019connections}, where $\epsilon$ can be determined numerically through the above procedure. 
    \item We define accessibility as the expected maximum utility and use the Theil index to quantify both spatial and social inequity to investigate the equity impacts of AV deployment in the absence of regulations. We find that although the proliferation of AVs improves overall accessibility, the benefits are not fairly distributed across different geographic locations and among distinct population groups, and the deployment of AVs will exacerbate both spatial and social inequity gaps at the same time. We point out that the increase in spatial inequity arises from the geographic concentration of ride-hailing services in high-demand areas due to the for-profit nature of the TNC platform, and the increased social inequity gap is due to the fact that the benefits of AV-enabled ride-hailing services are primarily enjoyed by  individuals with higher income, whereas those with lower income  are disproportionately transit-dependent.
    \item We evaluate the equity impacts of two regulatory policies, including (a) a minimum service-level requirement on ride-hailing services, which improves the spatial equity in the transport network; (b) a subsidy on transit services by taxing ride-hailing services, which promotes the use of public transit and improves spatial and social equity of the transport network. We show that the minimum service-level requirement entails a trade-off: as a higher minimum service level is imposed, the spatial inequity reduces, but the social inequity will be exacerbated. On the other hand, subsidies on transit services specifically benefit low-income households in underserved areas. In certain regimes, the subsidy increases public transit ridership and simultaneously bridges spatial and social inequity gaps. The reasons for the above results are identified and carefully discussed in Section \ref{service} and \ref{subsidies}, respectively.
\end{itemize}

The remainder of this article will proceed as follows.  Related works in the existing literature will be surveyed in Section \ref{related works}. A game-theoretic model will be developed for the competition between TNC and transit in Section \ref{game-theoretic model}. A solution method will be proposed in Section \ref{solution method}, and numerical studies will be presented in Section \ref{market outcomes}-\ref{regulations}. Finally, concluding remarks are offered in Section \ref{conclusion}.

\section{Related Works} \label{related works}

There is a large body of literature that investigates the operational strategies of for-hire AV fleets, examines the interaction between TNC and public transit, and quantifies their impacts on transport equity. Below we summarize the relevant literature from three aspects: (a) autonomous mobility-on-demand services; (b) interactions between TNC and public transit; and (c) impacts of AVs and TNCs on transport equity.

\subsection{Autonomous Mobility-on-Demand Services}

A recurrent topic of interest in the literature is the planning and operation of for-hire AV fleets to provide autonomous mobility-on-demand (AMoD) services. \cite{zhang2016control} presented a queuing-theoretic model to investigate the optimal rebalancing strategies of AMoD fleets. Their case study of New York showed that the current taxi demand in Manhattan can be met with about 8,000 AVs (roughly 70\% of the size of the current taxi fleet operating in Manhattan). \cite{wen2019value} investigated the value of demand information in AMoD systems and showed that aggregate demand information can lead to better service performance. \cite{zhang2016model} developed a model predictive control (MPC) approach to optimize vehicle scheduling and routing in AMoD systems. \cite{tsao2018stochastic} utilized short-term probabilistic forecasts for dispatching and rebalancing of AMoD systems and developed a stochastic MPC algorithm to efficiently solve the problem. \cite{alonso2017predictive} proposed a predictive method for the vehicle routing and assignment for an AMoD system with ridesharing. \cite{iglesias2019bcmp} formulated AMoD systems as a closed, multi-class Baskett-Chandy-Muntz-Palacios (BCMP) queuing network model to characterize the passenger arrival process, traffic, the state-of-charge of electric vehicles, and the availability of vehicles at the stations and derive the routing and charging policies. \cite{hyland2018dynamic} used agent-based simulations to evaluate the performance of AV-traveler assignment strategies. Their results showed that the spatial distribution of traveler requests significantly impacts the empty fleet miles generated by shared AVs. \cite{gueriau2018samod,guo2020deep} investigated the ride-sharing vehicle dispatching as well as the request-vehicle assignment problem for AMoD systems using reinforcement learning. \cite{chen2016operations} examined the operation of shared electric AV fleets under distinct vehicle range and charging infrastructure settings. \cite{turan2020dynamic} investigated the joint routing, battery charging, and pricing for electrified AMoD systems and used deep reinforcement learning to develop a near-optimal control policy. For a comprehensive literature review on the operation of AMoD systems, please refer to \cite{zardini2022analysis}.
{\em Note that all the aforementioned works primarily focus on the planning and operation of AMoD systems, while the interaction between AMoD and other modes is not explicitly considered.} 

\subsection{Interactions Between AMoD and Public Transit}

An increasing body of literature revealed the interaction between AMoD/ridesourcing services and public transit using both data-driven and model-based methods. Based on real-world data, \cite{hall2018uber} estimated the effect of Uber on public transit ridership using a difference-in-differences design. Their results showed that Uber is a complement to the average transit agency, increasing ridership by five percent after two years. On the other hand, \cite{meredith2021relationship} analyzed detailed individual trip records at a spatially and temporally granular level, and found that only approximately 2\% of TNC trips complement public transit, while 45\% to 50\% of TNC trips substitute transit and this percentage drops during COVID-19 shutdowns. Based on realistic TNC data and public transit data, \cite{erhardt2022transportation} found that TNCs are responsible for a net ridership decline of about 10\%, offsetting net gains from other factors such as service increases and population growth. From the modeling perspective, \cite{ke2021equilibrium} proposed a user equilibrium based mathematical model to explore the complement and substitution of ride-sourcing to public transit. \cite{zhu2020analysis} established a network equilibrium model to analyze the multimodal commute behavior with ridesplitting programs as both feeders and competitors to public transit. \cite{zhu2021competition} developed a bilevel game-theoretic approach to model the cooperative and competitive relationship between the TNC and the government. They highlighted that a carefully designed subsidy can benefit both the TNC and society, especially in areas with low public transit accessibility. \cite{salazar2019intermodal} and \cite{gurumurthy2020first} considered AMoD services for first/last-mile connection to public transport. \cite{salazar2018interaction} presented a network flow model to capture the interaction between AMoD and public and transit, and designed a pricing and tolling scheme that achieves the social optimum. \cite{pinto2020joint}, \cite{sieber2020improved} and \cite{kumar2022planning} studied the joint planning and operation of shared AV fleet and public transit to improve mobility services in low-density areas.

This paper significantly differs from the aforementioned works in two aspects: (a) the aforementioned studies primarily consider the multimodal transportation system at the aggregate level, and in contrast, our work considers a multimodal transportation network with multiclass travelers; (b) all aforementioned studies neglect the equity impacts of ride-hailing services in the multimodal transportation systems. {\em We believe our paper is the first to investigate the spatial and social equity implications of the strategic interaction between AMoD services and public transit over a transport network.}

\subsection{Impacts of AVs and TNCs on Transport Equity}

Emerging technologies, such as AVs and TNCs, have profoundly disrupted the transportation sector, but their impacts on transport equity remain unclear \cite{dianin2021implications}. On the one hand, AVs and TNCs have great potential to close existing inequity gaps \cite{shaheen2017travel}. \cite{brown2018ridehail} revealed that black riders were 73 percent more likely than white riders to have a taxi trip canceled and waited for 6-15 minutes longer than white riders, while TNC services nearly eliminate the racial-ethnic discrimination in service quality. \cite{palm2021equity} found significant academic evidence that new mobility technologies like ride-hailing have the potential to meaningfully address some disadvantaged travelers’ transportation problems or close existing gaps in transit services like the first/last mile. \cite{wen2018transit} simulated and evaluated integrated AV and public transportation systems and showed that encouraging ride-sharing, allowing in-advance requests, and combining fare with transit help enable service integration and encourage sustainable travel. \cite{chen2017connecting} compared two different relative spatial position (RSP) designs in an integrated e-hailing/fixed-route transit system and showed the great potential of e-hailing in filling transit gaps. \cite{cohn2019examining} and \cite{nahmias2021benefits} examined the equity impacts of AVs and showed that AMoD services could expand transportation access to car-free and underserved populations and bridge the mobility gaps between demographic groups. \cite{ahmed2020quantifying} revealed that shared AV mobility services can significantly increase job accessibility and low-income workers in low-density areas enjoy a higher benefit. On the other hand, there are growing concerns that TNCs and AVs may widen existing inequity gaps. \cite{ge2016racial} and \cite{yang2021equitable} investigated racial discrimination in TNCs and reported that customers with color experience more trip cancellations and longer waiting times. \cite{hughes2016transportation} studied the spatial distribution of TNC services and found that waiting time is lower in densely populated areas. In addition, the deployment of AV technology could also raise new issues such as induced demand \cite{cohn2019examining} and exacerbate geographic concentration \cite{jiao2021shared}. Given these concerns, it is widely agreed that regulatory policies should be imposed to ensure that AV technology promotes equity in future mobility systems \cite{milakis2017policy}.

{\em Distinct from most aforementioned works, we aim to reveal the fundamental mechanisms by which market incentives lead to systemic transport inequity under the interaction between TNC and public transit.} We will develop a mathematical framework that characterizes how market forces drive TNC services and public transit in the absence of regulatory intervention and quantifies both spatial and social inequity in the multimodal transportation system.

\newpage
\section{The Game-Theoretic Model} \label{game-theoretic model}


Consider a city divided into $M$ geographic zones and let $\mathcal{V}$ be the set of zone centroids. These zones are connected by a multimodal transport network which consists of a road network $\mathcal{G}_r(\mathcal{V}_r,\mathcal{E}_r)$ and a transit network $\mathcal{G}_t(\mathcal{V}_t,\mathcal{E}_t)$, where $\mathcal{V}_r$ denotes the set of road nodes corresponding to the intersections in the road network, $\mathcal{E}_r$ denotes the set of road links, $\mathcal{V}_t$ represents the set of transit stations, and $\mathcal{E}_t$ denotes the set of transit links. More specifically, the transit network consists of a set of transit lines $\mathcal{L}$. Each transit line $l\in \mathcal{L}$ is composed of a set of stations $\mathcal{V}_t^l \subset \mathcal{V}_t$ which are connected by a set of links $\mathcal{E}_t^l \subset \mathcal{E}_t$\footnote{In the link representation of the transit network, each transit link belongs to a single transit line, and distinct transit lines have no overlaps in transit links.}. Each geographic zone $i \in \mathcal{V}$ contains/covers a set of transit stations $\mathcal{V}_t^i \subset \mathcal{V}_t$. We consider passenger trips at the zonal granularity where zone centroids represent the origins and destinations of travelers and each trip is assigned with an origin zone $i\in\mathcal{V}$ and a destination zone $j\in\mathcal{V}$.
The travel demand is subdivided into $K$ classes based on income levels in each zone. In the multimodal transport system, the TNC platform operates a fleet of AVs to provide mobility-on-demand services on the road network, and the public transit agency provides transit services through the transit network, and passengers in different classes make mode choices over the multimodal transport network. The decision-making of passengers, the TNC platform, and the public transit agency interact with each other and constitute the market equilibrium. This section formulates a mathematical model to capture the competition between the TNC platform and the public transit agency over a transportation network. The details of the model are presented below.

\subsection{Incentives of Passengers}

In the multimodal transport network, passengers can take direct AMoD services, public transit, or bundled services combining AMoD and transit to reach their destinations. We employ the random utility model to characterize travelers' mode choice behavior and use a multinomial logit model to capture the passenger demand over distinct modes. Below we derive the disutility/cost of distinct mobility options and characterize the travel demand for the multimodal transport network.


\subsubsection{Generalized costs of AMoD trips} \label{sec_generalized_cost_AMoD}
The cost of passengers choosing AMoD services (indexed by $a$) comprises waiting time, travel time, and monetary cost. We define the generalized travel cost of direct AMoD trips as the weighted sum of waiting time, in-vehicle time and trip fare:
\begin{equation} \label{generalized_cost_a}
    c_{ij,k}^a = \alpha_k w_i^a + \beta_k \frac{l_{ij}^a}{v_a} + \gamma_k (b + r_i^a l_{ij}^a ) ,
\end{equation}
where $c_{ij,k}^a$ is the average disutility/generalized travel cost for travelers choosing direct AMoD services from origin zone $i$ to destination zone $j$ in income class $k$, $w_i^a$ is the average waiting time for AMoD services in zone $i$, $b$ is the average base fare of AMoD rides, $r_i^a$ is the average per-distance rate of AMoD rides originating from zone $i$, $l_{ij}^a$ denotes the average trip distance from zone centroid $i$ to zone centroid $j$, $v_a$ represents the average speed of TNC vehicles, $\alpha_k$, $\beta_k$ and $\gamma_k$ are the cost weights associated with waiting time, in-vehicle time and trip fare for travelers in income class $k$, which correspond to travelers' valuation on waiting time, trip time and monetary cost, respectively. 

The passenger waiting time for AMoD services $w_i^a$ depends on passenger demand and the vehicle supply in the TNC market. To characterize $w_i^a$, note that each trip is initiated by a passenger who requests a pickup from zone $i$ through the user app (for either direct AMoD service or first-mile/last-mile service). Upon receiving the request, the platform immediately matches the passenger to the closest idle vehicle in the same zone. The matched idle vehicle then travels to pick up the waiting passenger. In this case, the TNC platform does not prioritize either type of AMoD service in the vehicle-passenger matching process, and passengers are served on a first-come-first-served basis. Consequently, the direct AMoD rides and AMoD rides as first/last mile connections are indifferent in passenger waiting time. At the stationary state, the matching rate $m^{c-v}_i$ between passengers and vehicles depends on the number of waiting passengers $N_i^a$ and the number of idle vehicles $N_i^I$ in distinct zones\footnote{We assume that the matching rate depends on the idle vehicles and waiting passengers within the same zone. This is a reasonable assumption if the platform matches passengers with vehicles within a certain matching radius smaller than the size of the zone, or when there exists a large number of densely distributed idle vehicles. The case of inter-zone matching is left for future work.}. We introduce the Cobb-Douglas production function to capture the passenger-vehicle matching process \cite{yang2011equilibrium,zha2016economic}, which can be written as:
\begin{equation} \label{cobb_douglas_function}
    m_i^{c-v} = M^{c-v}(N_i^a,N_i^I) = A_i(N_i^a)^{a_1}(N_i^I)^{a_2} ,
\end{equation}
where $A_i$ is the zone-specific scaling parameter that captures possible factors in the matching between idle TNC vehicles and passengers, such as the size of the zone, the average traffic speed on roads, and the demand/supply distribution, and $a_1$ and $a_2$ are the elasticity parameters. Based on Little's law, the average number of waiting passengers/idle TNC vehicles equals the average arrival rate $\lambda_i$ multiplied by the average waiting time ($w_i^a$/$w_i^v$) that a passenger/vehicle spends in the system before being matched:
\begin{equation} \label{littles_law}
\begin{cases}
    N_i^a = \lambda_i w_i^a \\
    N_i^I = \lambda_i w_i^v
\end{cases} .
\end{equation}
Substituting (\ref{littles_law}) and the stationary condition\footnote{The matching rate should equal both the arrival rate of idle vehicles and the arrival rate of waiting passengers in distinct zones.} $m_i^{c-v}=\lambda_i$ into (\ref{cobb_douglas_function}) and setting $a_1=1$ and $a_2=0.5$ yields the average passenger waiting time as:
\begin{equation} \label{waiting_time_AMoD}
    w_i^a = \frac{1}{A_i \sqrt{N_i^I}} .
\end{equation}
Equation (\ref{waiting_time_AMoD}) indicates that passengers' average waiting time is inversely proportional to the square root of the number of idle vehicles in distinct zones, which has been widely used in street-hailing taxi market \cite{douglas1972price}, radio dispatching taxi market \cite{arnott1996taxi}, and online ride-hailing market \cite{li2021spatial}. 

\subsubsection{Generalized costs of transit trips} \label{sec_generalized_cost_transit}
Passengers choosing public transit (indexed by $p$) endure waiting time, travel time, monetary cost, and the access and egress cost of walking. Each zone probably contains multiple transit stations. Therefore,  passengers may board the transit network at different stations in the origin zone and exit the transit network at distinct stations in the destination zone. Passengers selecting different stations to board and alight have different access and egress distances and thereby different walking times. Besides, different boarding and alighting stations in the transit network may also correspond to distinct transit paths and incur different waiting times, in-vehicle times and trip fares. We first characterize the generalized costs of direct transit trips with distinct boarding/alighting stations. To this end, let $s_i \in \mathcal{V}_{t}^i$ be the boarding station in origin zone $i$ and $s_j \in \mathcal{V}_{t}^j$ be the boarding station in destination zone $j$, respectively. The generalized travel cost of direct transit trips with boarding station $s_i$ and alighting station $s_j$ can be defined as:
\begin{equation} \label{generalized_cost_p_station}
    c_{s_i s_j,k}^p = \alpha_k w_{s_i s_j,k}^p + \beta_k \frac{l_{s_i s_j,k}^p}{v_p} + \gamma_k r^p l_{s_i s_j,k}^p + \theta_k \left(\frac{d_{s_i}}{v_w} + \frac{d_{s_j}}{v_w} \right) ,
\end{equation}
where $c_{s_i s_j,k}^p$ is the average disutility/generalized travel cost for travelers choosing direct public transit from zone $i$ to zone $j$ with boarding station $s_i$ and alighting station $s_j$ in income class $k$, $w_{s_i s_j,k}^p$ represents the average waiting time from station $s_i$ to station $s_j$ for income class $k$ in the transit network, $r^p$ is the average per-distance transit fare\footnote{We consider an average per-distance fare to represent the fare structure of major metro systems, in which the fare calculation is based on the stations passengers enter and exit and the distance-based transit fare serves as a reasonable approximation.}, $l_{s_i s_j,k}^p$ denotes the average trip distance from station $s_i$ to station $s_j$ for income class $k$ in the transit network, $d_{s_i}$ is the access/egress distance from zone centroid $i$ to station $s_i$, $v_p$ and $v_w$ are the average operating speed of public transit and the average walking speed respectively, $\alpha_k$, $\beta_k$, $\gamma_k$ and $\theta_k$ are the cost weights associated with waiting time, in-vehicle time, trip fare, and walking time for travelers in income class $k$. The average waiting time $w_{s_i s_j,k}^p$ and trip distance $l_{s_i s_j,k}^p$ are endogenously related to passengers' route decisions in the transit network, which are further affected by the operational decisions of the public transit agency. We extend the hyperpath transit assignment model proposed by \cite{spiess1989optimal} to capture the route decision in the transit network and characterize the endogenous variables $w_{s_i s_j,k}^p$ and $l_{s_i s_j,k}^p$. Below we detail the hyperpath transit assignment model.


Consider an origin-destination station pair $s_i s_j$ in the transit network. Due to the multiplicity of the transit network, there may be multiple transit paths that connect station $s_i$ and station $s_j$, each of which may incur distinct waiting times, in-vehicle times, and trip fares for passengers. Besides, a transit path may traverse multiple transit lines due to transfers between different lines. Passengers wait for transit services at boarding and transfer stations, at which the expected waiting time depends on the frequencies of traversed lines/links. Let $f_l$ be the frequency of transit line $l \in \mathcal{L}$ and $f_a$ be the frequency of transit link $a \in \mathcal{E}_t$. The frequency of a transit link is determined by the frequency of the transit line it belongs to, i.e., $f_a = f_l$ if $a \in \mathcal{E}_t^l$. Let $\mathcal{E}_t^{s+}$ and $\mathcal{E}_t^{s-}$ denote the set of outgoing links and incoming links at station $s \in \mathcal{V}_t$, respectively. To characterize $w_{s_i s_j,k}^p$ and $l_{s_i s_j,k}^p$, we impose the following assumptions:


\begin{assumption} \label{assump_exponential}
    The vehicle inter-arrival time of a transit line $l \in \mathcal{L}$ follows an exponential distribution with rate $f_l$\footnote{We assume an exponential distribution of vehicle inter-arrival times for mathematical tractability. It is a widely used simplifying assumption in transit assignment literature \cite{spiess1989optimal,nguyen1988equilibrium} to serve as a reasonable approximation. Other exact distributions of vehicles' inter-arrival times that consider the minimum headway between trains (e.g., a truncated propability distribution) can be also incorporated into the model, but we leave this as the future work.}.
\end{assumption}
\begin{assumption} \label{assump_minimize}
    Passengers make route decisions in the transit network to minimize the expected generalized travel cost. At any station $s$, passengers waiting for transit services have selected a set of attractive lines $\mathcal{L}_s$, and they can be served by the first arriving vehicle from these attractive lines.
\end{assumption}
Assumptions \ref{assump_exponential} and \ref{assump_minimize} state that the vehicle inter-arrival times of transit line $l$ follow an exponential distribution with the mean of its headway $1/f_l$, and upon the arrival of the first vehicle from the set of transit lines, passengers can get on the vehicle immediately without waiting. Based on Assumptions \ref{assump_exponential} and \ref{assump_minimize}, the expected waiting time and the line probabilities at station $s$ relate to the frequencies through:
\begin{equation}
    w_s^p = \frac{1}{\sum\limits_{l \in \mathcal{L}_s} f_l} ,
\end{equation}
\begin{equation}
    \mathbb{P}_s^l = \frac{f_l}{\sum\limits_{l' \in \mathcal{L}_s} f_{l'}}, 
\end{equation}
where $w_s^p$ is the expected waiting time at station $s$ in the transit network and $\mathbb{P}_s^l$ is the probability that the passenger boarding line $l$ at station $s$. 

Consider a passenger that determines the sets of attractive lines at distinct stations (referred to as strategy \cite{spiess1989optimal} or hyperpath \cite{nguyen1988equilibrium}) to minimize the expected generalized travel cost from boarding station $s_i$ to alighting station $s_j$. To find the optimal hyperpath, we introduce the auxiliary variables $g_s, s\in \mathcal{V}_t$, which satisfy:
\begin{equation}
    g_s = \begin{cases}
1, \quad \text{if } s=s_i \\
-1, \quad \text{if } s=s_j \\
0, \quad \text{otherwise} .
    \end{cases}
\end{equation}
Intuitively, $g_s$ accounts for the net outflow of passengers, i.e., the difference between outflow and inflow, at distinct stations. The net outflow is 1 and -1 at the boarding/origin station $s_i$ and the alighting/destination station $s_j$, respectively. While at other stations, the inflow and outflow should be balanced.
Passengers in distinct income classes have heterogeneous cost weights of waiting time, in-vehicle time and trip fare in generalized costs of transit trips. Therefore, we formulate the transit assignment problem to identify the optimal hyperpath for distinct income classes. Let $l_a^p$ be the length of link $a \in \mathcal{V}_t$, and denote $v_{ak}$ and $w_{sk}^p$ as the flow of passengers in income class $k$ on link $a\in \mathcal{E}_t$ and the expected waiting time of passengers in income class $k$ at station $s \in \mathcal{V}_t$, respectively. The optimal hyperpath assignment program for income class $k$ can be formulated as the following optimization \cite{spiess1989optimal}:
\begin{subequations} \label{optimal_hyperpath}
\begin{align}
\max_{\{{v_{ak}\}}_{a\in \mathcal{E}_t},\{{w_{sk}^p\}}_{s \in \mathcal{V}_t}} \quad & \frac{\beta_k}{v_p} \sum_{a\in \mathcal{E}_t} l_a^p v_{ak} + \gamma_k r^p \sum_{a\in \mathcal{E}_t} l_a^p v_{ak} + \alpha_k \sum_{s \in \mathcal{V}_t} w_{sk}^p \label{objective_hyperpath} \\
\mathrm{s.t.} \quad
& \sum\limits_{a \in \mathcal{E}_t^{s+}} v_{ak} = \sum\limits_{a \in \mathcal{E}_t^{s-}} v_{ak} + g_s, \quad \forall s \in \mathcal{V}_t \label{flow_conservation_hyperpath} \\
&v_{ak} \leq f_a w_{sk}^p, \quad \forall a \in \mathcal{E}_t^{s+}, \forall s \in \mathcal{V}_t \label{flow_proportion_hyperpath} \\
&v_{ak} \geq 0, \quad \forall a \in \mathcal{E}_t \label{non_negativity_hyperpath} 
\end{align} 
\end{subequations}
The assignment program (\ref{optimal_hyperpath}) minimizes the expected generalized travel cost (\ref{objective_hyperpath}) from boarding station $s_i$ to alighting station $s_j$ for income class $k$ subject to the flow conservation constraints (\ref{flow_conservation_hyperpath}), flow proportion constraints (\ref{flow_proportion_hyperpath}), and the non-negativity constraints (\ref{non_negativity_hyperpath}). In the expected generalized cost (\ref{objective_hyperpath}), the first and second term account for the cost associated with in-vehicle time and trip fare on transit links respectively, while the last term represents the cost associated with waiting time at stations. The flow proportion constraints (\ref{flow_proportion_hyperpath}) are derived by multiplying the probability of choosing link $a \in \mathcal{E}_t^{s+}$ with the number of passengers waiting at station $s\in \mathcal{V}_t$. Note that given the transit fare $r^p$ and the service frequencies $\{f_l\}_{l \in \mathcal{L}}$, the objective function (\ref{objective_hyperpath}) and the involved constraints (\ref{flow_conservation_hyperpath})-(\ref{non_negativity_hyperpath}) are all linear. Consequently, the hyperpath assignment program (\ref{optimal_hyperpath}) yields a linear programming problem to which the globally optimal solution can be easily obtained. Let $\{v_{ak}^{*}\}_{a\in\mathcal{E}_t}$ and $\{w_{sk}^{p^*}\}_{s\in\mathcal{V}_t}$ be the optimal solution/hyperpath to (\ref{optimal_hyperpath}). The average waiting time $w_{s_i s_j,k}^p$ of income class $k$ from boarding station $s_i$ to alighting station $s_j$ can be characterized as:
\begin{equation}
    w_{s_i s_j,k}^p = \sum_{s\in \mathcal{V}_t} w_{sk}^{p^*} .
\end{equation}
The average trip distance $l_{s_is_j,k}^p$ of income class $k$ from boarding station $s_i$ to alighting station $s_j$ is given by:
\begin{equation}
    l_{s_i s_j,k}^p = \sum_{a \in \mathcal{E}_t} l_a^p v_{ak}^{*} .
\end{equation}

Passengers select boarding/alighting stations in the origin/destination zone by comparing the generalized costs of transit trips with distinct boarding/alighting stations. We use a logit model to characterize the probability of selecting boarding station $s_i \in \mathcal{V}_t^i$ and alighting station $s_j \in \mathcal{V}_t^j$ in direct transit services:
\begin{equation}
    \mathbb{P}_{s_i s_j,k}^p = \frac{\exp{\left(- \eta c_{s_i s_j,k}^p\right)}}{\sum\limits_{s_i' \in \mathcal{V}_t^i}\sum\limits_{s_j' \in \mathcal{V}_t^j}\exp{\left(-\eta c_{s_i' s_j',k}^p\right)}} ,
\end{equation}
where $\eta$ is the scaling parameter in the boarding/alighting station choice logit model. Overall, the expected generalized cost (also referred to as composite cost) for travelers in income class $k$ taking direct public transit from origin zone $i$ to destination zone $j$ given the set of boarding stations $\mathcal{V}_t^i$ and alighting stations $\mathcal{V}_t^j$ is given by the following logsum form \cite{de2007logsum}:
\begin{equation} \label{generalized_cost_p_zone}
    c_{ij,k}^p = -\frac{1}{\eta} \log \sum_{s_i \in \mathcal{V}_t^i} \sum_{s_j \in \mathcal{V}_t^j} \exp{\left(- \eta c_{s_i s_j,k}^p\right)} .
\end{equation}

\subsubsection{Generalized costs of bundled trips}
The travel cost of bundled services (indexed by $b$) comprises waiting time, in-vehicle time, trip fare, and the access and egress cost of walking (if any). Using AMoD for the first-mile or/and last-mile connection, bundled services can be further categorized into three possible scenarios: (1) the bundle of first-mile AMoD service and public transit (indexed by $b_1$); (2) the bundle of last-mile AMoD service and public transit (indexed by $b_2$); and (3) the bundle of both first-mile and last-mile AMoD service and public transit (indexed by $b_3$).
Analogously, bundled trips with distinct boarding/alighting stations incur different travel costs. When choosing bundled services, passengers select the boarding/alighting station in the origin/destination zone to minimize the generalized travel cost. We characterize the generalized costs of bundled trips with distinct boarding/alighting stations for distinct types of bundled services. Let $s_i \in \mathcal{V}_t^i$ be the boarding station in origin zone $i$ and $s_j \in \mathcal{V}_t^j$ be the alighting station in destination zone $j$, respectively. The generalized costs of distinct types of bundled trips with boarding station $s_i$ and $s_j$ are characterized as:
\begin{equation} \label{generalized_cost_b1_station}
    c_{s_i s_j,k}^{b_1} = \underbrace{\alpha_k w_i^a + \beta_k \frac{d_{s_i}}{v_a} + \gamma_k (b+r_i^a d_{s_i})}_\text{first leg AMoD} + \underbrace{\alpha_k w_{s_i s_j,k}^p  + \beta_k \frac{l_{s_i s_j,k}^p}{v_p} + \gamma_k r^p l_{s_i s_j,k}^p}_\text{public transit} + \underbrace{\theta_k \frac{d_{s_j}}{v_w}}_\text{walking} ,
\end{equation}

\begin{equation} \label{generalized_cost_b2_station}
    c_{s_i s_j,k}^{b_2} = \underbrace{\theta_k \frac{d_{s_i}}{v_w}}_\text{walking} +  \underbrace{\alpha_k w_{s_i s_j,k}^p  + \beta_k \frac{l_{s_i s_j,k}^p}{v_p} + \gamma_k r^p l_{s_i s_j,k}^p}_\text{public transit} + \underbrace{\alpha_k w_j^a + \beta_k \frac{d_{s_j}}{v_a} + \gamma_k (b+r_j^a d_{s_j})}_\text{last leg AMoD} ,
\end{equation}

\begin{equation} \label{generalized_cost_b3_station}
    c_{s_i s_j,k}^{b_3} = \underbrace{\alpha_k w_i^a + \beta_k \frac{d_{s_i}}{v_a} + \gamma_k (b+r_i^a d_{s_i})}_\text{first leg AMoD} +  \underbrace{\alpha_k w_{s_i s_j,k}^p  + \beta_k \frac{l_{s_i s_j,k}^p}{v_p} + \gamma_k r^p l_{s_i s_j,k}^p}_\text{public transit} + \underbrace{\alpha_k w_j^a + \beta_k \frac{d_{s_j}}{v_a} + \gamma_k (b+r_j^a d_{s_j})}_\text{last leg AMoD} .
\end{equation}
Equation (\ref{generalized_cost_b1_station})-(\ref{generalized_cost_b3_station}) indicate that the generalized travel costs of bundled services are determined by the sum of the costs of different trip segments (AMoD, public transit and walking) along the intermodal trip, where the generalized costs of AMoD trips and public transit trips are derived in Section \ref{sec_generalized_cost_AMoD} and \ref{sec_generalized_cost_transit}, separately.

Passengers determine the boarding/alighting station based on the generalized costs of bundled trips with distinct boarding/alighting stations\footnote{We assume that passengers choose boarding/exiting stations in the same zone. This already leaves passengers with a rich set of choices given the size of the zone is relatively large. We leave the modeling of inter-zone transfers as future work.}. Under the logit model, the probability of selecting boarding station $s_i \in \mathcal{V}_t^i$ and alighting station $s_j \in \mathcal{V}_t^j$ in distinct types of bundled services are
\begin{equation} \label{boarding_alighting_logit}
    \mathbb{P}_{s_i s_j,k}^t = \frac{\exp{\left(-\eta c_{s_i s_j,k}^t\right)}}{\sum\limits_{s_i' \in \mathcal{V}_t^i}\sum\limits_{s_j' \in \mathcal{V}_t^j}\exp{\left(-\eta c_{s_i' s_j',k}^t\right)}} , \quad t \in \{b_1, b_2, b_3\} .
\end{equation}
Finally, the expected generalized cost for travelers in income class $k$ choosing distinct bundled services from origin zone $i$ to destination zone $j$ are given by:
\begin{equation} \label{generalized_cost_b_zone}
    c_{ij,k}^t = -\frac{1}{\eta} \log \sum_{s_i \in \mathcal{V}_t^i} \sum_{s_j \in \mathcal{V}_t^j} \exp{\left(-\eta c_{s_i s_j,k}^t\right)}, \quad t \in \{b_1, b_2, b_3\} .
\end{equation}

\begin{remark}
This paper neglects the traffic congestion and transit crowding in the multimodal transport network. For simplicity, we do not model traffic speed as an endogenous variable that depends on the number of TNC vehicles. We believe this is a reasonable simplification and would not change the essential features of TNC services. We also neglect the crowding effect in the transit network since transit service typically has a large capacity. Overall, we believe that the impacts of traffic congestion and transit crowding are second-order factors that would not influence the major insights of this research. We leave the modeling of traffic congestion and transit crowding as future work.
\end{remark}

\subsubsection{Demand of passengers}

In the multimodal transportation system, passengers in distinct income classes choose among direct AMoD services (mode $a$), public transit (mode $p$), the bundle of first leg AMoD with transit (mode $b_1$), the bundle of last leg AMoD with transit (mode $b_2$), the bundle of first and last leg AMoD with transit (mode $b_3$), and outside option (mode $o$) to reach their destinations at the minimum cost. Let $\mathcal{T}=\{a,p,b_1,b_2,b_3,o\}$ be the set of possible travel modes in the multimodal transport network. The disutility/generalized travel costs of different mobility modes collectively determine passenger demand. We use a multinomial logit model to characterize the passenger demand over distinct travel modes:
\begin{equation} \label{logit_demand_function}
    \lambda_{ij,k}^t = \lambda_{ij,k}^0 \frac{\exp \left(-\mu c_{ij,k}^t\right)}{\sum_{{t'} \in \mathcal{T}} \exp \left(-\mu c_{ij,k}^{t'}\right)}, \quad t \in \mathcal{T},
\end{equation}
where $\lambda_{ij,k}^t$ represents the arrival rate of passengers of mode $t$ from origin zone $i$ to destination zone $j$ in income class $k$, $\lambda_{ij,k}^0$ is the arrival rate of potential passengers from zone $i$ to zone $j$ in income class $k$, $c_{ij,k}^t$ is the expected disutility/generalized travel costs for passengers from origin $i$ to destination $j$ in income class $k$ by choosing mode $t$, $\mu$ is the scaling parameter in the mode choice logit model.

\begin{remark}
    We apply an equal scaling parameter $\eta$ and $\mu$ for distinct income classes in the boarding/alighting station choice logit model (\ref{boarding_alighting_logit}) and the modal choice logit model (\ref{logit_demand_function}), respectively. In the logit model, the scaling parameter $\mu$ relates to the variance (denoted as $\sigma^2$) of Gumbel distributed random disturbance through $\mu=\frac{\pi}{\sqrt{6}\sigma}$ \cite{ben1985discrete}. Since the randomness is multiplied by different parameters in the generalized costs of distinct income classes, i.e, (\ref{generalized_cost_b1_station})-(\ref{generalized_cost_b3_station}), the value of $\mu$ may differ for distinct income classes. However, for sake of simplicity, we assume an equal scaling parameter $\eta$ and $\mu$ for distinct income classes, which is consistent with existing works \cite{ahmed2020quantifying}, \cite{dixit2020capturing}. The case of heterogeneous scaling parameters for distinct classes is left for future work.
\end{remark}

\subsection{Incentives of the TNC Platform}

Consider a TNC platform that hires a fleet of $N$ autonomous vehicles to provide mobility-on-demand services. Each TNC vehicle is in one of the following statuses: (a) cruising on the street and waiting for the passenger; (b) on the way to pickup a passenger; and (c) carrying a passenger. For passengers from origin zone $i$ to destination zone $j$ in income class $k$, let $d_{i,k}^{b_1}$ be the average access distance (first-mile distance) by AMoD in origin zone $i$ when choosing mode $b_1$, $d_{j,k}^{b_2}$ be the average egress distance (last-mile distance) by AMoD in destination zone $j$ when choosing mode $b_2$, and $d_{i,k}^{b_3}$ and $d_{j,k}^{b_3}$ be the average access and egress distance by AMoD in zone $i$ and zone $j$ when choosing mode $b_3$, respectively. They are determined by:
    \begin{subnumcases}{\label{distance_AMoD_b}}
        d_{i,k}^{b_1}= \sum_{s_i \in \mathcal{V}_t^i} \sum_{s_j \in \mathcal{V}_t^j} \mathbb{P}_{s_i s_j,k}^{b_1} d_{s_i} \\
        d_{j,k}^{b_2}= \sum_{s_i \in \mathcal{V}_t^i} \sum_{s_j \in \mathcal{V}_t^j} \mathbb{P}_{s_i s_j,k}^{b_2} d_{s_j} \\
        d_{i,k}^{b_3}= \sum_{s_i \in \mathcal{V}_t^i} \sum_{s_j \in \mathcal{V}_t^j} \mathbb{P}_{s_i s_j,k}^{b_3} d_{s_i} \\
        d_{j,k}^{b_3}= \sum_{s_i \in \mathcal{V}_t^i} \sum_{s_j \in \mathcal{V}_t^j} \mathbb{P}_{s_i s_j,k}^{b_3} d_{s_j}
    \end{subnumcases}
At the stationary state, the total number of vehicle hours $N$ should satisfy the following conservation law:
\begin{equation} \label{vehicle_hour_conservation}
\begin{split}
     N = & \sum_{i=1}^M N_i^I + \sum_{i=1}^M \sum_{j=1}^M \sum_{k=1}^K \left( \lambda_{ij,k}^a w_i^a + \lambda_{ij,k}^{b_1} w_i^a + \lambda_{ij,k}^{b_2} w_j^a + \lambda_{ij,k}^{b_3} \left(w_i^a + w_j^a\right)\right) + \\ &\sum_{i=1}^M \sum_{j=1}^M \sum_{k=1}^K \left(\lambda_{ij,k}^a \frac{l_{ij}^a}{v_a} + \lambda_{ij,k}^{b_1} \frac{d_{i,k}^{b_1}}{v_a} + \lambda_{j,k}^{b_2} \frac{d_{ij,k}^{b_2}}{v_a} + \lambda_{ij,k}^{b_3} \left(\frac{d_{i,k}^{b_3}}{v_a}+\frac{d_{j,k}^{b_3}}{v_a}\right)\right) .
\end{split}
\end{equation}
Based on Little's law, the first term in the right-hand side of (\ref{vehicle_hour_conservation}) represents the operating hours of idle vehicles cruising on the street; the second term accounts for the operating hours of vehicles that are on the way to pick up passengers; and the third term sums the operating hours of vehicles that are occupied with a passenger. The total operating hours of picking up vehicles and occupied vehicles can be further decomposed into operating hours of vehicles that serve four different types of AMoD services (e.g., direct service, first-mile service, last-mile service, and first-mile and last-mile service), respectively.

The AMoD platform determines the base fare $b$, the per-distance rates $r_i^a$, the spatial distribution of idle AVs $N_i^I$, and the fleet size $N$ to maximize its profit subject to the equilibrium conditions. The profit maximization for the TNC platform can be formulated as follows:\footnote{We acknowledge that since our primary goal is to characterize competition between distinct modes, for simplicity, the repositioning of idle TNC vehicles is not explicitly considered in this model. We leave it as future work. }
\begin{equation} \label{Incentives_AMoD}
    \begin{split}
    \max_{b, \mathbf{r^{a}},\mathbf{N^I}, N} \quad & \sum_{i=1}^M \sum_{j=1}^M \sum_{k=1}^K \lambda_{ij,k}^a \left(b+r_i^a l_{ij}^a\right) + \lambda_{ij,k}^{b_1}\left(b+r_i^a d_{i,k}^{b_1}\right) + \lambda_{ij,k}^{b_2} \left(b+r_j^a d_{j,k}^{b_2}\right) + \lambda_{ij,k}^{b_3} \left(2b + r_i^a d_{i,k}^{b_3} + r_j^a d_{j,k}^{b_3}\right) - N C_{av} \\
    \mathrm{s.t.} \quad & (\ref{generalized_cost_a})\text{-}(\ref{logit_demand_function}), (\ref{distance_AMoD_b})\text{-}(\ref{vehicle_hour_conservation}) \\
    \end{split}
\end{equation}
where $C_{av}$ is the hourly operating cost of an AV\footnote{The cost of operating an for-hire AV consists of the capital cost and the mileage-dependent operating cost. We convert them to an hourly basis and denote $C_{av}$ as the sum of the hourly capital cost and operating cost of an AV.}. The objective in (\ref{Incentives_AMoD}) defines the platform profit as the difference between the revenue and the total operating cost of AVs. The total revenue is the ride fares collected from passengers choosing four distinct types of AMoD services. The platform decisions are subject to the passenger demand model (\ref{generalized_cost_a})-(\ref{logit_demand_function}) and vehicle conservation (\ref{distance_AMoD_b})-(\ref{vehicle_hour_conservation}).

\subsection{Incentives of the Public Transit Agency}

Consider a public transit agency that operates the transit network with $L$ transit lines. Generally, based on the level of demand and the characteristics of the service area, transit lines are clustered into subsets with the same frequency to simplify the operation and coordination of transit services. High-demand lines usually have more frequent service, while others with lower demand may operate with less frequency. Therefore, we divide the transit lines into two subsets with high frequency and low frequency. Let $\mathcal{L}_{H} \subset \mathcal{L}$ and $\mathcal{L}_L \subset \mathcal{L}$ be the set of transit lines with high frequency and low frequency, respectively. The transit lines within the same subset operate with the same frequency, which imposes the following constraints\footnote{The consistency constraint on line frequencies reduces the dimension of the decision space, which enables us to apply grid search method to obtain the globally optimal solution to the ridership maximization of the transity agency. More levels of frequencies can be introduced, but one has to balance model accuracy with computational burden.}:
\begin{equation} \label{frequency_consistency}
\begin{cases}
    f_l = f_{l'}, \quad \forall l,l' \in \mathcal{L}_H \\
    f_l = f_{l'}, \quad \forall l,l' \in \mathcal{L}_L
\end{cases} .
\end{equation}
In the meanwhile, a budget constraint is imposed to guarantee the economic sustainability of the transit agency, which requires that the difference between the operating cost and revenue be smaller than the operational budget $\pi_0$. Let $l_{ij,k}^t, t \in \{p,b_1,b_2,b_3\}$ be the average trip distance by transit from zone $i$ to zone $j$ for income class $k$ when choosing transit-related mobility mode $t$, which satisfy:
\begin{equation} \label{distance_transit}
    l_{ij,k}^t = \sum_{s_i \in \mathcal{V}_t^i} \sum_{s_j \in \mathcal{V}_t^j} \mathbb{P}_{s_i s_j,k}^{t} l_{s_i s_j,k}^t, \quad t\in \{p,b_1,b_2,b_3\} .
\end{equation}
The budget constraint can be characterized as:
\begin{equation} \label{profit_constraint}
     \sum_{l=1}^L f_l C_l - r^p \sum_{i=1}^M \sum_{j=1}^M \sum_{k=1}^K \lambda_{ij,k}^p l_{ij,k}^p + \lambda_{ij,k}^{b_1} l_{ij,k}^{b_1} + \lambda_{ij,k}^{b_2} l_{ij,k}^{b_2} + \lambda_{ij,k}^{b_3} l_{ij,k}^{b_3} \leq \pi_0 ,
\end{equation}
where $C_l$ is the per-vehicle hourly operating cost of transit line $l$.

The public transit agency determines the transit fare $r^p$, and the service frequencies $f_l$ to maximize the public transit ridership subject to the passenger demand model (\ref{generalized_cost_a})-(\ref{logit_demand_function}), the frequency consistency (\ref{frequency_consistency}), and the budget constraint (\ref{distance_transit})-(\ref{profit_constraint}). The ridership maximization for public transit can be cast as:
\begin{equation} \label{Incentives_PT}
    \begin{aligned}
    \max_{r^p,\mathbf{f}} \quad & \sum_{i=1}^M \sum_{j=1}^M \sum_{k=1}^K \lambda_{ij,k}^p + \lambda_{ij,k}^{b_1} + \lambda_{ij,k}^{b_2} + \lambda_{ij,k}^{b_3} \\
    \mathrm{s.t.} \quad & (\ref{generalized_cost_a})\text{-}(\ref{logit_demand_function}), (\ref{frequency_consistency}), (\ref{distance_transit})\text{-}(\ref{profit_constraint})
    \end{aligned}
\end{equation}

\begin{remark}
We postulate that the public transit agency aims to maximize the total ridership. This is consistent with the nature of public transport agencies (e.g. the government) aiming to provide reliable and sustainable mobility services to minimize negative externalities of vehicle travel (e.g. emissions and congestion). Although a large number of existing studies model the objective of transit agencies as (1) user benefit maximization (or equivalently, user cost minimization) \cite{hasselstrom1982public,ouyang2014continuum}, (2) total welfare maximization \cite{chien2002optimization,zhu2021competition}, (3) waiting time minimization \cite{chakroborty2003genetic,verbas2015integrated}, and (4) total cost minimization \cite{fan2008tabu,daganzo2010structure}, etc, we argue that ridership maximization \cite{yoo2010frequency,verbas2013optimal} is a good proxy for the maximization of social welfare or user benefit under a fixed budget. Maximizing total social welfare or passengers' surplus would complicate the formulation without providing extra insights.
\end{remark}

\subsection{Equilibrium of the Game}

The profit maximization for the TNC platform (\ref{Incentives_AMoD}) and the ridership maximization for the public transit (\ref{Incentives_PT}) constitute the game problem. Note that although the operational decision of the TNC platform $(b,\mathbf{r^a},\mathbf{N^I})$ and the public transit agency $(r^p,\mathbf{f})$ does not explicitly appear in the objective of the other, they interact with each other by affecting passengers' mode choices in the passenger demand model (\ref{generalized_cost_a})-(\ref{logit_demand_function}). Therefore, the platform profit of the TNC and the ridership of public transit not only depend on each player's own operational decisions but also on the operational strategy of the counterpart. We assume that the TNC platform and the public transit agency interact simultaneously and they have perfect information on each other's strategy. The game is at equilibrium if the TNC platform cannot increase its profit and the public transit agency cannot increase the transit ridership by unilaterally changing their operational strategies. Let $\xi^{a}=(b,\mathbf{r^a},\mathbf{N^I})$ and $\xi^{p}=(r^p,\mathbf{f})$ be the operational strategy of the TNC platform and the public transit agency, respectively, and denote $\pi^a(\xi^a,\xi^p)$ and $\pi^p(\xi^a,\xi^p)$ as the TNC platform profit and the transit ridership, respectively. We can formally define the equilibria $\xi^*=\left(\xi^{a^*},\xi^{p^*}\right)$ that satisfies:
\begin{equation} \label{equilibrium_condition}
    \begin{cases}
    \pi^a(\xi^{a^*},\xi^{p^*}) \geq \pi^a(\xi^a,\xi^{p^*}) \\
    \pi^p(\xi^{a^*},\xi^{p^*}) \geq \pi^p(\xi^{a^*},\xi^p)
    \end{cases}, \quad \forall \xi^a, \xi^p \geq 0 .
\end{equation}
which is consistent with the general definition of Nash equilibrium.

In practice, each player in the game may be indifferent to a small change in its objective function (e.g., the change only accounts for a negligible portion of the objective value). This motivates a more relaxed solution concept that encapsulates a broader range of equilibrium solutions, where players will not unilaterally change their strategies as far as the current solution is approximately optimal. More rigorously, we can define such relaxed solution as the $\epsilon$-Nash equilibrium  \cite{daskalakis2006note, li2019connections}, denoted as $\xi^*=\left(\xi^{a^*},\xi^{p^*}\right)$, which satisfies:
\begin{equation} \label{epsilon_equilibrium_condition}
    \begin{cases}
    \pi^a(\xi^{a^*},\xi^{p^*}) \geq \pi^a(\xi^a,\xi^{p^*})-\epsilon \\
    \pi^p(\xi^{a^*},\xi^{p^*}) \geq \pi^p(\xi^{a^*},\xi^p)-\epsilon
    \end{cases}, \quad \forall \xi^a, \xi^p \geq 0 .
\end{equation}
At an $\epsilon$-Nash equilibrium, each player can improve its objective value by at most $\epsilon$ via deviating from the equilibrium strategy $\xi^*=\left(\xi^{a^*},\xi^{p^*}\right)$, given that the other player follows the equilibrium strategy. Therefore, it indicates that the players are motivated to play the equilibrium strategy if they are indifferent to a change of $\epsilon$. Note that Nash equilibrium can be viewed as the special case of the $\epsilon$-Nash equilibrium, correpsonding to the case of $\epsilon=0$. 

\section{The Solution Method} \label{solution method}

This section solves the game problem by computing the Nash equilibrium of the game using the best response method and conducting an ex-post evaluation of the performance of the obtained solution. In game theory, best response is the strategy that maximizes the player's payoff given other players' strategies, and the Nash equilibrium, by definition, is the stable state at which each player in the game has selected the best response to the other player's strategies and no player has the incentive to deviate from the current strategy. The best response methods repeatedly update each player's strategy based on the perceived best response to other players' strategies until the stable state is reached. When applying best response methods to our game problem, we iteratively solve the profit maximization problem (\ref{Incentives_AMoD}) of the TNC platform using standard interior-point methods and the ridership maximization problem (\ref{Incentives_PT}) of the transit agency by grid search methods until convergence (see line 2-9 in Algorithm \ref{algorithm1}). We comment that this step is not difficult because the interior point methods are well-established approaches, which are effective in solving the constrained nonlinear optimization problem (\ref{Incentives_AMoD}). Besides, grid search is feasible and guarantees to find the globally optimal solution to the small-scale optimization (\ref{Incentives_PT}) which involves a three-dimensional search space. However, the derived solution is only a candidate Nash equilibrium since the profit maximization problem (\ref{Incentives_AMoD}) is non-convex due to the complex nonlinearity in the demand model (\ref{generalized_cost_a})-(\ref{logit_demand_function}), which indicates that the solution obtained by interior point methods at each iteration is only locally optimal. To validate how good the derived solution is compared to the globally optimal solution to each decision maker (i.e., the ground truth Nash equilibrium), we focus on the profit maximization problem (\ref{Incentives_AMoD}) of the TNC platform, and relax the consistency of its decisions to derive a tight upper bound for characterizing the optimality gap of the derived solution. The proposed method enables us to to conduct an ex-post evaluation of the performance of the obtained solution by showing that the derived solution is at least as good as an $\epsilon$-Nash equilibrium, where the value of $\epsilon$ can be computed numerically. Below we investigate the profit maximization problem for the TNC platform to numerically derive $\epsilon$.

In the profit maximization problem (\ref{Incentives_AMoD}), the involved constraints are all equality constraints. Note that given $\xi^p$, (\ref{generalized_cost_a})-(\ref{logit_demand_function}) jointly determine $\lambda_{ij,k}^t,t\in\left\{a,b_1,b_2,b_3\right\}$ as a function of $b$, $r_i^a$, $r_j^a$, $N_i^I$, and $N_j^I$. With a slight abuse of notation, we denote it as $\lambda_{ij,k}^t=\lambda_{ij,k}^t\left(b,r_i^a,r_j^a,N_i^I,N_j^I\right), t\in\{a,b_1,b_2,b_3\}$. By further substituting (\ref{vehicle_hour_conservation}) into the objective function, the original problem  can be transformed into the following unconstrained optimization:
\begin{equation} \label{Incentives_AMoD_unconstrained}
    \begin{split}
    \max_{b, \mathbf{r^{a}},\mathbf{N^I}} \quad & \sum_{i=1}^M \sum_{j=1}^M \sum_{k=1}^K  \lambda_{ij,k}^a\left(b,r_i^a,r_j^a,N_i^I,N_j^I\right) \left[b+\left(r_i^a-\frac{C_{av}}{v_a}\right) l_{ij}^a - C_{av} \cdot w_i^a\left(N_i^I\right)\right] \\ + & \lambda_{ij,k}^{b_1}\left(b,r_i^a,r_j^a,N_i^I,N_j^I\right)\left[b+\left(r_i^a-\frac{C_{av}}{v_a}\right) d_{i,k}^{b_1} - C_{av} \cdot w_i^a\left(N_i^I\right)\right] \\ + & \lambda_{ij,k}^{b_2}\left(b,r_i^a,r_j^a,N_i^I,N_j^I\right) \left[b+\left(r_j^a-\frac{C_{av}}{v_a}\right) d_{j,k}^{b_2} - C_{av} \cdot w_j^a\left(N_j^I\right)\right] \\ + & \lambda_{ij,k}^{b_3}\left(b,r_i^a,r_j^a,N_i^I,N_j^I\right) \left[2b + \left(r_i^a-\frac{C_{av}}{v_a}\right) d_{i,k}^{b_3} + \left(r_j^a-\frac{C_{av}}{v_a}\right) d_{j,k}^{b_3} - C_{av} \cdot w_i^a\left(N_i^I\right) - C_{av} \cdot w_j^a\left(N_j^I\right)\right] \\ - & \frac{C_{av}}{M} \sum_{i=1}^M \sum_{j=1}^M N_i^I ,
    \end{split}
\end{equation}
where we equivalently rewrite $C_{av}\sum_{i=1}^M N_i^I$ as $\frac{C_{av}}{M}\sum_{i=1}^M \sum_{j=1}^M N_i^I$.
Based on (\ref{Incentives_AMoD_unconstrained}), we introduce auxiliary variables $b_j$, $r_{ij}^a$, and $N_{ij}^I$, where $i,j=1,\dots,M$ to formulate the equivalent problem. The intuition behind the auxiliary variables is that the platform determines the base fare $b_j$, the ride fares $r_{ij}^a,i=1,\dots,M$, and the number of idle vehicles $N_{ij}^I,i=1,\dots,M$ for the specific destination zone $j$. Let $\xi_j^a=\left(b_j,r_{1j}^a,N_{1j}^I,r_{2j}^a,N_{2j}^I,\dots,r_{Mj}^a,N_{Mj}^I\right)$ be the operational strategy exclusively for destination zone $j$, and define $\pi_j^a$ as the corresponding profit from destination zone $j$:
\begin{equation}
    \begin{split}
        & \pi_j^a\left(\xi_j^a\right) = \sum_{i=1}^M \sum_{k=1}^K \lambda_{ij,k}^a\left(b_j,r_{ij}^a,r_{jj}^a,N_{ij}^I,N_{jj}^I\right) \left[b_j+\left(r_{ij}^a-\frac{C_{av}}{v_a}\right) l_{ij}^a - C_{av} \cdot w_i^a\left(N_{ij}^I\right)\right] \\ + & \lambda_{ij,k}^{b_1}\left(b_j,r_{ij}^a,r_{jj}^a,N_{ij}^I,N_{jj}^I\right)\left[b_j+\left(r_{ij}^a-\frac{C_{av}}{v_a}\right) d_{i,k}^{b_1} - C_{av} \cdot w_i^a\left(N_{ij}^I\right)\right] \\ + &\lambda_{ij,k}^{b_2}\left(b_j,r_{ij}^a,r_{jj}^a,N_{ij}^I,N_{jj}^I\right) \left[b_j+\left(r_{jj}^a-\frac{C_{av}}{v_a}\right) d_{j,k}^{b_2} - C_{av} \cdot w_j^a\left(N_{jj}^I\right)\right] \\ + &\lambda_{ij,k}^{b_3}\left(b_j,r_{ij}^a,r_{jj}^a,N_{ij}^I,N_{jj}^I\right) \left[2b_j + \left(r_{ij}^a-\frac{C_{av}}{v_a}\right) d_{i,k}^{b_3} + \left(r_{jj}^a-\frac{C_{av}}{v_a}\right) d_{j,k}^{b_3} - C_{av} \cdot w_i^a\left(N_{ij}^I\right) - C_{av} \cdot w_j^a\left(N_{jj}^I\right)\right] \\ - & \frac{C_{av}}{M} \sum_{i=1}^M N_{ij}^I, \quad j=1,\dots,M.
    \end{split}
\end{equation}
Note that the platform actually deploys a uniform operational strategy for distinct destination zones, therefore the destination-specific operational strategy $\xi_j^a$ should be consistent across different destination zones. In this case, the original problem (\ref{Incentives_AMoD_unconstrained}) is equivalent to the following optimization:
\begin{equation} \label{Incentives_AMoD_equivalent}
    \begin{split}
        \max_{\xi_1^a,\xi_2^a,\dots,\xi_M^a} \quad & \sum_{j=1}^M \pi_j^a(\xi_j^a) \\
        \mathrm{s.t.} \quad & \xi_1^a = \xi_2^a = \dots =\xi_M^a
    \end{split}
\end{equation}

Based on the equivalent formulation (\ref{Incentives_AMoD_equivalent}), we will slightly modify the decision consistency constraint $\xi_1^a=\xi_2^a=\dots=\xi_M^a$ to derive an upper bound. In particular, let $\overline{\mathcal{V}}=\left\{\mathcal{V}_1,\mathcal{V}_2,\dots,\mathcal{V}_m\right\}$ be a partition of $\mathcal{V}$, i.e., $\bigcup_{x=1,\dots,m} \mathcal{V}_x = \mathcal{V}$ and $\mathcal{V}_x \cap \mathcal{V}_y = \emptyset, \forall \mathcal{V}_x,\mathcal{V}_y\in \overline{\mathcal{V}}, x \neq y$, and define the relaxed problem as
\begin{equation} \label{Incentives_AMoD_relaxed}
    \begin{split}
        \max_{\xi_1^a,\xi_2^a,\dots,\xi_M^a} \quad & \sum_{j=1}^M \pi_j^a(\xi_j^a) \\
        \mathrm{s.t.} \quad & \xi_j^a = \xi_{j'}^a \quad \forall j,j' \in \mathcal{V}_1 \\
        & \xi_j^a = \xi_{j'}^a \quad \forall j,j' \in \mathcal{V}_2 \\
        & \dots \\
        & \xi_j^a = \xi_{j'}^a \quad \forall j,j' \in \mathcal{V}_m
    \end{split} ,
\end{equation}
where we break the complete decision consistency and only require that the operational strategies are consistent across a subset of destination zones. 
Note that (\ref{Incentives_AMoD_relaxed}) can be further decomposed over the partition $\overline{\mathcal{V}}$:
\begin{equation} \label{Incentives_AMoD_relaxed_subproblem}
    \begin{split}
        \max_{\xi_j^a,j \in \mathcal{V}_x} \quad & \sum_{j \in \mathcal{V}_x} \pi_j^a(\xi_j^a) \\
        \mathrm{s.t.} \quad & \xi_j^a = \xi_{j'}^a \quad \forall j,j' \in \mathcal{V}_x
    \end{split} ,
    \quad x=1,\dots,m .
\end{equation}
The optimal solution to (\ref{Incentives_AMoD_relaxed_subproblem}), if obtained, provides an upper bound for the optimal value of the original problem, thus we have the following result:
\begin{proposition} \label{proposition_upper_lower_bound}
Suppose that $\xi^*=\left(\xi^{a^*},\xi^{p^*}\right)$  is a candidate Nash equilibrium. Let $\xi_j^{a^*},j\in\mathcal{V}_x$ be the optimal solution to (\ref{Incentives_AMoD_relaxed_subproblem}) given $\xi^{p^*}$, and let $\pi_{\mathcal{V}_x}^{a^*}$ be the corresponding optimal value. We have $\sum_{x=1}^m \pi_{\mathcal{V}_x}^{a^*} \geq \pi^a\left(\xi^{a^*},\xi^{p^*}\right) \geq \pi^a \left(\frac{\sum_{x=1}^m \pi_{\mathcal{V}_x}^{a^*} \xi_{j,j \in \mathcal{V}_x}^{a^*}}{\sum_{x=1}^m \pi_{\mathcal{V}_x}^{a^*}}, \xi^{p^*} \right)$ .
\end{proposition}
The proof of Proposition \ref{proposition_upper_lower_bound} can be found in Appendix B. It establishes an upper bound and a lower bound for the platform profit at Nash equilibrium. The intuition behind the upper bound is that the relaxation of the consistency of platform decisions provides the platform a higher degree of freedom to manage the AMoD service and thus leads to higher profitability. In the meanwhile, we use a weighted sum of optimal solutions to (\ref{Incentives_AMoD_relaxed_subproblem}) to generate a feasible solution and a tight lower bound of the original problem (\ref{Incentives_AMoD_equivalent}). These weights are determined based on the intuition that the platform combines distinct operational strategies $\xi_{j,j\in\mathcal{V}_x}^{a^*}$ for distinct sets of destination zones $\mathcal{V}_x$ to generate a uniform strategy, and the operational strategy on the set of destination zones with a higher potential profit $\pi_{\mathcal{V}_x}^{a^*}$ is assigned with larger weight.

The upper bound derived by Proposition \ref{proposition_upper_lower_bound} is very useful for us because whenever a local solution to (\ref{Incentives_AMoD_equivalent}) is given, we can compare the optimal value with respect to the upper bound, which offers an upper bound on the distance between the optimal value of the local solution and the optimal value of the unknown globally optimal solution. This enables us to assert how good a solution is compared to the globally optimal one, and it will be an important intermediate step for assessing the overall quality of the Nash equilibrium. However, to use this result, we do need to numerically compute the globally optimal solution to the nonconvex program (\ref{Incentives_AMoD_relaxed_subproblem}). Fortunately, despite its non-convexity, (\ref{Incentives_AMoD_relaxed_subproblem}) has a decomposable structure, which can be addressed by primal decomposition. To this end, let $\xi_{ij}^a=\left(r_{ij}^a,N_{ij}^I\right)$ be the operational decision on the specific OD pair $ij$ and define $\pi_{ij}^a$ as the corresponding profit from OD pair $ij$:
\begin{equation}
    \begin{split}
        & \pi_{ij}^a\left(b_j,\xi_{jj}^a,\xi_{ij}^a\right) = \sum_{k=1}^K \lambda_{ij,k}^a\left(b_j,r_{ij}^a,r_{jj}^a,N_{ij}^I,N_{jj}^I\right) \left[b_j+\left(r_{ij}^a-\frac{C_{av}}{v_a}\right) l_{ij}^a - C_{av} \cdot w_i^a\left(N_{ij}^I\right)\right] \\ + & \lambda_{ij,k}^{b_1}\left(b_j,r_{ij}^a,r_{jj}^a,N_{ij}^I,N_{jj}^I\right)\left[b_j+\left(r_{ij}^a-\frac{C_{av}}{v_a}\right) d_{i,k}^{b_1} - C_{av} \cdot w_i^a\left(N_{ij}^I\right)\right] \\ + &\lambda_{ij,k}^{b_2}\left(b_j,r_{ij}^a,r_{jj}^a,N_{ij}^I,N_{jj}^I\right) \left[b_j+\left(r_{jj}^a-\frac{C_{av}}{v_a}\right) d_{j,k}^{b_2} - C_{av} \cdot w_j^a\left(N_{jj}^I\right)\right] \\ + &\lambda_{ij,k}^{b_3}\left(b_j,r_{ij}^a,r_{jj}^a,N_{ij}^I,N_{jj}^I\right) \left[2b_j + \left(r_{ij}^a-\frac{C_{av}}{v_a}\right) d_{i,k}^{b_3} + \left(r_{jj}^a-\frac{C_{av}}{v_a}\right) d_{j,k}^{b_3} - C_{av} \cdot w_i^a\left(N_{ij}^I\right) - C_{av} \cdot w_j^a\left(N_{jj}^I\right)\right] \\ - & \frac{C_{av}}{M} N_{ij}^I, \quad i,j=1,\dots,M.
    \end{split}
\end{equation}
Note that $\xi_j^a = \left(b_j,\xi_{1j}^a,\xi_{2j}^a,\dots,\xi_{Mj}^a\right)$ and $\pi_{j}^a\left(\xi_j^a\right)=\sum_{i=1}^M \pi_{ij}^a\left(b_j,\xi_{jj}^a,\xi_{ij}^a\right)$. Therefore, the optimization (\ref{Incentives_AMoD_relaxed_subproblem}) is equivalent to:
\begin{equation} \label{Incentives_AMoD_relaxed_equivalent}
    \begin{split}
        \max_{b_j,\xi_{ij}^a,j \in \mathcal{V}_x,i=1,\dots,M} \quad & \sum_{j \in \mathcal{V}_x} \sum_{i=1}^M \pi_{ij}^a (b_j,\xi_{jj}^a,\xi_{ij}^a) \\
        \mathrm{s.t.} \quad & b_j = b_{j'} \quad \forall j,j' \in \mathcal{V}_x \\
        & \xi_{ij}^a = \xi_{ij'}^a \quad \forall j,j' \in \mathcal{V}_x, i=1,\dots,M
    \end{split} ,
    \quad x = 1,\dots,m .
\end{equation}
For given $b_j,\xi_{jj}^a,j\in \mathcal{V}_x$, we define the subproblems:
\begin{equation} \label{Incentives_AMoD_primal_subproblem}
    \begin{split}
        \max_{\xi_{ij}^a, j \in \mathcal{V}_x} \quad & \sum_{j \in \mathcal{V}_x} \pi_{ij}^a(b_j,\xi_{jj}^a,\xi_{ij}^a) \\
        \mathrm{s.t.} \quad & \xi_{ij}^a = \xi_{ij'}^a \quad \forall j,j' \in \mathcal{V}_x
    \end{split} ,
    \quad i = 1,\dots,M .
\end{equation}
with optimal values $\pi_{i,\mathcal{V}_x}^{a^*},i=1,\dots,M$. In this case, $\pi_{i,\mathcal{V}_x}^{a^*},i=1,\dots,M$ is a function of $b_j,\xi_{jj}^a,j\in \mathcal{V}_x$, and (\ref{Incentives_AMoD_relaxed_equivalent}) is equivalent to the master problem:
\begin{equation} \label{Incentives_AMoD_primal_master}
\begin{split}
    \max_{b_j,\xi_{jj}^a,j\in\mathcal{V}_x} \quad & \sum_{i=1}^M \pi_{i,\mathcal{V}_x}^{a^*} \\
\mathrm{s.t.} \quad & b_j = b_{j'} \quad \forall j,j' \in \mathcal{V}_x
\end{split} .
\end{equation}
Note that (\ref{Incentives_AMoD_primal_subproblem}) only involves two decision variables $\xi_{ij}^a=\left(r_{ij}^a,N_{ij}^I\right)$ if we absorb the equality constraint into the objective. Similarly, (\ref{Incentives_AMoD_primal_master}) has $1+2\cdot |\mathcal{V}_x|$ decision variables if we absorb the equality constraint, where $|\mathcal{V}_x|$ is the number of zones in $\mathcal{V}_x$. Since both (\ref{Incentives_AMoD_primal_subproblem}) and (\ref{Incentives_AMoD_primal_master}) are small-scale problems, we can use grid search method to find the globally optimal solutions in parallel.
\begin{remark}
The partition $\bar{\mathcal{V}}$ is flexible and we can carefully choose the composition of $\bar{\mathcal{V}}$ to address the trade-off between computational complexity and the quality of the upper bound. In particular, if we partition the space $\mathcal{V}$ into a large number of smaller sets, then the decomposed problem (\ref{Incentives_AMoD_primal_master}) has a smaller dimension, which can be more efficiently solved. However, this induces more relaxation for the consistency of decisions, which leads to a worse upper bound. Our numerical simulation indicates that dividing all the zones $\mathcal{V}$ into pairwise groups (i.e., $|\mathcal{V}_x|=2$) will lead to  fast computation and  tight upper bound at the same time.
\end{remark}

To summarize, we run the best response algorithm to iteratively compute the profit maximization problem (\ref{Incentives_AMoD}) using the standard interior-point method \cite{mehrotra1992implementation} and solve the ridership maximization problem (\ref{Incentives_PT}) by grid search methods (line 2-9 in Algorithm \ref{algorithm1}). After the algorithm converges, we obtain a candidate Nash equilibrium, which might be the local solution to the profit maximization problem of the TNC platform. To validate the quality of the candidate Nash equilibrium, we use the result of Proposition \ref{proposition_upper_lower_bound}, so that the profit maximization problem (\ref{Incentives_AMoD}) can be solved approximately with a theoretical upper bound. Details of the algorithm are summarized in Algorithm \ref{algorithm1} in Appendix C.

The aforementioned numerical framework is very useful for us to evaluate the quality of the derived Nash equilibrium. This is because it indicates that the candidate Nash equilibrium is at least as good as an $\epsilon$-Nash equilibrium, where the value of $\epsilon$ can be given by the difference between the profit of the ride-hailing platform at the Nash equilibrium $\pi^a(\xi^{a^*},\xi^{p^*})$ and its upper bound $\bar{\pi}^a$. In particular, we summarize the above discussion as the following proposition:
\begin{proposition} \label{proposition_epsilon_nash}
Suppose the best response process (line 2-9) in Algorithm \ref{algorithm1} converges to a candidate Nash equilibrium $(\xi^{a^*},\xi^{p^*})$, then $(\xi^{a^*},\xi^{p^*})$ is an $\epsilon$-Nash equilibrium of the game problem defined by (\ref{Incentives_AMoD}) and (\ref{Incentives_PT}) with: 
\begin{equation}
\label{value_of_epsilon}
\epsilon=\bar{\pi}^a-\pi^a(\xi^{a^*},\xi^{p^*}).
\end{equation}
\end{proposition}
The above proposition can be easily derived based on Proposition \ref{proposition_upper_lower_bound} and the definition of the $\epsilon$-Nash equilibrium, thus its proof is omitted.

\section{Market Outcomes in the Absence of Regulation} \label{market outcomes}

This section investigates the equity impacts of AVs in the unregulated environment. To this end, we first propose evaluation metrics to quantify both spatial and social equity in the multimodal transportation system. Second, we demonstrate how AVs impact market outcomes and transport equity through numerical studies for San Francisco.

\subsection{Evaluation of Equity}

{We use Theil coefficient \cite{theil1967economics} to quantify the transport equity in the multimodal transport system. Theil coefficient, as one of the most common equity indicators, derives from the concept of information theory and was originally aimed at quantifying the level of disorder within a distribution of the variable, which can be formally written as: 
\begin{equation}
    T = \frac{1}{P_T} \sum_{i=1}^{P_T} \frac{y_i}{\bar{y}} \ln \left(\frac{y_i}{\bar{y}}\right)
\end{equation}
where $P_T$ is the total population, $y_i$ is the value of the variable associated with individual $i$, and $\bar{y}$ is the average per capita value of the variable in the population. Theil’s measure falls between 0 in the case of perfect equality and $\ln(P_T)$ for perfect inequality. The key feature of the Theil coefficient is that it can be perfectly decomposed into two components: inequity {\em within} and {\em between} population subgroups, which allows for a more detailed analysis of the sources of inequality in a system. Due to the decomposable structure, it has been widely applied to assess transportation equity from both spatial and social perspectives \cite{caggiani2019urban,camporeale2019study}.

To characterize spatial and social equity using the Theil coefficient, we first define $u_{ij,k}^t=-c_{ij,k}^t, t\in\mathcal{T}$ as the systematic utility of mobility mode $t$ for income class $k$ from zone $i$ to zone $j$. We then define the accessibility measure $A_{ij,k}$ as the {\em expected maximum utility}, which evaluates the expected received utility of passengers in income class $k$ from zone $i$ to zone $j$ in the multimodal transportation system. Based on the multinomial logit (MNL) model (\ref{logit_demand_function}), the expected maximum utility $A_{ij,k}$ can be written as the logsum formula \cite{ben1985discrete}:
\begin{equation} \label{logsum accessibility}
    A_{ij,k} = \frac{1}{\mu} \log \sum_{t\in\mathcal{T}} \exp\left(-\mu c_{ij,k}^t\right) .
\end{equation} 
Given the accessibility measure $A_{ij,k}$, we further denote $\lambda_{i,k}$, $\lambda_{k}$, $\overline{\lambda}$ as the passenger demand from zone $i$ in income class $k$, the passenger demand in income class $k$, and the total passenger demand, respectively. They satisfy:
\begin{subnumcases}{}
    \lambda_{i,k} = \sum_{j=1}^M \sum_{t\in\mathcal{T}} \lambda_{ij,k}^t \\
    \lambda_{k} = \sum_{i=1}^M \sum_{j=1}^M \sum_{t\in\mathcal{T}} \lambda_{ij,k}^t \\
    \overline{\lambda} = \sum_{i=1}^M \sum_{j=1}^M \sum_{k=1}^K \sum_{t\in\mathcal{T}} \lambda_{ij,k}^t
\end{subnumcases}
Correspondingly, we denote $A_{i,k}$, $A_k$, $\bar{A}$ as the average accessibility of passengers originating from zone $i$ in income class $k$, the average accessibility of passengers in income class $k$, and the average accessibility of all passengers, respectively. They are calculated as:
\begin{subnumcases}{}
    A_{i,k} = \frac{\sum_{j=1}^M \sum_{t\in\mathcal{T}} \lambda_{ij,k}^t A_{ij,k}}{\lambda_{i,k}} \\
A_k = \frac{\sum_{i=1}^M \sum_{j=1}^M \sum_{t\in\mathcal{T}} \lambda_{ij,k}^t A_{ij,k}}{\lambda_{k}} \\
\overline{A} = \frac{\sum_{i=1}^M \sum_{j=1}^M \sum_{k=1}^K \sum_{t\in\mathcal{T}} \lambda_{ij,k}^t A_{ij,k}}{\bar{\lambda}}
\end{subnumcases}
Finally, we define the Theil (T) coefficient of the accessibility distribution in the multimodal transportation system as:
\begin{equation} \label{Theil_coefficient}
    T = \text{WITHIN} + \text{BETWEEN} = \underbrace{\sum_{k=1}^K \sum_{i=1}^M \left(\frac{\lambda_{i,k}}{\overline{\lambda}} \frac{A_{i,k}}{A_k}\right) \ln \left(\frac{A_{i,k}}{A_k}\right)}_\text{spatial equity} + \underbrace{\sum_{k=1}^K \left(\frac{\lambda_k}{\bar{\lambda}} \frac{A_k}{\overline{A}}\right) \ln \left(\frac{A_k}{\bar{A}}\right)}_\text{social equity} .
\end{equation}
As shown in (\ref{Theil_coefficient}), the Theil coefficient can be decomposed into two components. The WITHIN component calculates the inequality in the distribution of accessibility across distinct geographic zones {\em within} different groups. Consequently, it characterizes the spatial inequity due to the differentiated mobility service across distinct zones. The BETWEEN component evaluates the disparity in accessibility distributions {\em between} different income classes, which captures the social inequity arising from the unequal socioeconomic status. Note that a larger Theil coefficient indicates a more inequitable distribution of benefits across different zones or socioeconomic groups. It can be easily verified that if all zones have the same level of accessibility, then the WITHIN component of (\ref{Theil_coefficient}) is zero. Similarly, if all classes of passengers have the same level of accessibility, then the BETWEEN component of (\ref{Theil_coefficient}) is zero.

\begin{remark}
    We acknowledge that in transport economic analysis, it is common to scale utility by the marginal utility of income, which converts utility to monetary units and measures accessibility as the expected consumer surplus. The difference in expected consumer surplus is also widely used to measure accessibility benefits arising from transport investments or policies \cite{ahmed2020quantifying}. However, monetizing utility as the accessibility measure would bias the accessibility impacts across distinct income classes, as pointed out by \cite{bills2017looking}. Assessing the disparity in the distribution of accessibility benefits may also not align with our objective of promoting an equitable multimodal transport network. Therefore, we define accessibility as the expected maximum utility, which coincides with the concept of accessibility that is first proposed in \cite{ben1985discrete}. It provides a subjective value that reflects the satisfaction or well-being experienced by distinct income classes, which serves as a comprehensive and unbiased indicator of accessibility.
\end{remark}

\subsection{Case Studies} \label{case study}

Consider a case study for San Francisco, where a TNC platform operates autonomous vehicles for mobility-on-demand services and the public transit agency manages a transit network to provide public transit services. We conduct a numerical study using realistic synthetic data for San Francisco. The data consists of the origin and destination of TNC trips at the zip-code granularity, which is synthesized based on the historic total TNC pickups and dropoffs in each zone \cite{SanFrancisco2022TNC} together with a choice model calibrated with survey data. The zip code zones of San Francisco are shown in Figure \ref{fig:SF_multimodal}. Based on the TNC data, we remove zip code zones 94123, 94127, 94129, 94130, 94134 from our analysis since they have negligible trip volumes. We also merge zip code zones 94104, 94105, and 94111 into a single zone, and merge zip code zones 94108 and 94133 into a single zone since each of these individual zones is very small. We synthesize the public transit network based on the Muni system in San Francisco, which is a network of light-rail metro trains, rapid buses, regular buses, cable cars, etc \cite{SanFrancisco2022Muni}. The synthetic transit network consists of the five light rail lines ('F', 'J', 'KT', 'M', and 'N') and the three rapid bus lines ('5R', '9R', '38R') since they form the major skeleton of the Muni network and undertake the majority of trip volumes. The route information of these transit lines can be found in \cite{SanFrancisco2022Muniroutes} and the synthetic transit network is shown in Figure~\ref{fig:SF_multimodal}. To guarantee that passengers are at least accessible to public transit in each zone, we further aggregate zip code zones 94123 and 94115 since they are geographically adjacent. We further define the underserved area $\mathcal{U}$ as consisting of zip code zones 94112, 94114, 94116, 94117, 94118, 94121, 94122, 94124, 94131, 94132 based on the density of transit stations in different zones. Note that the underserved area is generally the remote area in the city. Based on the service areas, we divide transit lines into two subsets: the high-frequency lines $\mathcal{L}_H=\text{\{'F','J','KT','9R'\}}$ that primarily serve the core area, and low-frequency lines $\mathcal{L}_L=\text{\{'N','M','5R','38'\}}$ which mainly serve the remote area.

\begin{figure}[t]
    \centering
    \includegraphics[width=0.58\linewidth]{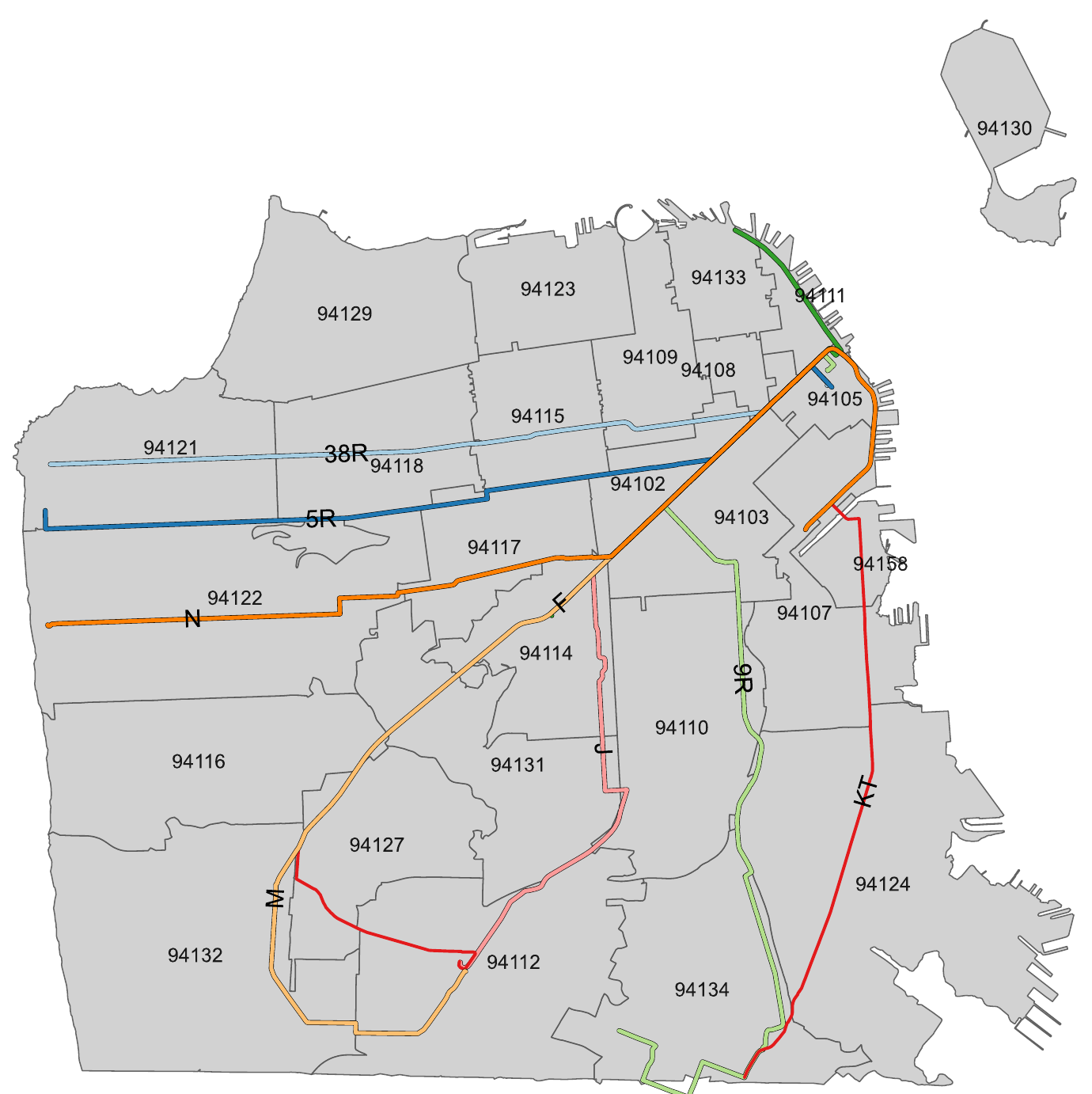}
    \caption{Zip code zones of San Francisco County and the synthetic public transit network.}
    \label{fig:SF_multimodal}
\end{figure}

We consider heterogeneous passengers with different income levels. In particular, passengers are categorized into three distinct classes, namely, low-income class (class 1), medium-income class (class 2), and high-income class (class 3). Generally, low-income people have a lower valuation of time but a higher valuation of money compared to high-income people, and passengers have a higher valuation of waiting/walking time than that of in-vehicle time. Therefore, the following conditions hold:
\begin{equation}
    \begin{cases}
    \alpha_k,\theta_k > \beta_k \quad \forall k=1,2,3 \\
    \alpha_1 < \alpha_2 < \alpha_3 \\
    \beta_1 < \beta_2 < \beta_3 \\
    \gamma_1 > \gamma_2 > \gamma_3
    \end{cases} .
\end{equation}

In summary, the model parameters involved in the multimodal transportation system are
\begin{equation}
    \Theta = \left\{\lambda_{ij,k}^0, \alpha_k, \beta_k, \gamma_k, \eta, \mu, c_{ij,k}^o, A_i, l_{ij}^a, l_{a}^p, d_{s_i}, v_a, v_p, v_w, C_{av}, C_{l}, \pi_0  \right\} .
\end{equation}
Due to the lack of real data on AMoD, we calibrate the values of these model parameters based on human-driver-based TNC data and Muni data in San Francisco and map it to the future scenario. In particular, $\lambda_{ij}^0$ is set to satisfy $0.15\lambda_{ij}^0=\lambda_{ij}$, where $\lambda_{ij}$ is the current trip distribution of ride-sourcing trips in San Francisco from zone $i$ to zone $j$, such that 15\% of potential passengers choose to take AMoD services. We further split the potential passenger demand $\lambda_{ij}^0$ among three income classes to obtain $\lambda_{ij,k}^0$. To characterize the spatial difference in socioeconomic characteristics, we assume population distributions of low-income, middle-income, and high-income individuals as 0.4, 0.5, and 0.1 in remote areas and 0.2, 0.5, and 0.3 in core areas\footnote{For simplicity, we define the remote area as the underserved area $\mathcal{U}$. The core area consists of the zip code zones except those in $\mathcal{U}$.}. The potential passenger demand $\lambda_{ij,k}^0$ is then calculated proportionally. The trip distance $l_{ij}^a$, the transit link length $l_{a}^p$, and the access/egress distance $d_{s_i}$ are obtained from Google map estimates. The average speed of TNC vehicles $v_a=17.937$ mph and the average operating speed of public transit $v_p=14.349$ mph are estimated based on the current TNC data and Muni data. 
The average walking speed is $v_w$ is set as 3.48 mph. The travel cost of the outside option $c_{ij,k}^o$ is assumed to be proportional to $l_{ij}^a$. To account for both socioeconomic and spatial heterogeneity in $c_{ij,k}^o$, we assume that $c_{ij,1}^o=1.2c_{ij,2}^0$ and $c_{ij,2}^o=1.25c_{ij,3}^o$ such that the travel cost of the outside option of low-income passengers is 20\% higher than that of medium-income passengers and 50\% higher than that of high-income passengers. We further assume that the per-distance cost of the outside transport mode in the remote area is 50\% higher than that in the urban core. Based on data from the Muni system \cite{SanFrancisco2022Munifares,SanFrancisco2022Munifrequencies}, the per-vehicle operating costs of different transit lines $C_{l}$ are calibrated.

The rest of the model parameters are set as:
\begin{gather*}
    \alpha_1 = 0.5; \quad \alpha_2 = 1.0; \quad \alpha_3 = 2.0; \quad \beta_1 = 0.15; \quad \beta_2=0.30; \quad \beta_3 = 0.65; \\ \gamma_1 = 3.0; \quad \gamma_2 = 1.5; \quad \gamma_3 = 0.75; \quad \theta_1 = 0.5; \quad \theta_2 = 1.0; \quad \theta_3 = 2.0; \\ \eta = 0.10; \quad \mu = 0.10; \quad A_i^{-1} = 7.894 \times \text{Area of zone $i$}; \quad \pi_0 =  \$1\mathrm{e}4 \text{/hour} .
\end{gather*}
We adjust the values of these parameters such that when the hourly operating cost of AVs is roughly equivalent to the average driver wage of TNCs, the corresponding market outcomes at Nash equilibrium are close to real-world data of San Francisco (e.g., modal share, trip volumes, average trip fare and frequencies of transit lines). In particular, when $C_{av}=\$30\text{/hour}$, we execute Algorithm \ref{algorithm1} to compute the Nash equilibrium and conduct the ex-post evaluation on the obtained equilibrium. The best response algorithm converges to the same equilibrium $\xi^*=(\xi^{a^*},\xi^{p^*})$ in a few iterations under distinct initial guesses. At the equilibrium, the TNC profit $\pi^a(\xi^{a^*},\xi^{p^*})$ is \$95,563/hour. Under a pairwise partition $\bar{\mathcal{V}}$\footnote{$\mathcal{V}$ is partitioned into subsets, each of which contains two zones, e.g., $\mathcal{V}_1=\{1,2\}$, $\mathcal{V}_2=\{3,4\}$, $\dots$, and $\mathcal{V}_9=\{17,18\}$.}, the established upper bound $\overline{\pi}^a$ is \$97,283/hour. Based on Proposition \ref{proposition_epsilon_nash}, this indicates that $\epsilon$ only accounts for 2.0\% compared with the globally optimal solution to (\ref{Incentives_AMoD}), which is a very tight approximate Nash equilibrium. In the meanwhile, the corresponding lower bound $\underline{\pi}^a$ is \$95,096/hour, which is very close to the equilibrium profit $\pi^a(\xi^{a^*},\xi^{p^*})$. This validates the nature of TNC's decision-making: geographic zones with higher potential profits are given higher priority.
For the market outcomes under $C_{av}=\$30/\text{hour}$, the TNC base fare is \$9.6/trip and the average per-distance rate of AMoD services is \$3.61/mile. The calculated trip fare $9.6+3.61\times2.6$ is close to the fare estimates \cite{SanFrancisco2022rate} for a 2.6-mile trip \cite{castiglione2016tncs}. The per-distance fare of public transit is \$1.15/mile and the service frequencies of low-frequency lines and high-frequency lines are 11.79 per hour and 18.08 per hour, which is roughly consistent with the current fare and frequency setting of Muni system \cite{SanFrancisco2022Muniroutes,SanFrancisco2022Munifares}. Moreover, the modal share of public transit is 22\%, which also coincides with the reported data in the survey \cite{SanFrancisco2022modalshare}.

To investigate the impacts of AVs on the multimodal transportation system, we gradually reduce the  cost of AVs, i.e., $C_{av}$, to simulate the scenario in which technological progress makes AVs less expensive over time, resulting in the ride-hailing industry's gradual acceptance of AVs. To this end, we fix all other model parameters, solve the game problem under distinct $C_{av}$ based on Algorithm \ref{algorithm1}, examine how the evolution of AV technology affects the market outcomes in the absence of regulations, and evaluate transport equity using Theil coefficient based on (\ref{Theil_coefficient}). The results are presented in Figure \ref{fig:TNC_no_reg}-\ref{fig:modal_share_no_reg}. Figure \ref{fig:TNC_no_reg} shows TNC's operational strategy and corresponding outcomes under distinct $C_{av}$. As the cost of AV reduces, the ride fare of TNC trips decreases (Figure \ref{fig:base_fare_no_reg}), the number of vehicles on the ride-hailing platform increases (Figure \ref{fig:total_number_vehicles_no_reg}), and the profit of the platform improves (Figure \ref{fig:TNC_profit_no_reg}). These results are intuitive since the reduction in AV cost lowers the marginal cost of AMoD services and thus both the TNC platform and passengers are direct beneficiaries of AV deployment. However, a less intuitive insight is that the spatial distribution of AMoD services is unbalanced:
\begin{itemize}
    \item Idle AVs are highly concentrated in the urban core (solid lines in Figure \ref{fig:number_idle_vehicle_no_reg}), indicating that the platform prefers to deploy idle vehicles in high-demand areas.  This is intuitive because TNC has higher utilization of AVs and higher potential profits in high-demand areas, which promotes it to deploy more idle vehicles to attract passengers from urban centers. However, the geographical concentration of ride-hailing vehicles will not only exacerbate traffic congestion in the urban core, but also leads to spatial inequality in the waiting time/service quality across distinct zones.
    \item As AV cost reduces, the number of idle vehicles increases in both high-demand and low-demand areas of the city, but the gap of service quality between high-demand and low-demand areas is further enlarged (i.e., compare solid and dashed lines in Figure \ref{fig:number_idle_vehicle_no_reg}). This indicates that without regulatory intervention, lowering the cost of AV will promote the proliferation of AVs, but at the same time exacerbate the existing spatial inequality gaps in the multimodal transportation network.
    \item AMoD services exhibit economies of scale, which contributes to the spatial disparity in AMoD services and exacerbates spatial inequity. To fully utilize the scale economies of AMoD services, the TNC platform disproportionally deploys idle vehicles in the urban center (Figure \ref{fig:number_idle_vehicle_no_reg}) since it has a higher potential passenger demand, which could exhibit stronger economies of scale and generate higher potential profits. As AV cost reduces, the impacts of scale economies become more significant, which exacerbates the geographic concentration of TNC services and enlarges spatial inequity gaps (see Figure \ref{fig:idle_AV_scale_04}-\ref{fig:T_scale}). A detailed exploration of the impacts of scale economies can be found in Appendix~E.
\end{itemize}

\begin{figure}[H]
     \centering
     \begin{subfigure}[b]{0.48\textwidth}
         \centering
         \includegraphics[width=\textwidth]{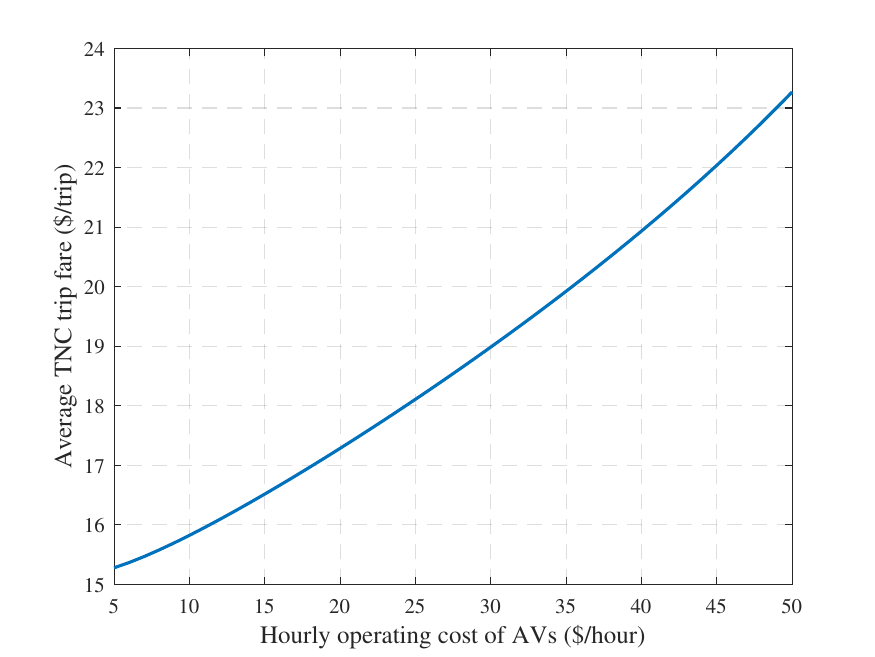}
         \caption{{Average TNC trip fare under distinct $C_{av}$.}}
         \label{fig:base_fare_no_reg}
     \end{subfigure}
     \hfill
     \begin{subfigure}[b]{0.48\textwidth}
         \centering
         \includegraphics[width=\textwidth]{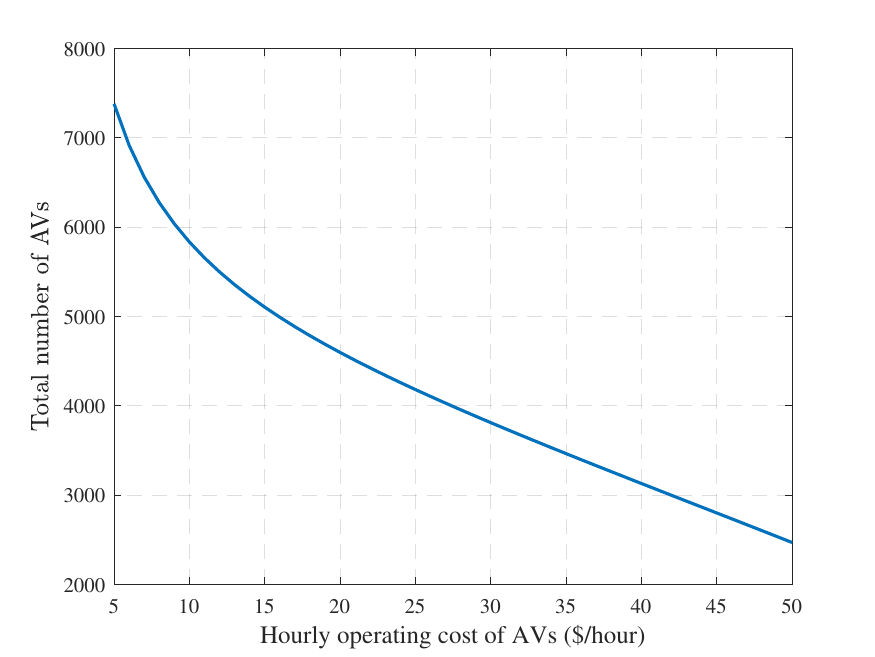}
         \caption{Total number of AVs under distinct $C_{av}$.}
         \label{fig:total_number_vehicles_no_reg}
     \end{subfigure}

     \begin{subfigure}[b]{0.48\textwidth}
         \centering
         \includegraphics[width=\textwidth]{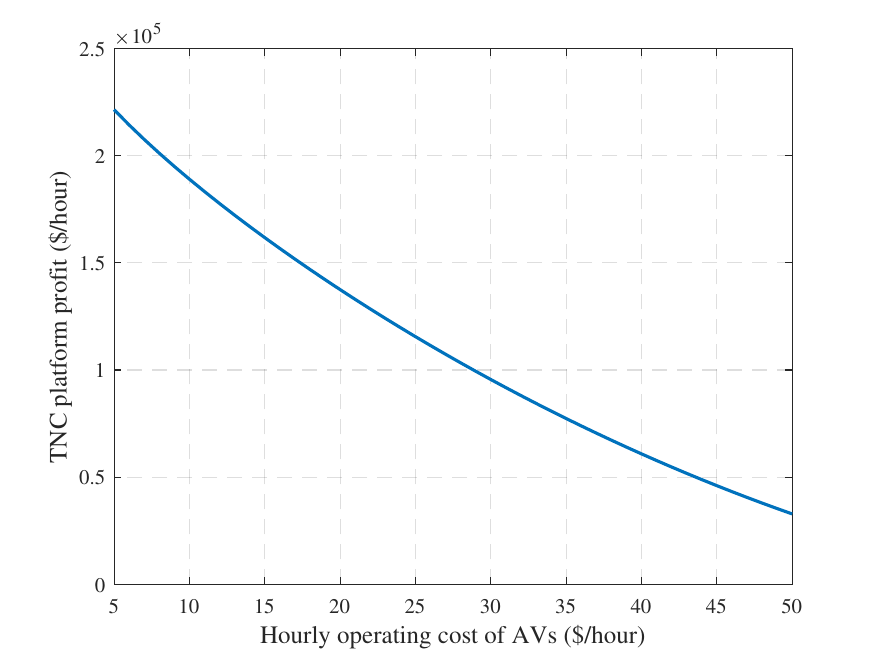}
         \caption{Hourly TNC profit under distinct $C_{av}$.}
         \label{fig:TNC_profit_no_reg}
     \end{subfigure}
     \hfill  
     \begin{subfigure}[b]{0.48\textwidth}
         \centering
         \includegraphics[width=\textwidth]{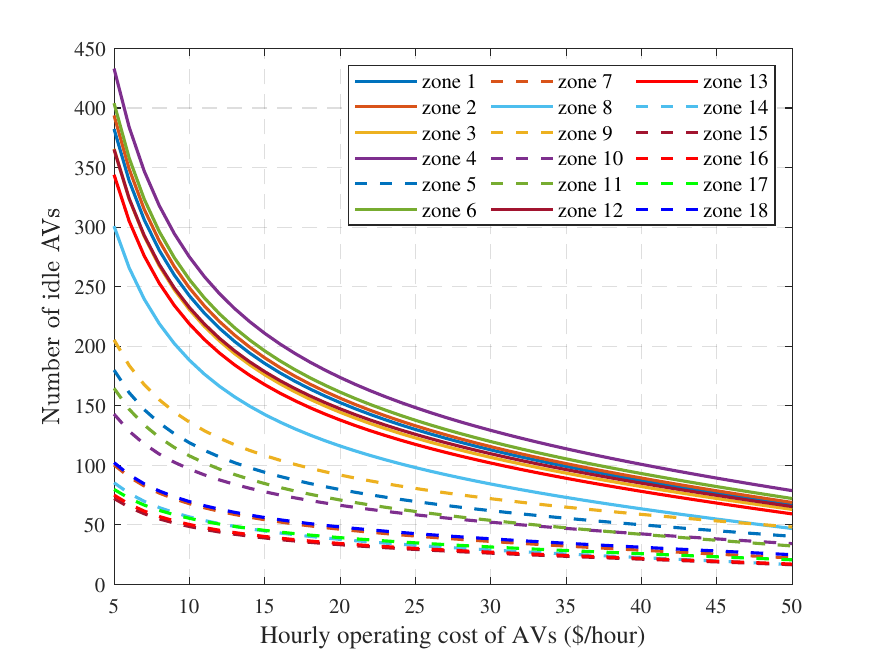}
         \caption{Number of idle vehicles in different zones under distinct $C_{av}$.}
         \label{fig:number_idle_vehicle_no_reg}
     \end{subfigure}
    \caption{TNC's operational decisions and outcomes under distinct $C_{av}$.}
    \label{fig:TNC_no_reg}
\end{figure}

To understand how the deployment of AVs affects social equity, we shall compare the impacts of AVs on distinct passenger income classes. Note that low-income populations are more transit-dependent, thus to characterize social equity, we will first examine the operational decisions of the transit agency. 
Figure \ref{fig:PT_no_reg} presents the public transit agency's operational decisions and the corresponding outcomes under distinct $C_{av}$. Based on the simulation results, the proliferation of AVs has a moderate impact on the transit operational decisions:  as $C_{av}$ reduces, the fare of public transit monotonically increases (Figure \ref{fig:transit_fare_no_reg}),  the service frequencies of most transit lines decreases (Figure \ref{fig:transit_frequency_no_reg}), the revenue and operating cost of the transit network slightly decrease (Figure \ref{fig:transit_revenue_no_reg}), and the total ridership of public transit under distinct $C_{av}$ exhibits two regimes: when $C_{av}>\$36/\text{hour}$, the transit ridership increases as $C_{av}$ reduces; when $C_{av}\leq \$36/\text{hour}$, the transit ridership decreases as $C_{av}$ drops (Figure \ref{fig:transit_ridership_no_reg}).  The above results reveal some interesting competitive and complementary effects of AMoD services on public transit. Note that the TNC provides on-demand door-to-door mobility services (direct AMoD services, i.e., mode $a$), which is a substitution for public transit. On the contrary, TNC trips serving as first/last-mile connections to public transit (bundled services, i.e., mode $b_1$, $b_2$, and $b_3$) could be a complement/feeder to public transit. The TNC and public transit agency carefully determine their operational strategies to balance the competitiveness and complementarity between AMoD services and public transit. When $C_{av}$ is relatively high, the competitiveness/substitution of AMoD services to public transit is weak, and the complementarity dominates. To strengthen the complementary effect and fully utilize the benefits of reduced AV cost, the public transit agency improves the service frequencies of some transit lines to attract more passengers to choose bundled services (Figure \ref{fig:transit_frequency_no_reg}). In this case, although the number of passengers choosing direct public transit drops, the passenger arrival rates of bundled services increase, which leads to increased total ridership of public transit (Figure \ref{fig:transit_ridership_no_reg}). The total operating cost of the transit network increases (Figure \ref{fig:transit_revenue_no_reg}), and therefore the agency raises the transit fare to maintain a balance of revenue and expenditure (Figure \ref{fig:transit_fare_no_reg}). On the other hand, when $C_{av}$ is relatively low, the competitiveness/substitution of direct AMoD services to public transit increases as $C_{av}$ reduces and gradually becomes the dominant effect. The competition between the TNC and public transit deteriorates transit services and procures passengers to deviate from public transit (Figure \ref{fig:transit_frequency_no_reg}). The total transit ridership significantly decreases as $C_{av}$ reduces (Figure \ref{fig:transit_ridership_no_reg}), and both the revenue and operating cost of the transit network decrease (Figure \ref{fig:transit_revenue_no_reg}). However, we acknowledge that the above impacts are insignificant because the AV cost only influences the transit agency indirectly through competition.

To further investigate the impact of AVs on social equity, we compare the mode choices of different income classes of the multimodal transportation system in Figure \ref{fig:modal_share_no_reg}.  As AV cost reduces, the modal shares of direct AMoD services (mode $a$) and bundled services (mode $b_1$, $b_2$, and $b_3$) increase, and the proportion of passengers choosing public transit (mode $p$) decreases for distinct income classes (Figure \ref{fig:modal_share_low_no_reg}-\ref{fig:modal_share_high_no_reg}). Overall, passengers shift from the outside option (e.g., walking, bicycling, and driving) and public transit to direct AMoD services and bundled services as AV cost reduces (Figure \ref{fig:modal_share_overall_no_reg}). These results are intuitively straightforward: the decrease in AV cost lowers the marginal cost of AMoD services, which promotes cheaper and improved AMoD services and instead disrupts transit services (in the second regime). However, the impacts of AVs on the modal choices of distinct classes of travelers are different, indicating that the distribution of benefits of AV is unequal:
\begin{itemize}
    \item Low-income individuals are highly transit-dependent. More than half of low-income individuals take public transit and over a third of them choose the outside option (e.g., walking or bicycling) (Figure \ref{fig:modal_share_low_no_reg}). As AV cost reduces, the increase in the modal share of AMoD services and bundled services is insignificant (Figure \ref{fig:modal_share_low_no_reg}). To understand this, note that low-income people have a relatively high valuation of monetary cost and a low valuation of time. As AV cost reduces, although ride fares and the waiting time of AMoD services decrease, the TNC deploys a small number of idle AVs and set relatively high ride fares (compared to transit) in remote areas with a high proportion of low-income individuals, who have to stick to public transit although its travel cost increases.
    \item The evolution of AV technology significantly transforms medium-income individuals' mode choices. As AV cost reduces, the number of medium-income individuals taking transit decreases while the passenger arrival rate of direct AMoD services gradually increases and finally surpasses that of public transit (Figure \ref{fig:modal_share_medium_no_reg}). In the meanwhile, an increasing number of medium-income individuals choose bundled services (Figure \ref{fig:modal_share_medium_no_reg}), although they prefer to use AMoD as the first-mile or last-mile connection rather than for both the first-mile and last-mile connection (Figure \ref{fig:modal_share_medium_no_reg}). 
    \item High-income individuals prefer convenient and fast mobility options. As AV cost reduces, the majority of high-income individuals shift from the outside option (e.g., driving) to direct AMoD services and only a small proportion of individuals choose public transit and bundled services (Figure \ref{fig:modal_share_high_no_reg}). High-income people have a high valuation of time and therefore they are unwilling to experience a long trip time and prefer driving or direct AMoD services.
\end{itemize}

\begin{figure}[t]
     \centering
     \begin{subfigure}[b]{0.48\textwidth}
         \centering
         \includegraphics[width=\textwidth]{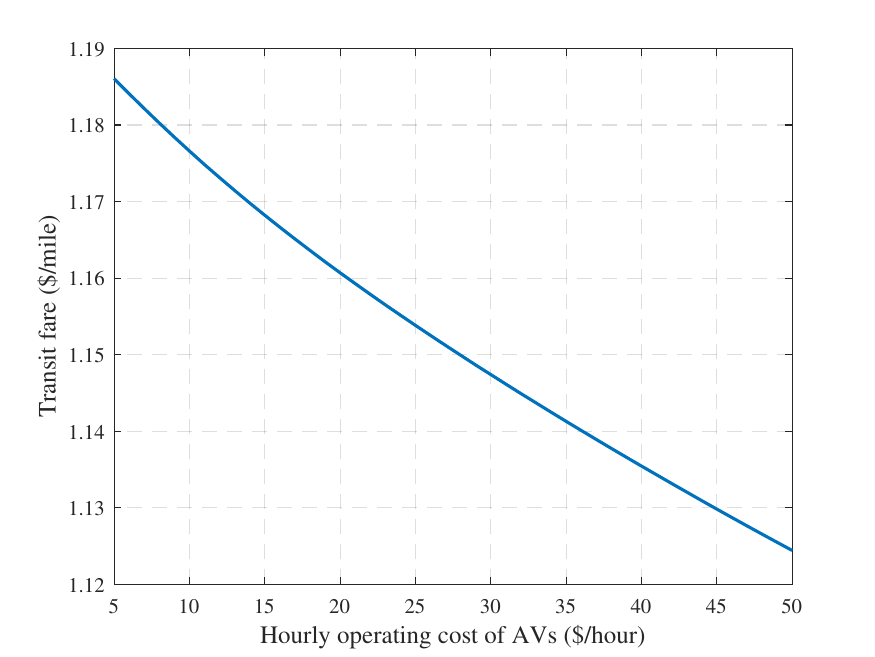}
         \caption{Per-distance transit fare under distinct $C_{av}$.}
         \label{fig:transit_fare_no_reg}
     \end{subfigure}
     \hfill
    \begin{subfigure}[b]{0.48\textwidth}
     \centering
     \includegraphics[width=\textwidth]{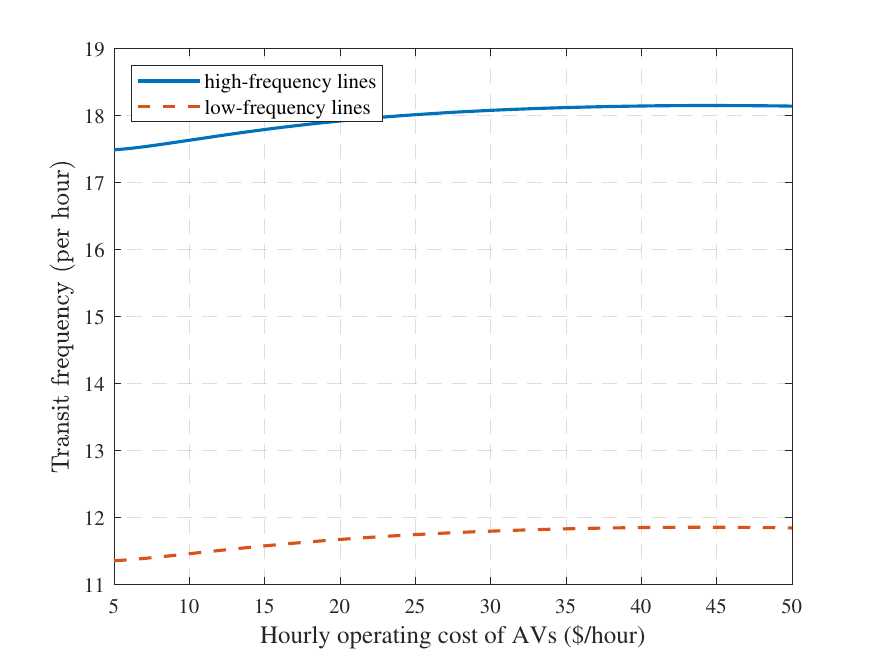}
     \caption{Frequencies of different transit lines under distinct $C_{av}$.}
     \label{fig:transit_frequency_no_reg}
    \end{subfigure}
    \begin{subfigure}[b]{0.48\textwidth}
         \centering
         \includegraphics[width=\textwidth]{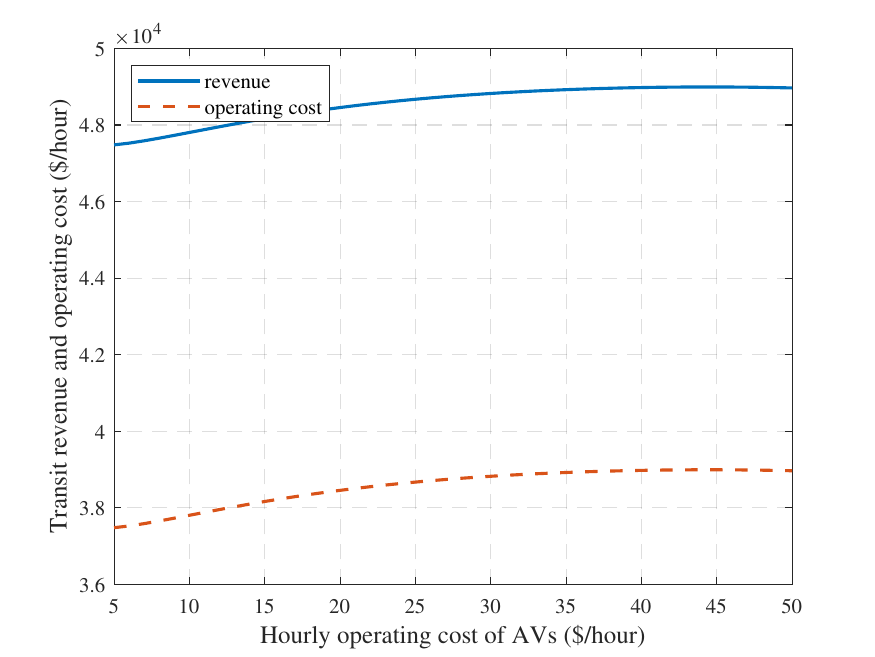}
         \caption{Transit revenue and operating cost under distinct $C_{av}$.}
         \label{fig:transit_revenue_no_reg}
    \end{subfigure}
    \hfill
    \begin{subfigure}[b]{0.48\textwidth}
         \centering
         \includegraphics[width=\textwidth]{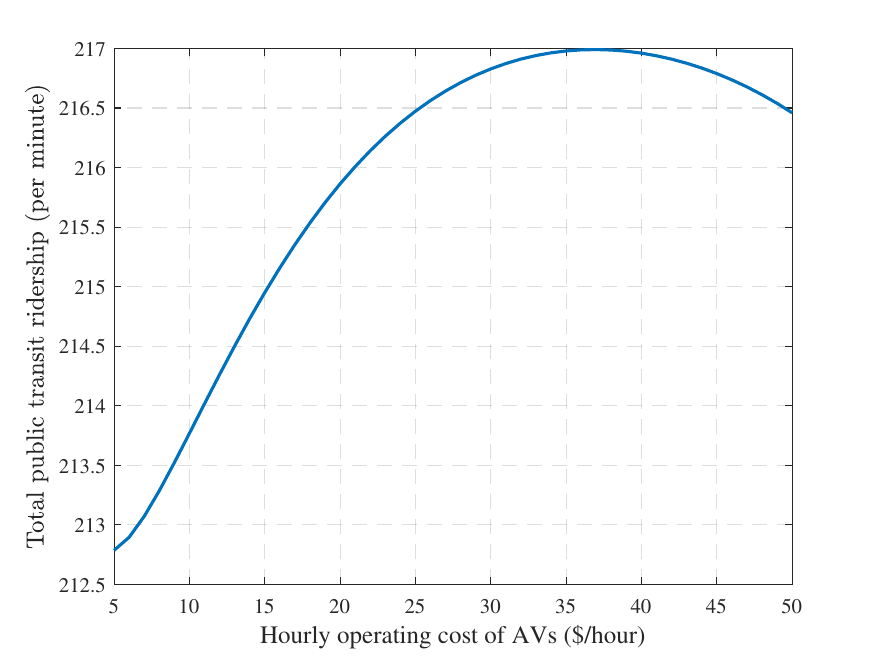}
         \caption{Total public transit ridership under distinct $C_{av}$.}
         \label{fig:transit_ridership_no_reg}
    \end{subfigure}
    \caption{Public transit's operational decisions and outcomes under distinct $C_{av}$.}
    \label{fig:PT_no_reg}
\end{figure}

To summarize, the above analysis indicates that the proliferation of AVs will increase both spatial and social inequity gaps in the multimodal transportation network. This is consistent with the simulation results presented in  Figure \ref{fig:Theil_no_reg}, where both the WITHIN component (i.e, spatial inequity) and the BETWEEN component (i.e., social inequity) increase and the Theil coefficient significantly increases as $C_{av}$ reduces. We comment that the reason for the expansion of both the spatial inequity gap and social inequity gap is twofold. On the one hand, the spatial inequality in accessibility across different zones is mainly caused by the geographic concentration of AMoD services. As mentioned above, due to the for-profit nature of the TNC platform, idle AVs are highly concentrated in the urban core (Figure \ref{fig:number_idle_vehicle_no_reg}), and passengers in the urban core enjoy more convenient services compared to those in remote areas. The geographical imbalance in AMoD services across distinct zones causes spatial inequity in accessibility. Besides, such geographic imbalance in AMoD services exacerbates as AV cost reduces (Figure \ref{fig:number_idle_vehicle_no_reg}), which enlarges the spatial inequality in accessibility and reinforces the spatial inequity gap (Figure \ref{fig:Theil_no_reg}). On the other hand, the increase in social inequity is not only due to the moderate disruptiveness of AMoD services to public transit, but also because the benefits of reduced AV costs are primarily enjoyed by medium-income and high-income individuals, and the increase in the accessibility of low-income individuals is less significant compared to that of medium-income and high-income individuals (Figure \ref{fig:accessibility_groups}). Therefore, the social inequity gap exacerbates as AV cost reduces (Figure \ref{fig:Theil_no_reg}).

\begin{figure}[H]
     \centering
     \begin{subfigure}[b]{0.48\textwidth}
         \centering
         \includegraphics[width=\textwidth]{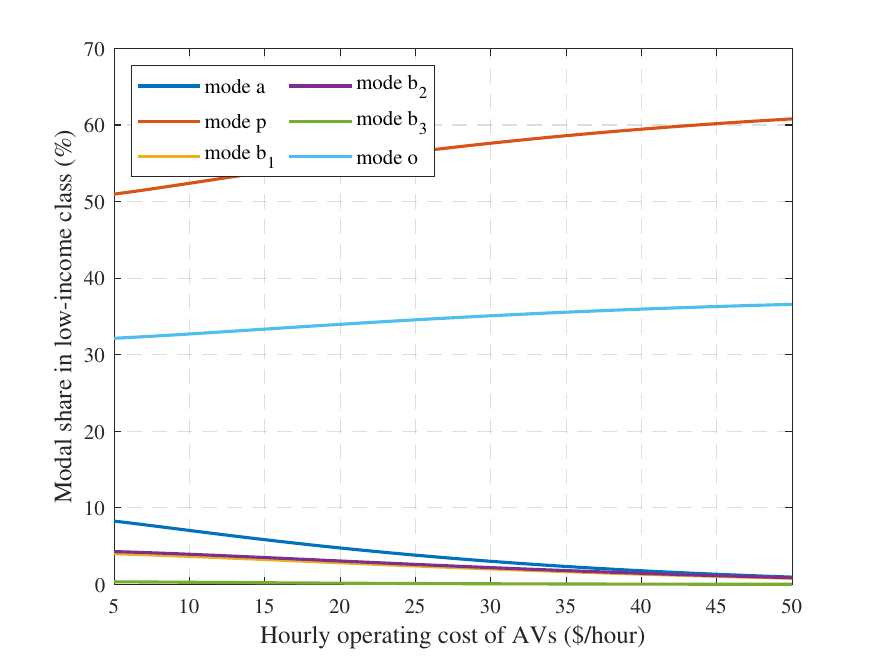}
         \caption{Modal share in low-income class under distinct $C_{av}$.}
         \label{fig:modal_share_low_no_reg}
     \end{subfigure}
     \hfill
    \begin{subfigure}[b]{0.48\textwidth}
     \centering
     \includegraphics[width=\textwidth]{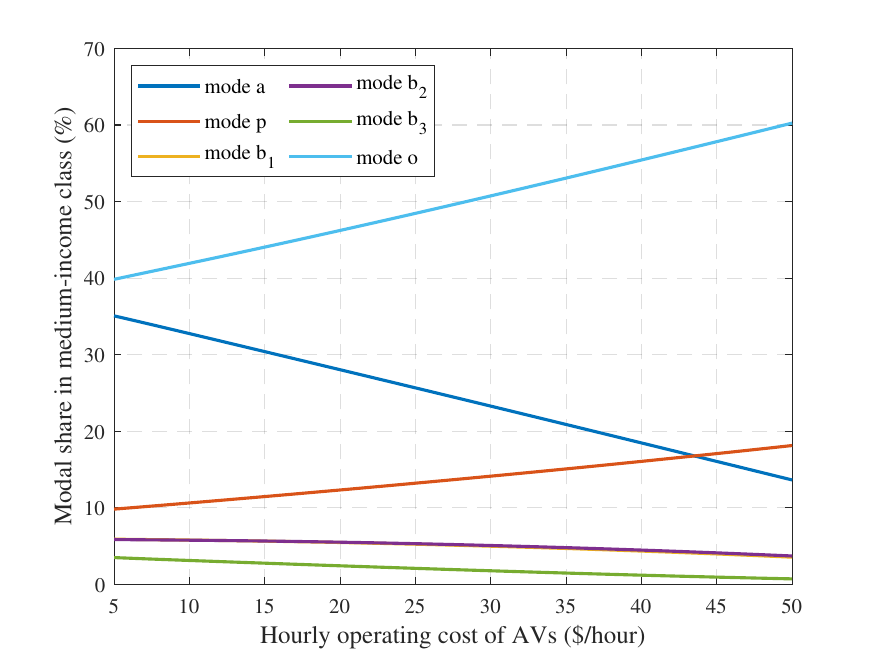}
     \caption{Modal share in medium-income class under distinct $C_{av}$.}
     \label{fig:modal_share_medium_no_reg}
    \end{subfigure}
    \begin{subfigure}[b]{0.48\textwidth}
         \centering
         \includegraphics[width=\textwidth]{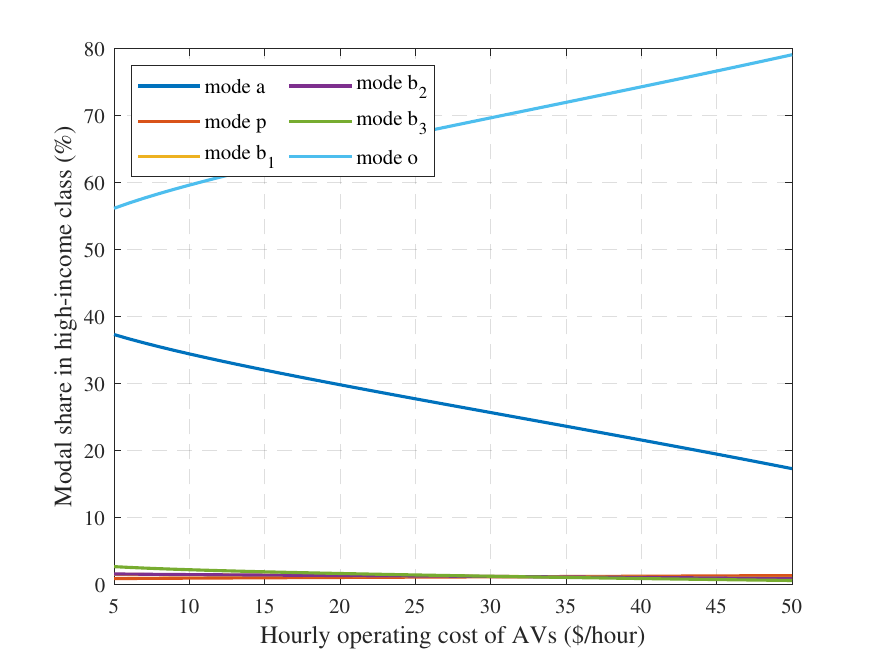}
         \caption{Modal share in high-income class under distinct $C_{av}$.}
         \label{fig:modal_share_high_no_reg}
    \end{subfigure}
    \hfill
    \begin{subfigure}[b]{0.48\textwidth}
         \centering
         \includegraphics[width=\textwidth]{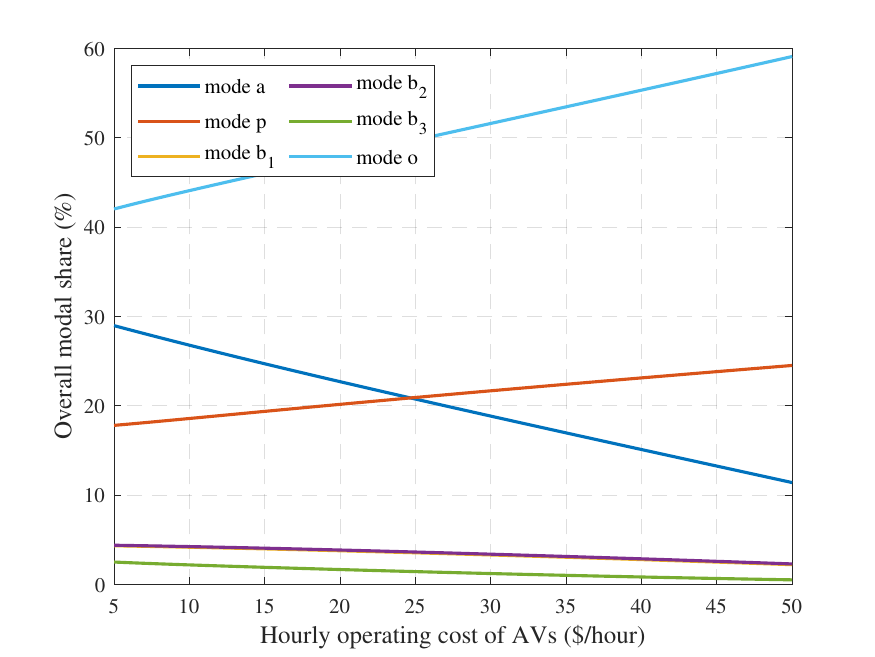}
         \caption{Overall modal share under distinct $C_{av}$.}
         \label{fig:modal_share_overall_no_reg}
    \end{subfigure}
    \begin{subfigure}[b]{0.48\textwidth}
         \centering
         \includegraphics[width=\textwidth]{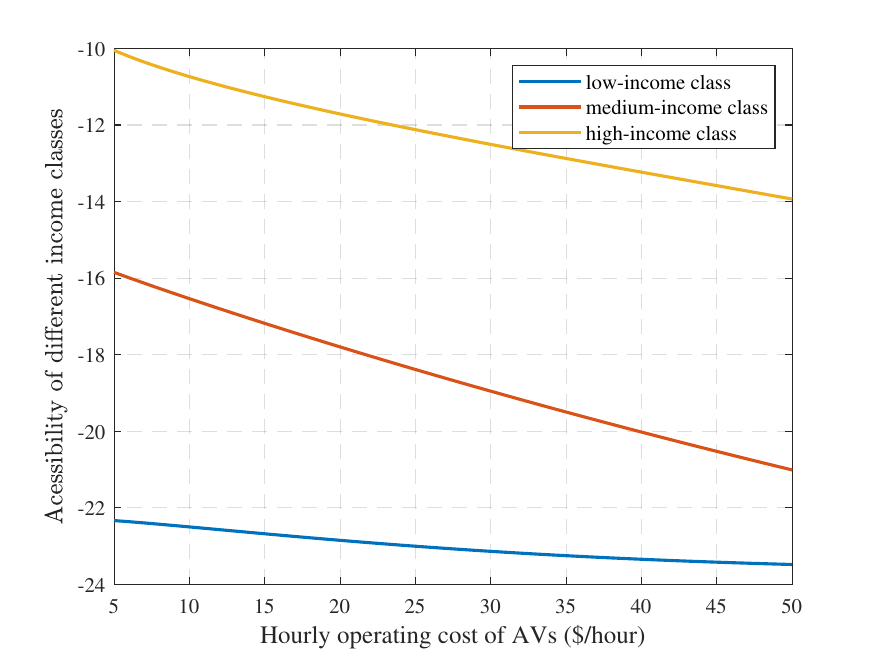}
         \caption{Average accessibility of different income groups under distinct $C_{av}$.}
         \label{fig:accessibility_groups}
     \end{subfigure}
     \hfill
     \begin{subfigure}[b]{0.48\textwidth}
         \centering
         \includegraphics[width=\textwidth]{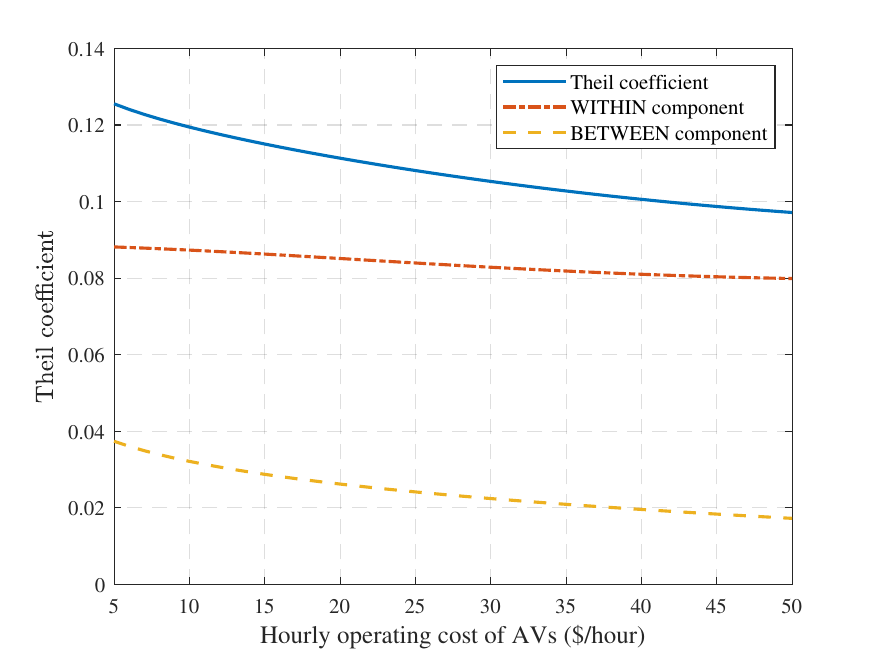}
         \caption{Theil coefficient, WITHIN component, and BETWEEN component under distinct $C_{av}$}
         \label{fig:Theil_no_reg}
     \end{subfigure}    
    \caption{Modal share and transport equity in the multimodal transportation system under distinct $C_{av}$.}
    \label{fig:modal_share_no_reg}
\end{figure}

\section{Impacts of Regulatory Policies} \label{regulations}

Our numerical study so far suggests that spatial inequality and social inequality gap enlarges as AV cost reduces. This section proposes two regulatory policies, i.e., a minimum service-level requirement and a subsidy on transit services, to mitigate the negative equity impacts of AVs in the multimodal transportation network.

\subsection{A Minimum Service-Level Requirement} \label{service}

To improve spatial equity, we consider a service-level requirement that mandates TNCs to guarantee an acceptable waiting time for AMoD services in all areas of the city. Such policy has already been implemented in practice, for instance,  New York City requires that 90\% of wheelchair-accessible vehicle requests be fulfilled within 15 minutes \cite{NYCWAV}. Under the proposed model, this policy can be easily modeled as:
\begin{equation} \label{minimum_serivce}
    w_i^a\left( N_i^I \right) \leq w_{\text{max}}^a, \forall i=1,2,\dots,M ,
\end{equation}
where $w_i^a(\cdot)$ is the average waiting time for AMoD trips starting in zone $i$, and $w_{\text{max}}^a$ is the maximum allowable average zonal waiting time set by the regulatory agency to ensure reasonable and equitable service quality for travelers across distinct zones. In this case, the TNC determines its operational strategy to maximize its profit subject to the market equilibrium conditions (\ref{generalized_cost_a})-(\ref{vehicle_hour_conservation}) and the extra minimum service-level requirement (\ref{minimum_serivce}). The public transit agency strategically adjusts its operational decisions to maximize transit ridership subject to the equilibrium constraints (\ref{generalized_cost_a})-(\ref{logit_demand_function}) and (\ref{frequency_consistency})-(\ref{profit_constraint}). Note that the minimum service-level requirement does not change the structure of the profit maximization problem for the TNC and the ridership maximization problem for public transit. We can use Algorithm \ref{algorithm1} to solve the game problem and conduct an ex-post evaluation on the computed Nash equilibrium.

To investigate the impacts of the minimum service-level requirement, we fix $C_{av}$ as 20\$/hour and keep other model parameters unchanged, solve the game problem under distinct $w_{\text{max}}^a$, and present how it impacts the market outcomes and transport equity in Figure \ref{fig:TNC_service}-\ref{fig:equity_service}.

\begin{figure}[t!]
     \centering
     \begin{subfigure}[b]{0.48\textwidth}
         \centering
         \includegraphics[width=\textwidth]{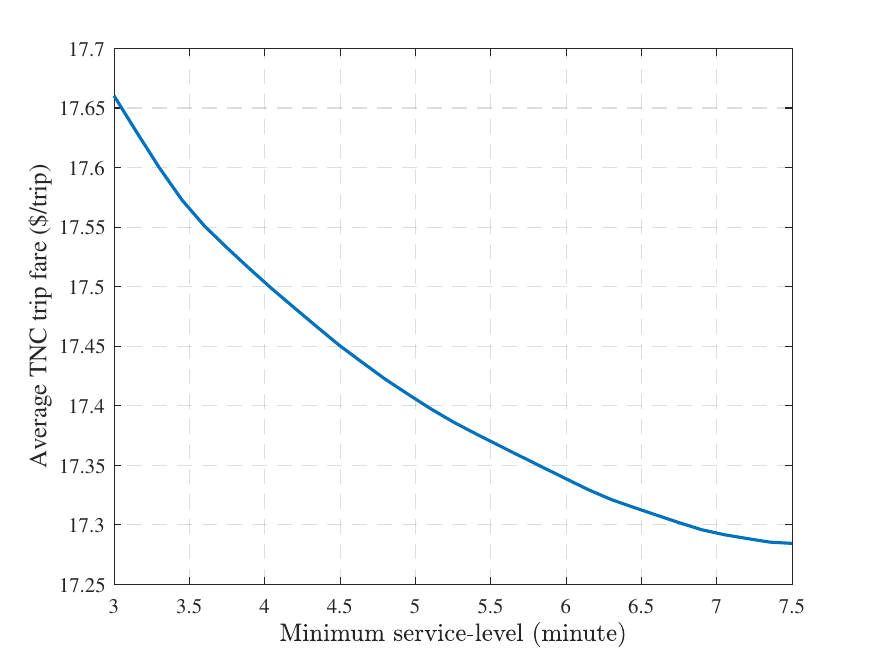}
         \caption{{Average trip fare of the ride-hailing services under distinct values of $w_{\text{max}}^a$.}}
         \label{fig:base_fare_service}
     \end{subfigure}
     \hfill
     \begin{subfigure}[b]{0.48\textwidth}
         \centering
         \includegraphics[width=\textwidth]{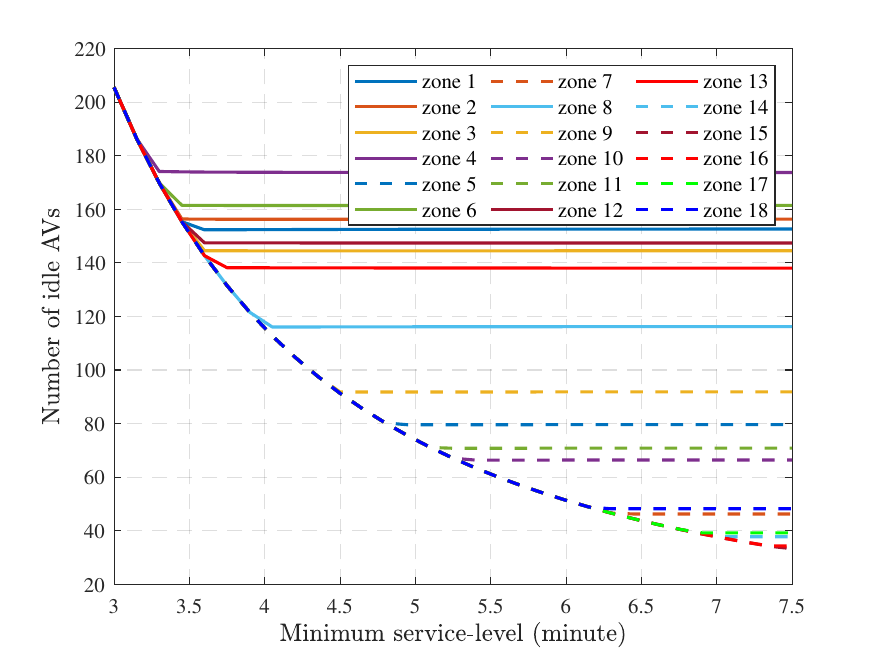}
         \caption{Number of idle vehicles in different zones of the transport network under distinct values of $w_{\text{max}}^a$.}
         \label{fig:number_idle_vehicle_service}
     \end{subfigure}

     \begin{subfigure}[b]{0.48\textwidth}
         \centering
         \includegraphics[width=\textwidth]{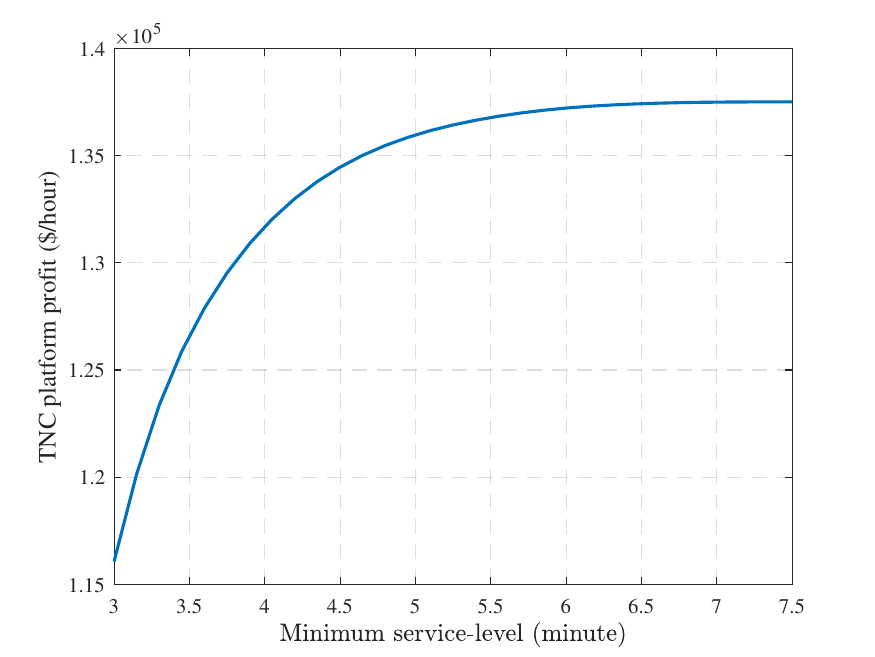}
         \caption{Hourly profit of the TNC platform under distinct values of $w_{\text{max}}^a$.}
         \label{fig:TNC_profit_service}
     \end{subfigure}
     \hfill
     \begin{subfigure}[b]{0.48\textwidth}
         \centering
         \includegraphics[width=\textwidth]{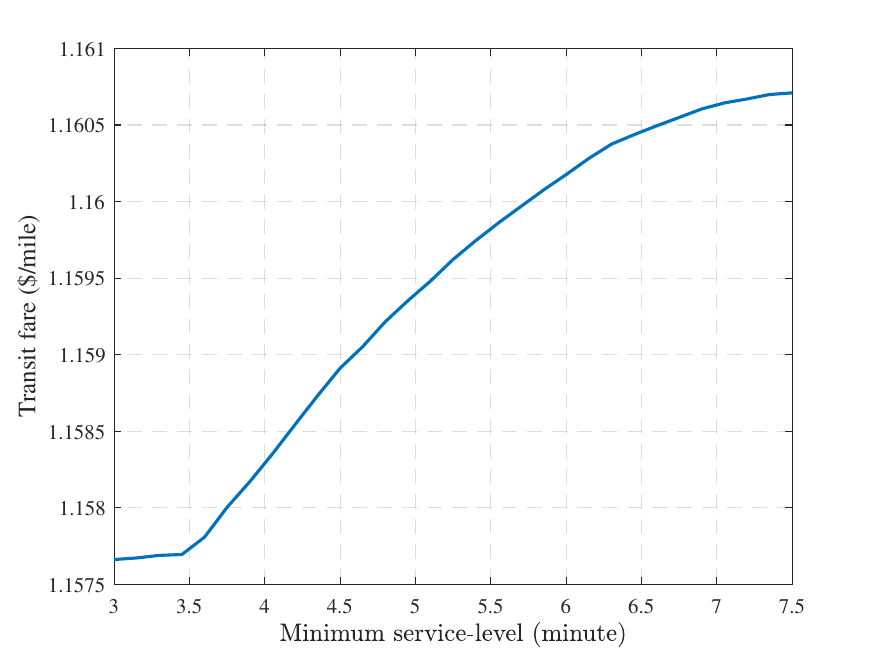}
         \caption{ Trip fare (per-distance) of transit services under different values of $w_{\text{max}}^a$.}
         \label{fig:transit_fare_service}
     \end{subfigure}
     
    \begin{subfigure}[b]{0.48\textwidth}
     \centering
     \includegraphics[width=\textwidth]{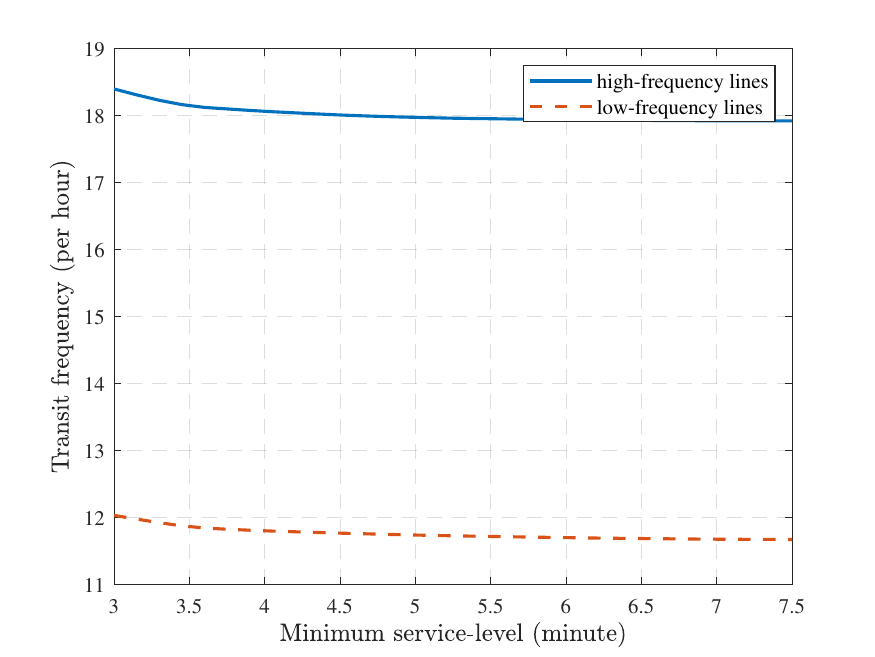}
     \caption{Frequencies of different transit lines under distinct values of $w_{\text{max}}^a$.}
     \label{fig:transit_frequency_service}
    \end{subfigure}
     \hfill
    \begin{subfigure}[b]{0.48\textwidth}
         \centering
         \includegraphics[width=\textwidth]{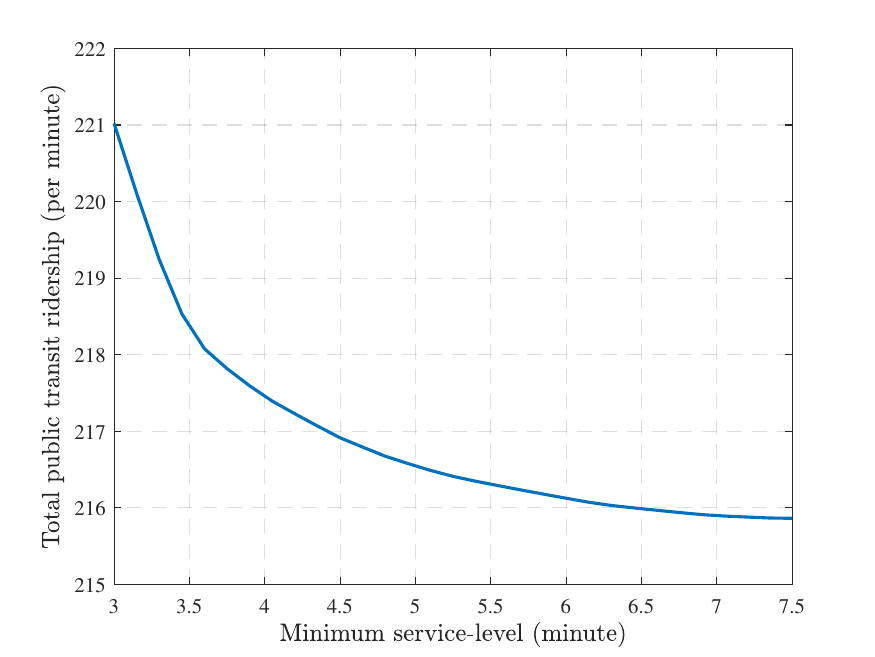}
         \caption{Total public transit ridership per minute under distinct values of $w_{\text{max}}^a$.}
         \label{fig:transit_ridership_service}
    \end{subfigure}
    \caption{Operational decisions of the TNC platform and transit agency under distinct $w_{\text{max}}^a$.}
    \label{fig:TNC_service}
\end{figure}

Figure \ref{fig:TNC_service} shows the operational strategies of TNC platform and the public transit agency under distinct $w_{\text{max}}^a$. Under a stricter minimum service-level requirement,  the number of idle vehicles in distinct zones remains unchanged initially and then increases, especially in the low-demand areas (Figure \ref{fig:number_idle_vehicle_service}). In this case, the TNC earns a reduced profit under a stricter minimum service-level requirement (Figure \ref{fig:TNC_profit_service}) and increases the average ride fare to improve profitability (Figure \ref{fig:base_fare_service})). Clearly, this lower bound enforces the TNC to improve the service quality in low-demand areas where the TNC is reluctant to place its idle AVs in the absence of regulations. This improves spatial equity while hurting the profitability of TNC's AMoD business.
Figure \ref{fig:transit_fare_service}-\ref{fig:transit_ridership_service} presents the transit agency's operational strategies and the corresponding outcomes under distinct $w_{\text{max}}^a$. As $w_{\text{max}}^a$ reduces, the agency slightly lowers the transit fare (Figure \ref{fig:transit_fare_service}) and increases the service frequencies of its transit lines (Figure \ref{fig:transit_frequency_service}). This leads to reduced waiting time for transit services and increased total transit ridership (Figure \ref{fig:transit_ridership_service}). Similarly, this can be explained as a consequence of the complementarity between AMoD services and public transit. Imposing a minimum service-level requirement improves the service quality of AMoD services, especially in underserved/remote areas (Figure \ref{fig:number_idle_vehicle_service}). It promotes the agency to improve the transit service and reinforce the AMoD-transit complementarity to attract more passengers to choose bundled services and stimulate a higher transit ridership.

Figure \ref{fig:equity_service} demonstrates the impacts of the minimum service-level requirement on accessibility and transport equity in the multimodal transportation system. As $w_{\text{max}}^a$ reduces,  the inequality in accessibility across distinct zones reduces (Figure \ref{fig:Theil_service}), while the inequality gap between the accessibility of different income groups keeps widening (Figure \ref{fig:accessibility_groups_service}-\ref{fig:Theil_service}). Consequently, as $w_{\text{max}}^a$ reduces, the spatial inequity (WITHIN component) decreases, whereas the social inequity (BETWEEN component) keeps increasing, and the overall inequity gap (Theil coefficient) first reduces and then enlarges (Figure \ref{fig:Theil_service}). This indicates that imposing a minimum service-level requirement is double-edged: on the one hand, it reduces the geographic concentration of idle vehicles and improves spatial equity; on the other hand, enforcing more idle AVs in less lucrative areas primarily benefits high-income individuals in underserved areas, who rely on AMoD much more than low-income individuals in the same zone. In this case, the accessibility improvement of low-income individuals is insignificant due to the increased TNC ride fares and high-income individuals are the primary beneficiaries of the minimum service-level requirement, which further reinforces the existing social inequity gap.

\begin{figure}[t]
     \centering
     \begin{subfigure}[b]{0.48\textwidth}
         \centering
         \includegraphics[width=\textwidth]{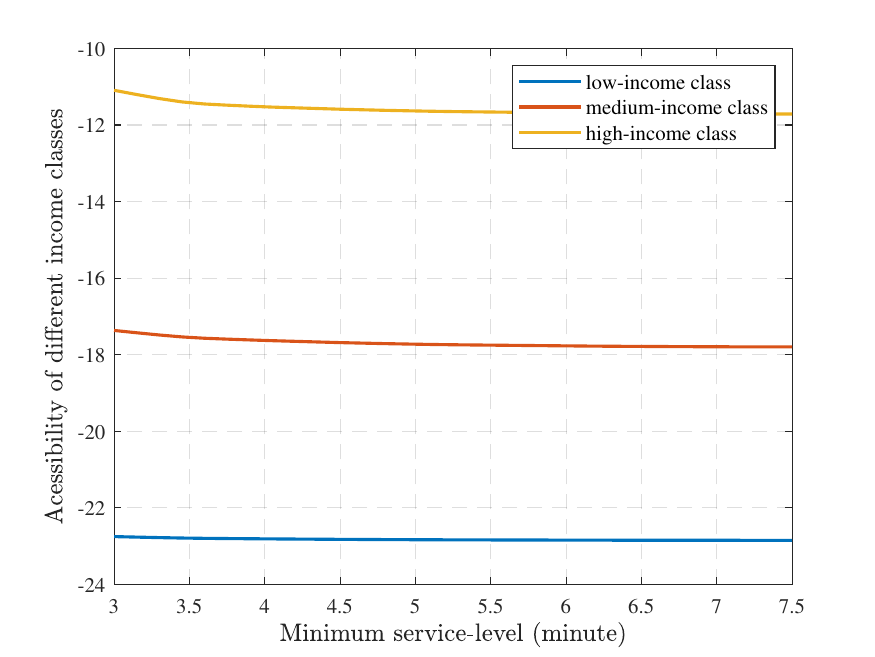}
         \caption{Average accessibility of different income groups under distinct valuse of $w_{\text{max}}^a$.}
         \label{fig:accessibility_groups_service}
     \end{subfigure}
     \hfill
     \begin{subfigure}[b]{0.48\textwidth}
         \centering
         \includegraphics[width=\textwidth]{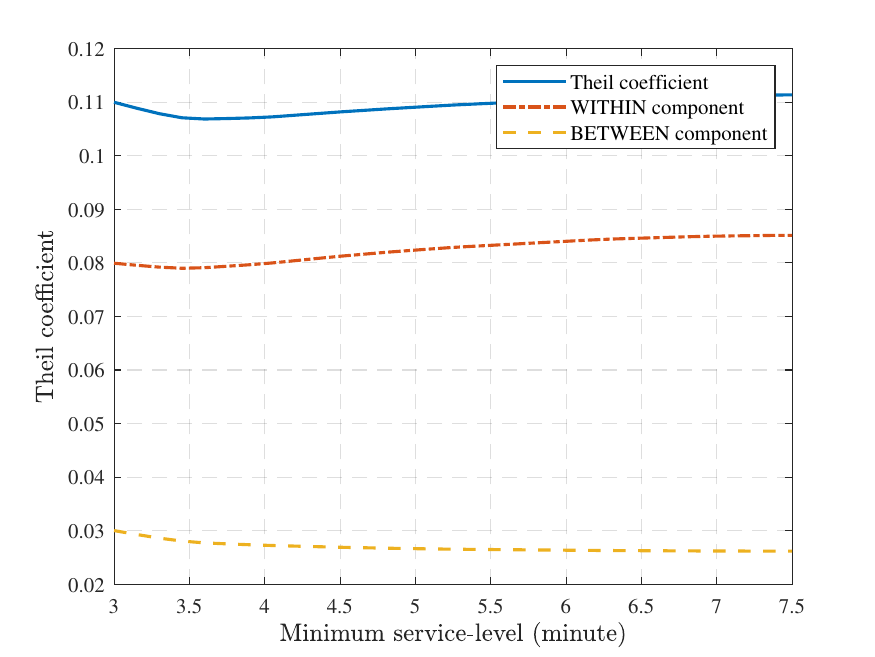}
         \caption{Theil coefficient, WITHIN component, and BETWEEN component under distinct $w_{\text{max}}^a$}
         \label{fig:Theil_service}
     \end{subfigure}
    \caption{Accessibility and transport equity under distinct $w_{\text{max}}^a$.}
    \label{fig:equity_service}
\end{figure}

\subsection{Subsidies on transit services} \label{subsidies}

To simultaneously improve spatial and social equity, we consider a subsidy on transit trips that target low-income households in underserved/remote areas. Specifically, a subsidy $s$ is provided for transit trips (including those bundled with AMoD) that originate from zone $i \in \mathcal{U}$ for income class $k=1$. Under the subsidy, the disutility/generalized travel costs of distinct mobility options can be characterized as:
\begin{subnumcases}{\label{generalized_costs_subsidy}}
    c_{ij,k}^a = \alpha_k w_i^a + \beta_k \frac{l_{ij}^a}{v_a} + \gamma_k (b + r_i^a l_{ij}^a) \\
    c_{s_i s_j,k}^p = \alpha_k w_{s_i s_j,k}^p + \beta_k \frac{l_{s_i s_j,k}^p}{v_p} + \gamma_k (r^p l_{s_i s_j,k}^p - s \mathbbm{1}_{i,k}^{\mathcal{U}}) + \theta_k \left(\frac{d_{s_i}}{v_w}+\frac{d_{s_j}}{v_w}\right) \\
    c_{s_i s_j,k}^{b_1} = \alpha_k \left(w_i^a + w_{s_i s_j,k}^p \right) + \beta_k \left(\frac{d_{s_i}}{v_a} +  \frac{l_{s_i s_j,k}^{p}}{v_p}\right) + \gamma_k \left(b + r_i^a d_{s_i} + r^p l_{s_i s_j,k}^p- s \mathbbm{1}_{i,k}^{\mathcal{U}} \right) + \theta_k \frac{d_{s_j}}{v_w} \\
    c_{s_i s_j,k}^{b_2} = \alpha_k \left( w_{s_i s_j,k}^p + w_j^a \right) + \beta_k \left( \frac{l_{s_i s_j,k}^p}{v_p} + \frac{d_{s_j}}{v_a}\right) + \gamma_k \left( r^p l_{s_i s_j,k}^p - s \mathbbm{1}_{i,k}^{\mathcal{U}} + b + r_j^a d_{s_j} \right) + \theta_k \frac{d_{s_i}}{v_w} \\
    c_{s_i s_j,k}^{b_3} = \alpha_k \left(w_i^a + w_{s_i s_j,k}^p + w_j^a\right) + \beta_k \left(\frac{d_{s_i}}{v_a} +  \frac{l_{s_i s_j,k}^p}{v_p} + \frac{d_{s_j}}{v_a}\right) + \gamma_k \left(b + r_i^a d_{s_i} + r^p l_{s_i s_j,k}^p - s \mathbbm{1}_{i,k}^{\mathcal{U}} + b + r_j^a d_{s_j} \right) \\
    c_{ij,k}^t = -\frac{1}{\eta} \log \sum_{s_i \in \mathcal{V}_t^i} \sum_{s_j \in \mathcal{V}_t^j} \exp{\left(-\eta c_{s_i s_j,k}^t\right)}, \quad t \in \{p, b_1, b_2, b_3\}
\end{subnumcases}
where $\mathbbm{1}_{i,k}^{\mathcal{U}}$ is an indicator function that equals 1 if $i\in \mathcal{U}$ and $k=1$ and 0 otherwise. The generalized costs (\ref{generalized_costs_subsidy}) under the subsidy differ from the generalized costs without regulations in the component of trip fare: low-income people enjoy a reduced fare of $s$ for transit trips (either direct transit trips or bundled trips with AMoD) starting from the underserved/remote area. 

Generally, subsidies for public transport are raised by imposing taxes on the private sector. For instance, New York City began to impose a congestion charge of \$2.75 for solo ride-hailing trips and \$0.75 for shared trips in specific areas in Manhattan in January 2019. The revenue from the congestion surcharge is collected by the Metropolitan Transportation Authority (MTA) and dedicated to improving the subway system in NYC \cite{lehe2021taxation}. In January 2020, the San Francisco County Transportation Authority (SFCTA) also imposed an ad valorem tax of 3.25\% for single TNC rides and 1.5\% for shared rides in San Francisco. The collected taxes are used to subsidize public transit and mitigate traffic congestion \cite{lehe2021taxation}. We consider a regulatory agency that imposes a per-ride tax $\tau$ of AMoD services on the TNC platform, by which the regulatory agency collects taxation to subsidize transit services for low-income people in the underserved/remote area $\mathcal{U}$. The regulatory agency imposes a revenue-neutral tax such that the tax revenue collected from AMoD services offsets the expense of subsidies on transit services. This leads to the following financial equilibrium condition for the regulatory agency:
\begin{equation} \label{economic_balance_subsidy}
    \tau \sum_{i=1}^M \sum_{j=1}^M \sum_{k=1}^K \left(\lambda_{ij,k}^a + \lambda_{ij,k}^{b_1} + \lambda_{ij,k}^{b_2} + 2\lambda_{ij,k}^{b_3} \right) = s \sum_{i=1}^M \sum_{j=1}^M \sum_{k=1}^K \mathbbm{1}_{i,k}^{\mathcal{U}} \left(\lambda_{ij,k}^p + \lambda_{ij,k}^{b_1} + \lambda_{ij,k}^{b_2} + \lambda_{ij,k}^{b_3} \right) ,
\end{equation}
The tax is levied on the TNC platform. Therefore, the profit of the TNC platform under the subsidy/tax (denoted as $\pi_s^a(\xi^a,\xi^p)$) should be reduced to:
\begin{equation} \label{AMoD_profit_subsidy}
\begin{split}
    \pi_s^a(\xi^a,\xi^p) = &\sum_{i=1}^M \sum_{j=1}^M \sum_{k=1}^K \lambda_{ij,k}^a \left(b+r_i^a l_{ij}^a\right) + \lambda_{ij,k}^{b_1}\left(b+r_i^a d_{i,k}^{b_1}\right) + \lambda_{ij,k}^{b_2} \left(b+r_j^a d_{j,k}^{b_2}\right) + \lambda_{ij,k}^{b_3} \left(2b + r_i^a d_{i,k}^{b_3} + r_j^a d_{j,k}^{b_3}\right) \\ & - N C_{av} - \tau \sum_{i=1}^M \sum_{j=1}^M \sum_{k=1}^K \left(\lambda_{ij,k}^a + \lambda_{ij,k}^{b_1} + \lambda_{ij,k}^{b_2} + 2\lambda_{ij,k}^{b_3} \right) .
\end{split}
\end{equation}
Overall, the regulatory agency controls the subsidy $s$ for transit services and the tax $\tau$ on AMoD services to improve transport equity while maintaining the financial equilibrium condition (\ref{economic_balance_subsidy}). In response to the subsidy $s$ and tax $\tau$ set by the regulatory agency, the TNC platform deploys its operational strategy to maximize its profit (\ref{AMoD_profit_subsidy}) subject to the modified generalized travel costs (\ref{generalized_costs_subsidy}) and other unchanged equilibrium conditions (\ref{logit_demand_function}) and (\ref{distance_AMoD_b})-(\ref{vehicle_hour_conservation}); the public transit agency determines the operational strategy to maximize the transit ridership subject to the modified generalized travel costs (\ref{generalized_costs_subsidy}) and other unchanged equilibrium constraints (\ref{logit_demand_function}) and (\ref{frequency_consistency})-(\ref{profit_constraint}). Given the subsidy $s$ and tax $\tau$, the profit maximization problem of the TNC platform and the ridership maximization problem of the transit agency constitute a game problem. Note that the subsidy $s$ and tax $\tau$ do not alter the structure of the profit maximization problem of the TNC and the ridership maximization problem for public transit. We can still use Algorithm \ref{algorithm1} to compute the Nash equilibrium and perform an ex-post evaluation on the obtained equilibrium. Since the Nash equilibrium depends on both $s$ and $\tau$, the total taxes (left-hand side of Equation (\ref{economic_balance_subsidy})) and the total subsidies (right-hand side of Equation (\ref{economic_balance_subsidy})) at equilibrium should be a function of $(s,\tau)$. For a specific subsidy level $s$, we use the bisection method to find the corresponding tax $\tau$ that satisfies the financial equilibrium condition (\ref{economic_balance_subsidy}).

To investigate the impacts of subsidies for transit services, we focus on the nominal case with $C_{av}=20\$/\text{hour}$ and keep other model parameters unchanged as in Section \ref{case study}. For distinct values of $s$, we find the corresponding tax $\tau$ that satisfies the financial equilibrium (\ref{economic_balance_subsidy}) for the regulatory agency, solve the game problem under distinct $(s,\tau)$, and demonstrate how the subsidy/tax affects the market outcomes and transport equity in Figure \ref{fig:regulatory_agency_subsidy}-\ref{fig:equity_subsidy}.

\begin{figure}[h!]
     \centering
     \begin{subfigure}[b]{0.48\textwidth}
         \centering
         \includegraphics[width=\textwidth]{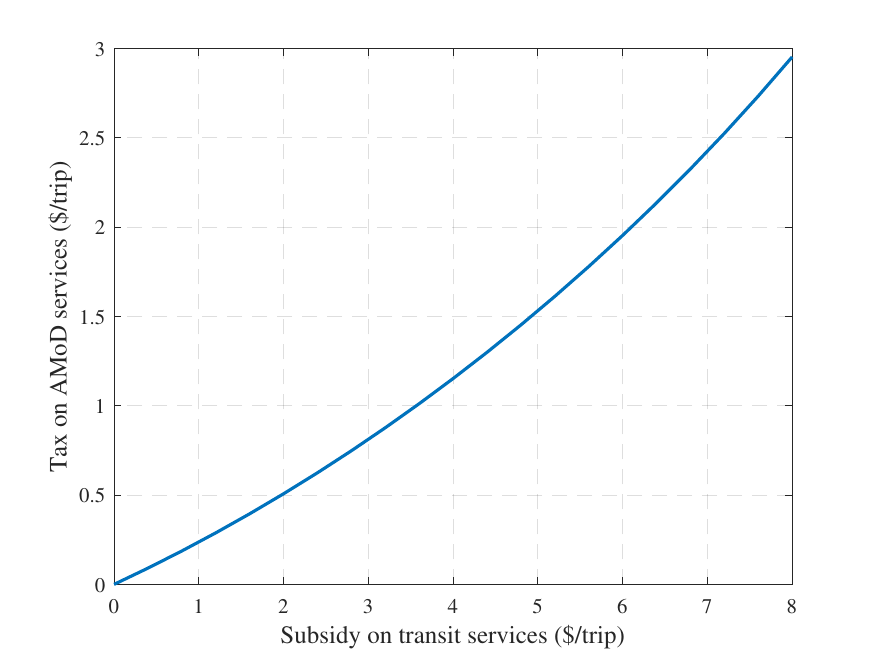}
         \caption{Tax on AMoD services under distinct values of $s$.}
         \label{fig:tax_subsidy}
     \end{subfigure}
     \hfill
     \begin{subfigure}[b]{0.48\textwidth}
         \centering
         \includegraphics[width=\textwidth]{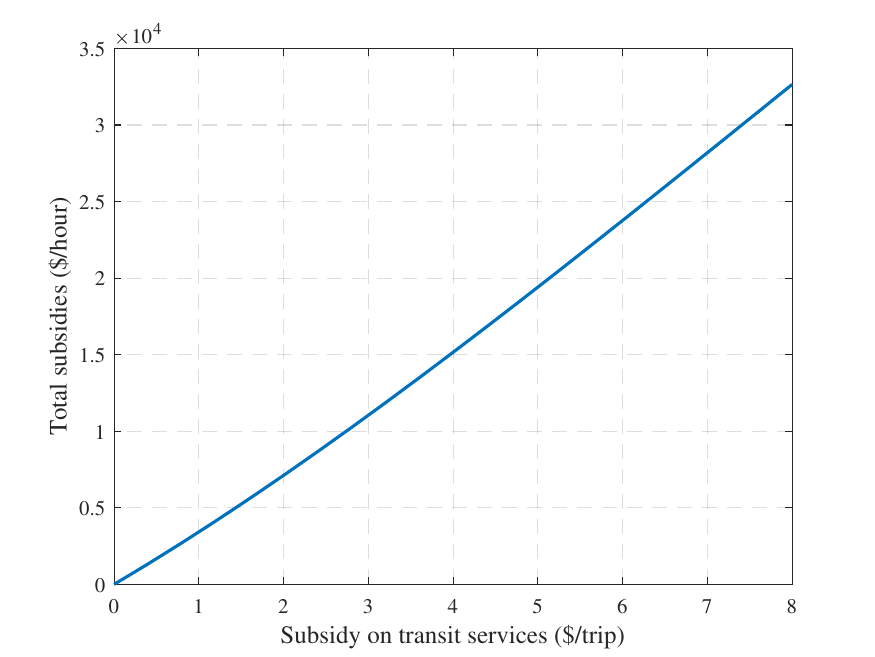}
         \caption{Total subsidies under distinct values of $s$}
         \label{fig:Total_subsidy_subsidy}
     \end{subfigure}
    \caption{Tax on AMoD services and total subsidies under distinct $s$.}
    \label{fig:regulatory_agency_subsidy}
\end{figure}

\begin{figure}[h!]
     \centering
     \begin{subfigure}[b]{0.48\textwidth}
         \centering
         \includegraphics[width=\textwidth]{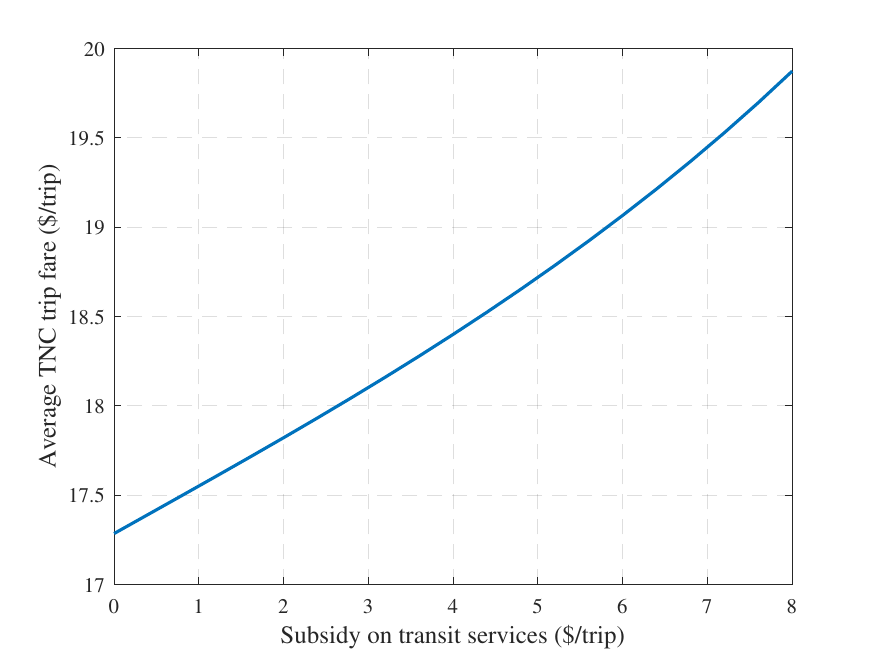}
         \caption{{Average trip fare of the ride-hailing services under distinct values of subsidy $s$.}}
         \label{fig:base_fare_subsidy}
     \end{subfigure}
     \hfill
     \begin{subfigure}[b]{0.48\textwidth}
         \centering
         \includegraphics[width=\textwidth]{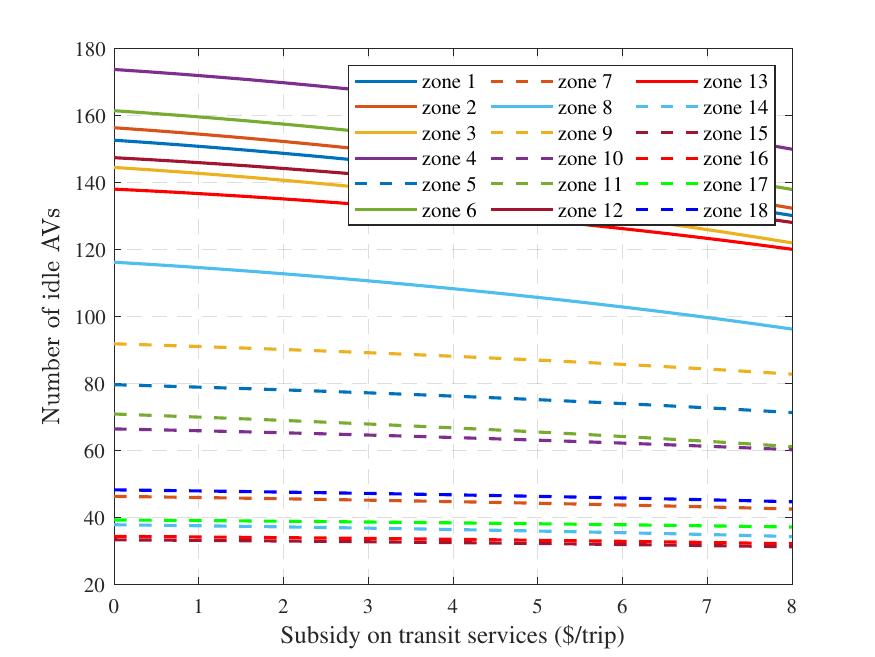}
         \caption{Number of idle vehicles in different zones of the transport network under distinct values of subsidy $s$.}
         \label{fig:number_idle_vehicle_subsidy}
     \end{subfigure}
    \begin{subfigure}[b]{0.48\textwidth}
         \centering
         \includegraphics[width=\textwidth]{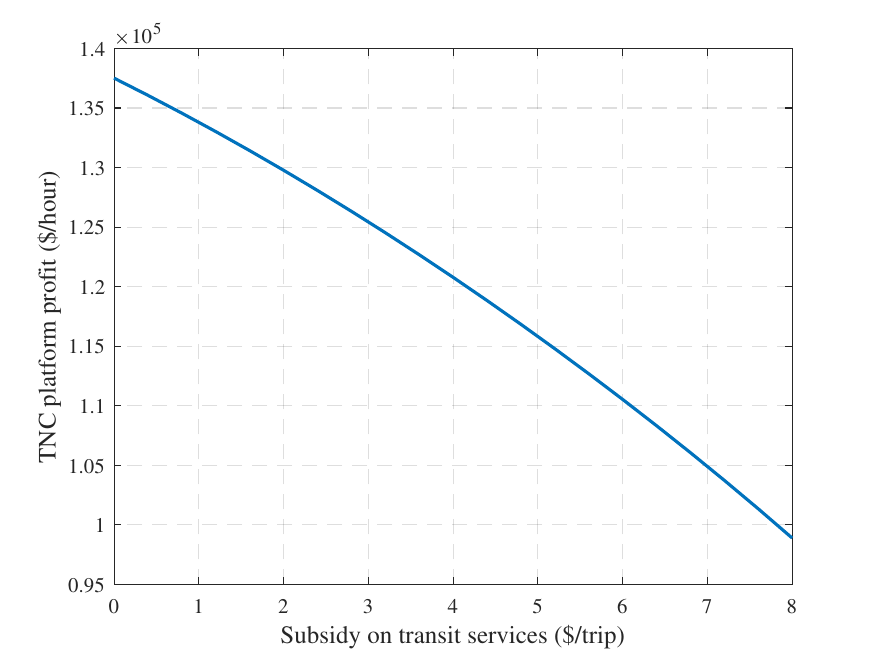}
         \caption{Hourly profit of the TNC platform under distinct values of subsidy $s$.}
         \label{fig:TNC_profit_subsidy}
     \end{subfigure}
     \hfill
     \begin{subfigure}[b]{0.48\textwidth}
         \centering
         \includegraphics[width=\textwidth]{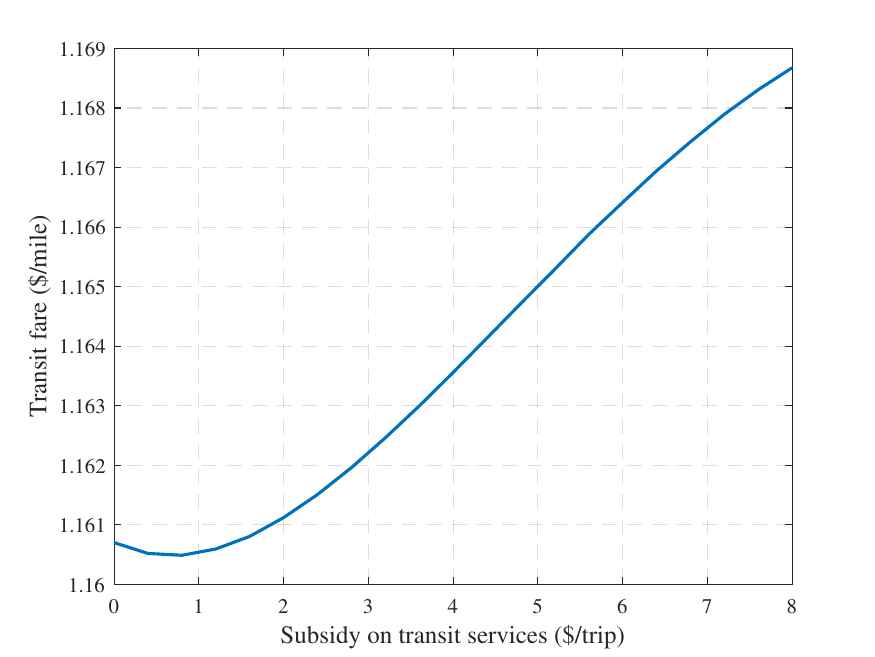}
         \caption{Trip fare (per-distance) of transit services under different values of  subsidy $s$.}
         \label{fig:transit_fare_subsidy}
     \end{subfigure}
     \begin{subfigure}[b]{0.48\textwidth}
     \centering
     \includegraphics[width=\textwidth]{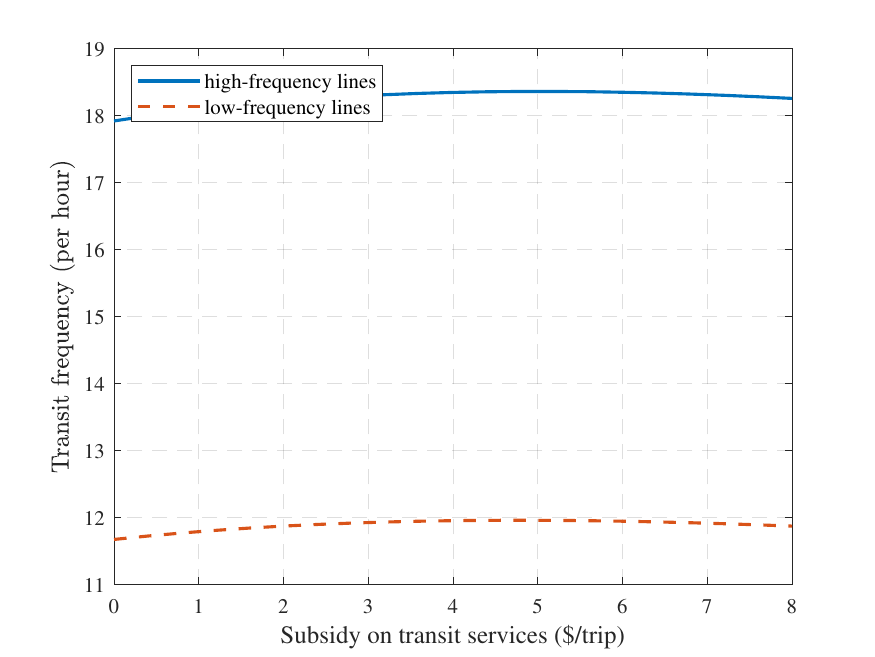}
     \caption{Frequencies of different transit lines under distinct values of subsidy $s$.}
     \label{fig:transit_frequency_subsidy}
    \end{subfigure}
    \hfill
    \begin{subfigure}[b]{0.48\textwidth}
         \centering
         \includegraphics[width=\textwidth]{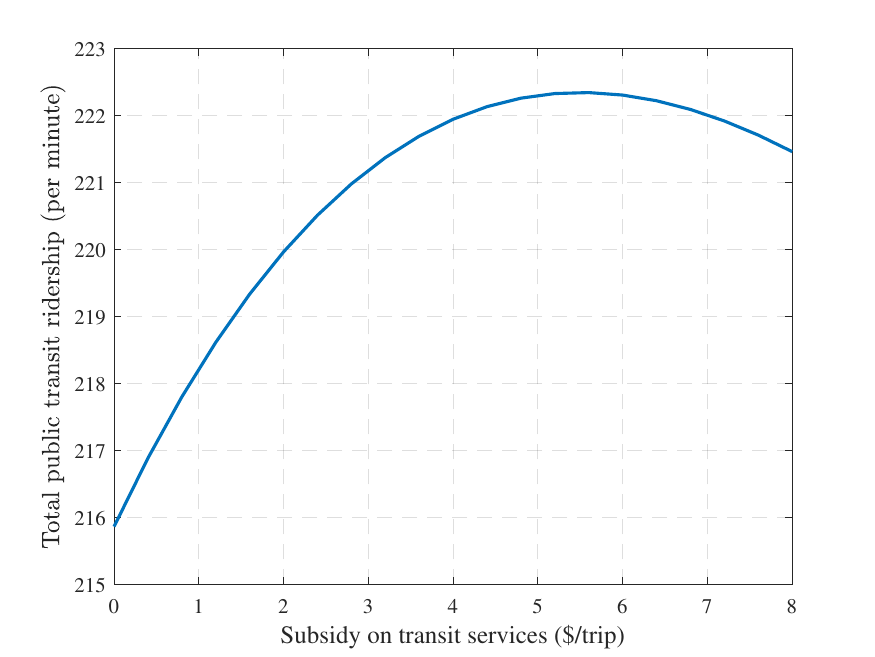}
         \caption{Total public transit ridership under distinct values of subsidy $s$.}
         \label{fig:transit_ridership_subsidy}
    \end{subfigure}
    \caption{Operational decisions of the TNC platform and the public transit agency under distinct $s$.}
    \label{fig:TNC_subsidy}
\end{figure}

\begin{figure}[h!]
     \centering
     \begin{subfigure}[b]{0.48\textwidth}
         \centering
         \includegraphics[width=\textwidth]{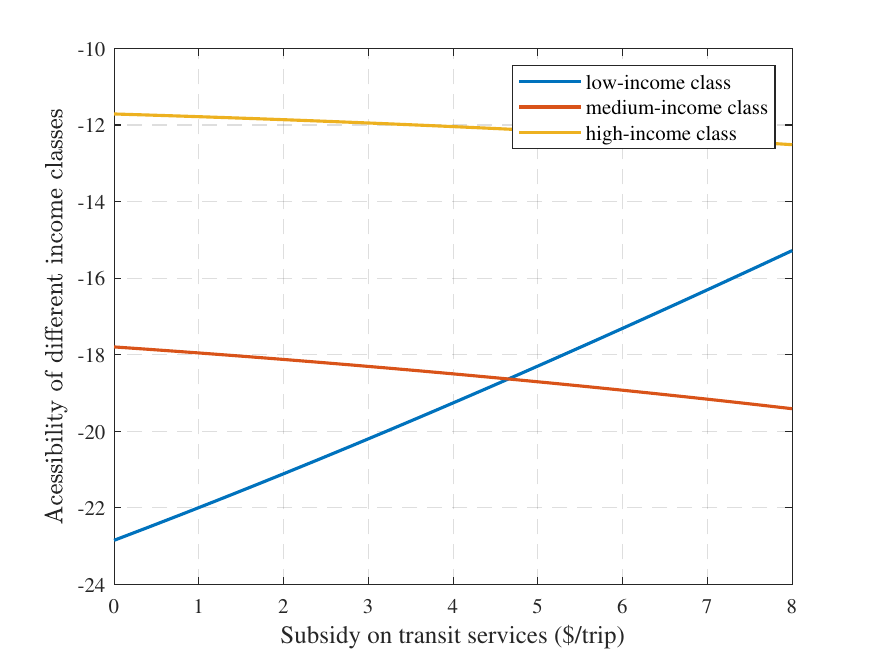}
         \caption{Average accessibility of different income groups under distinct values of $s$.}
         \label{fig:accessibility_groups_subsidy}
     \end{subfigure}
     \hfill
     \begin{subfigure}[b]{0.48\textwidth}
         \centering
         \includegraphics[width=\textwidth]{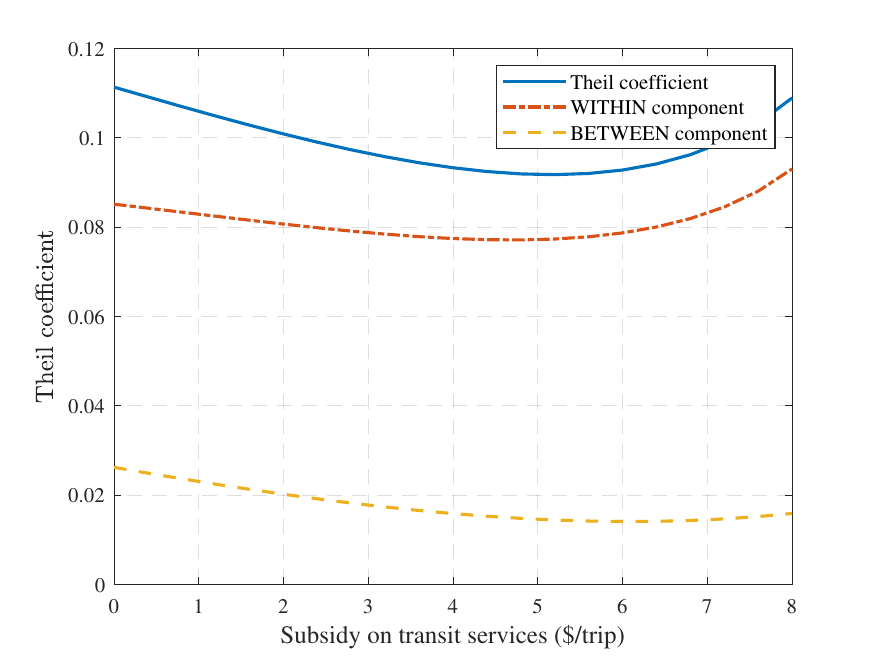}
         \caption{Theil coefficient, WITHIN component, and BETWEEN component under distinct values of $s$}
         \label{fig:Theil_subsidy}
     \end{subfigure}
    \caption{Accessibility and transport equity under distinct $s$.}
    \label{fig:equity_subsidy}
\end{figure}

Figure \ref{fig:tax_subsidy} shows the corresponding tax on AMoD services under distinct $s$ set by the regulatory agency and Figure \ref{fig:Total_subsidy_subsidy} presents the total subsidies on transit services as a function of $s$. As the subsidy increases, the total expense of subsidizing transit services increases (Figure \ref{fig:Total_subsidy_subsidy}) and thereby the regulatory agency needs to impose a higher tax on AMoD services to maintain the financial equilibrium (Figure \ref{fig:tax_subsidy}). Besides, a small tax on AMoD services enables the regulatory agency to fund a relatively high level of subsidy on transit services. This is because the subsidy on transit services specifically target low-income households in underserved/remote areas.

Figure \ref{fig:base_fare_subsidy}-\ref{fig:TNC_profit_subsidy} shows TNC's operational strategy and the corresponding market outcomes under distinct $s$. As the subsidy increases, the TNC raises the trip fare (Figure \ref{fig:base_fare_subsidy}), reduces the number of idle vehicles in distinct zones (Figure \ref{fig:number_idle_vehicle_subsidy}), and earns a reduced profit under a higher subsidy level (Figure \ref{fig:TNC_profit_subsidy}). These results are intuitive since a higher subsidy on transit services demands a higher tax on AMoD services, which increases the tax burden on the TNC platform and promotes it to increase the trip fare and reduce the vehicle supply to guarantee profitability. Interestingly, the reduction in the number of idle AVs in remote zones (dashed lines in Figure \ref{fig:number_idle_vehicle_subsidy}) is insignificant compared to that in core zones (solid lines in Figure \ref{fig:number_idle_vehicle_subsidy}). Although subsidizing transit services imposes a tax on AMoD services which curbs the TNC platform's ride-hailing business, it reduces the monetary cost of bundled services that combine AMoD and public transit for low-income passengers in underserved/remote zones and helps to promote AMoD services in these areas. To mitigate the negative impacts of imposed tax on AMoD services and leverage the benefits of subsidy on bundled services, the TNC platform primarily reduces the number of idle AVs in core areas while maintaining the vehicle supply in underserved/remote areas. 

Figure \ref{fig:transit_fare_subsidy}-\ref{fig:transit_ridership_subsidy} demonstrates the public transit agency's operational decisions and the corresponding outcomes under distinct $s$. 
As the subsidy level improves, the transit agency raises the transit fare (Figure \ref{fig:transit_fare_subsidy}). The service frequencies of distinct transit lines initially increase and then drop with the subsidy (Figure \ref{fig:transit_frequency_subsidy}). The total public transit ridership under different subsidies exhibits two distinct regimes: when $s\leq \$5.6$/trip, a higher subsidy on transit services yields a higher transit ridership; when $s> \$5.6$\$/trip, the transit ridership decreases with the subsidy (Figure \ref{fig:transit_ridership_subsidy}). We argue that this is again attributed to the competitiveness and complementarity between AMoD services and public transit. Note that the total public transit ridership consists of the ridership of direct transit services (mode $p$) and the ridership of bundled services (mode $b$). The regulatory agency subsidizes transit services for low-income households in underserved/remote areas while taxing AMoD services. The subsidy improves transit services and strengthens the competitiveness of public transit (direct transit services, i.e., mode $p$) to AMoD services (direct AMoD services, i.e., mode $a$), while the tax deteriorates AMoD services and suppresses the complementarity between AMoD services and public transit (bundled services, i.e., mode $b_1$, $b_2$, and $b_3$). When the subsidy/tax is relatively low, the increased competitiveness of public transit to AMoD services due to the subsidy is more significant compared to the reduced complementarity caused by the tax. It attracts more passengers, especially low-income passengers in underserved/remote areas, to choose direct transit services, although the ridership of bundled services slightly drops, which leads to an increased total public transit ridership. However, when the subsidy/tax is high, the imposed tax largely worsens AMoD services, and the reduced complementarity between AMoD services and public transit gradually becomes the dominant factor. In this case, the number of passengers, especially medium-income and high-income passengers in core areas, choosing bundled services significantly decreases, and the total transit ridership reduces with the subsidy.

Figure \ref{fig:equity_subsidy} shows the impacts of subsidies/taxes on accessibility and transport equity in the multimodal transportation system. As the subsidy increases, the accessibility of high-income and medium-income classes slightly drops, while the accessibility of low-income passengers significantly increases (Figure \ref{fig:accessibility_groups_subsidy}). The spatial inequality (WITHIN component in Figure \ref{fig:Theil_subsidy}) and social inequality (BETWEEN component in Figure \ref{fig:Theil_subsidy}) first decline and then increase, which leads to an initial decrease and then an increase in overall transport equity (Theil coefficient in Figure \ref{fig:Theil_subsidy}). The equity impacts of subsidies on transit services are natural. The regulatory agency subsidizes transit services and taxes AMoD services at the same time. The subsidy on transit services is devoted to low-income passengers from underserved/remote areas, which reduces the generalized costs of pubic transit and bundled services. Low-income transit-dependent individuals in underserved/remote areas are the direct beneficiaries of the subsidy. When the subsidy is relatively low, the accessibility of low-income passengers in underserved/remote areas increases with the subsidy, which bridges the existing spatial and social inequity gaps at the same time. However, when transit services are over-subsidized (correspondingly, AMoD services are over-taxed), the excessive tax significantly deteriorates AMoD services, especially in core areas, and hurts medium-income and high-income classes. This will enlarge the disparity in accessibility across distinct income classes and geographic zones and exacerbate spatial and social inequity. The regulatory agency should carefully adjust its control knobs to lead to a subsidy/tax level that archives an equitable distribution of accessibility across distinct income classes and geographic regions.

\section{Conclusion} \label{conclusion}

This paper assesses the equity impacts of AVs and investigates regulatory policies that guarantee the benefits of AV deployment reach underserved areas and transportation-disadvantaged groups. A network game-theoretic model is formulated to characterize the intimate interactions among TNC's location-differentiated price, the transit fare, services frequencies of distinct transit lines, waiting times for AMoD services, waiting times for transit services, TNC fleet size, the boarding/alighting station decisions of public transit, the disutility/generalized costs of distinct mobility modes, and the modal share of different income classes in the multimodal transportation system, by which the incentives of multiclass passengers, the TNC, and the public transit agency are captured. An algorithm is developed to compute the Nash equilibrium and conduct an ex-post evaluation of the performance of the obtained Nash equilibrium. Theil coefficient is utilized to quantify both the spatial and social inequality in individuals' accessibility in the multimodal transportation system.

Based on the developed framework, we reveal the spatial and social inequity gaps without regulatory interventions and investigate regulatory policies that improve transport equity. Through numerical study, we showed that although individuals' accessibility improves, spatial and social inequity gaps enlarge as AV technology evolves. The for-profit nature of TNC leads to the geographic concentration of AMoD services and exacerbates spatial inequity. In the meanwhile, the competition between AMoD and public transit disrupts transit services and the benefits of reduced costs for AMoD services are primarily distributed to higher-income individuals who are less transit-dependent, which expands the social inequity gap. The impacts of two regulatory policies are evaluated: (a) a minimum service-level requirement on AMoD services; and (b) a subsidy on transit services by taxing AMoD services. We showed that a minimum service-level requirement incurs a trade-off: as a higher minimum service level is required, the spatial inequity decreases while the social inequity keeps increasing. Therefore, the regulatory agency should evaluate the trade-off between spatial equity and social equity and carefully control the minimum service-level requirement. On the other hand, subsidizing transit services for low-income individuals in underserved areas using taxes on AMoD services can promote the use of public transit and improve spatial and social equity at the same time. However, when transit services are over-subsidized, excessive taxation significantly deteriorates AMoD services, which sacrifices the accessibility of medium-income and high-income classes and enlarges the spatial and social inequity gaps. The regulatory agency should choose the proper subsidy/tax level to reach an equitable distribution of accessibility across distinct income classes and geographic zones.

This paper delivers a comprehensive socioeconomic analysis of AV-enabled mobility future, with a special focus on transport equity. One future extension is extending the model to capture a mixture of AVs and human drivers in a TNC market and consider equity from both passenger and driver sides. Another extension would be incorporating other socio-economic factors (other than income levels) into the study of social equity.

\section*{Acknowledgments}  {This research was supported by Hong Kong Research Grants Council under project 16202922, 26200420, and National Science Foundation of China under project 72201225.}

\bibliographystyle{unsrt}
\bibliography{reference}


\subsection*{\bf{Appendix A: List of notations}}

\begin{longtable}{r p{15cm}}
\caption{The list of notations.} \label{tab:notation_table} \\
\hline
\multicolumn{2}{c}{\bf{Decision Variables}}\\
\hline
$b$ & Base fare (\$/trip) of AMoD services \\
$r_i^a$ & Per-distance fare (\$/mile) of AMoD services starting from zone $i$ \\
$N_i^I$ & Number of idle AVs deployed in zone $i$ \\
$N$ & Total number of AVs \\  
$\xi^a$ & Operational strategy of the TNC platform \\
$r^p$ & Per-distance fare of transit services \\
$f_l$ & Service frequency of transit line $l$ \\
$\xi^p$ & Operational strategy of the public transit agency \\
\hline
\multicolumn{2}{c}{\bf{Endogenous Variables}}\\
\hline
$c_{s_i s_j,k}^t$ & Expected generalized cost of passengers in income class $k$ choosing mode $t$ from zone $i$ to zone $j$ with boarding station $s_i$ and alighting station $s_j$ \\
$c_{ij,k}^t$ & Expected generalized cost of passengers in income class $k$ choosing mode $t$ from zone $i$ to zone $j$ \\
$\lambda_{ij,k}^t$ & Arrival rate (per minute) of passengers in income class $k$ choosing mode $t$ from zone $i$ to zone $j$ \\
$m_i^{c-v}$ & Matching rate between passengers and vehicles in zone $i$ \\
$N_i^a$ & Number of waiting passengers in zone $i$ \\
$\lambda_i$ & Average arrival rate of waiting passengers/idle vehicles in zone $i$ \\
$w_i^v$ & Average vehicle idle time in zone $i$ \\
$w_i^a$ & Average waiting time (minute) for AMoD services in zone $i$ \\
$w_{s_i s_j,k}^p$ & Average waiting time (minute) of passengers in income class $k$ for transit services from boarding station $s_i$ to alighting station $s_j$ \\
$l_{s_i s_j,k}^p$ & Average trip distance of passengers in income class $k$ by public transit from boarding station $s_i$ to alighting station $s_j$ \\
$l_{ij,k}^t$ & Average trip distance by transit from zone $i$ to zone $j$ for income class $k$ when choosing mode $t$ \\
$d_{i,k}^{b_1}$ & Average access distance by AMoD in origin zone $i$ when choosing mode $b_1$ \\
$d_{j,k}^{b_2}$ & Average egress distance by AMoD in destination zone $i$ when choosing mode $b_2$ \\
$d_{i,k}^{b_3}$ & Average access distance by AMoD in origin zone $i$ when choosing mode $b_3$ \\
$d_{j,k}^{b_3}$ & Average egress distance by AMoD in destination zone $j$ when choosing mode $b_3$ \\
$\mathbb{P}_{s_i s_j,k}^t$ & Probability of selecting boarding station $s_i$ and alighting station $s_j$ for passengers in income class $k$ by mode $t$ from zone $i$ to zone $j$ \\
$\pi^a$ & Profit (per hour) of the TNC platform \\
$\pi^p$ & Total ridership (per minute) of public transit \\
$A_{ij,k}$ & Accessibility of passengers in income class $k$ from zone $i$ to zone $j$ \\
$A_{i,k}$ & Average accessibility of passengers in income class $k$ starting from zone $i$ \\
$A_k$ & Average accessibility of passengers in income class $k$ \\
$\overline{A}$ & Average accessibility of all passengers \\
$T$ & Theil coefficient of accessibility distribution in the multimodal transportation system \\
\hline
\multicolumn{2}{c}{\bf{Auxiliary Variables}}\\
\hline
$\mathcal{L}_s$ & Set of attractive lines at station $s$ \\
$w_s^p$ & Expected waiting time at station $s$ in the transit network \\
$\mathbb{P}_s^l$ & Probability of passengers boarding line $l$ at station $s$ \\
$f_a$ & Frequency of transit link $a$ \\
$g_s$ & Net outflow of passengers at station $s$ \\
$v_{ak}$ & Flow of passengers in income class $k$ on transit link $a$ \\
$w_{sk}^p$ & Expected waiting time of passengers in income class $k$ at station $s$ \\
$\mathcal{E}_t^{s+}$ & Set of outgoing transit links at station $s$ \\
$\mathcal{E}_t^{s-}$ & Set of incoming transit links at station $s$ \\
$b_j$ & Base fare (\$/trip) of AMoD trips destined to zone $j$ \\
$r_{ij}^a$ & Ride fare (\$/mile) of AMoD trips from zone $i$ to zone $j$ \\
$N_{ij}^I$ & Number of idle AVs deployed for zone $i$ to zone $j$ \\
$\xi_j^a$ & Operational strategy of the TNC platform exclusively for destination zone $j$ \\
$\xi_{ij}^a$ & Operational strategy of the TNC platform exclusively for OD pair $ij$ \\
$\pi_j^a$ & Profit of the TNC platform from destination zone $j$ \\
$\pi_{\mathcal{V}_x}^a$ & Profit of the TNC platform from the set of destination zones $\mathcal{V}_x$ \\
$\pi_{ij}^a$ & Profit of the TNC platform from OD pair $ij$ \\
$\pi_{i,\mathcal{V}_x}^a$ & Profit of the TNC platform from serving trips from zone $i$ to destination zones $\mathcal{V}_x$ \\
$\overline{\mathcal{V}}$ & Partition of the set of zones $\mathcal{V}$ \\
$\bar{\pi}^a$ & Upper bound (\$/hour) of TNC platform profit \\
$\underline{\pi}^a$ & Lower bound (\$/hour) of TNC platform profit \\
$\epsilon$ & Value of $\epsilon$ in calculating $\epsilon$-Nash equilibrium \\
\hline
\multicolumn{2}{c}{\bf{Exogenous Parameters}}\\
\hline
$M$ & Number of zones in the city \\
$K$ & Number of income classes among populations \\
$L$ & Number of transit lines in the transit network \\
$A_i$ & Scaling parameter in the AMoD waiting time function \\
$\mathcal{T}$ & Set of modes in the multimodal transportation system \\
$\mathcal{V}$ & Set of indexes of all zones \\
$\mathcal{L}$ & Set of transit lines \\
$\mathcal{L}_H$ & Set of high-frequency transit lines \\
$\mathcal{L}_L$ & Set of low-frequency transit lines \\
$\mathcal{V}_t$ & Set of transit stations \\
$\mathcal{E}_t$ & Set of transit links \\
$\mathcal{U}$ & Set of indexes of zones in underserved area \\
$a_1,a_2$ & Elasticity parameters in AMoD waiting time function \\
$c_{ij,k}^o$ & Expected generalized cost of outside option for passengers in income class $k$ from zone $i$ to zone $j$ \\
$\lambda_{ij,k}^0$ & Arrival rate (per minute) of potential passengers in income class $k$ from zone $i$ to zone $j$ \\
$\eta$ & Parameter of the boarding/alighting station choice logit model \\
$\mu$ & Parameter of the passenger demand logit model \\
$\alpha_k$ & Cost parameter of waiting time for passengers in income class $k$ \\
$\beta_k$ & Cost parameter of in-vehicle time for passengers in income class $k$ \\
$\gamma_k$ & Cost parameter of trips fares for passengers in income class $k$ \\
$\theta_k$ & Cost parameter of walking time for travelers in group $k$ \\
$l_{ij}^a$ & Average distance (mile) of AMoD trips from zone $i$ to zone $j$ \\
$l_a^p$ & Length of transit link $a$ \\
$d_{s_i}$ & Distance from the zone centroid $i$ to station $s_i$ \\
$v_a$ & Average speed (mile/hour) of TNC AVs \\
$v_p$ & Average operating speed (mile/hour) of the transit network \\
$v_w$ & Average walking speed \\
$\pi_0$ & Budget (\$/hour) on the operation of public transit\\
$C_{av}$ & Average operating cost (\$/hour) of an TNC AV \\
$C_l$ & Average per-vehicle operating cost (\$/hour) of transit line $l$ \\
$w_{\text{max}}^a$ & Maximum allowable waiting time (minute) for AMoD services \\
$s$ & Subsidy (\$/trip) on transit services for low-income passengers from underserved areas \\
$\tau$ & Tax (\$/trip) on AMoD services \\ 
\hline
\end{longtable}

\subsection*{\bf{Appendix B: Proof of Proposition \ref{proposition_upper_lower_bound}}}

We prove the result of Proposition \ref{proposition_upper_lower_bound} in the following two steps:
(1) $\sum_{x=1}^m \pi_{\mathcal{V}_x}^{a^*}$ is the optimal value of the relaxed problem (\ref{Incentives_AMoD_relaxed}) given $\xi^{p^*}$. Since (\ref{Incentives_AMoD_relaxed}) is a relaxed problem of (\ref{Incentives_AMoD_equivalent}), and $\xi^{a^*}$ is the optimal solution to (\ref{Incentives_AMoD_equivalent}) given $\xi^{p^*}$, then we have $\sum_{x=1}^m \pi_{\mathcal{V}_x}^{a^*} \geq \sum_{j=1}^M \pi_j^a \left(\xi^{a^*}\right) = \pi^a\left(\xi^{a^*},\xi^{p^*}\right)$.
(2) Since $\xi_{j,j\in\mathcal{V}_x}^{a^*}, \forall x=1,\dots,m$, is a feasible solution to (\ref{Incentives_AMoD_equivalent}), a non-negative linear combination of $\xi_{j,j\in\mathcal{V}_x}^{a^*},x=1,\dots,m$ is also a feasible solution to (\ref{Incentives_AMoD_equivalent}). Therefore, $\sum_{j=1}^M \pi_j^a \left(\frac{\sum_{x=1}^m \pi_{\mathcal{V}_x}^{a^*} \xi_{j,j \in \mathcal{V}_x}^{a^*}}{\sum_{x=1}^m \pi_{\mathcal{V}_x}^{a^*}}\right) \leq \sum_{j=1}^M \pi_j^a \left(\xi^{a^*}\right) = \pi^a\left(\xi^{a^*},\xi^{p^*}\right) $. This completes the proof.

\newpage

\subsection*{\bf{Appendix C: Algorithm \ref{algorithm1}}}

\begin{algorithm}
\caption{Best response algorithm and the ex-post evaluation on Nash equilibrium} \label{algorithm1}
\begin{algorithmic}[1]
\REQUIRE initial guess of operational strategy $\xi_{(0)}^a=\left(b,\mathbf{r^a},\mathbf{N^I}\right)$ and $\xi_{(0)}^p=\left(r^p,\mathbf{f}\right)$, the convergence tolerance $\sigma$, the maximum number of iterations $\hat{n}$, a partition $\overline{\mathcal{V}}=\left\{\mathcal{V}_1,\mathcal{V}_2,\dots,\mathcal{V}_m\right\}$ of $\mathcal{V}$. 
\STATE Setup stopping criterion: $||\xi_{(n)}^a-\xi_{(n-1)}^a||_2 \leq \sigma$ and $||\xi_{(n)}^p-\xi_{(n-1)}^p||_2 \leq \sigma$
\FOR {$n=1,\dots,\hat{n}$} 
\STATE Fix $\xi^p=\xi_{(n-1)}^p$, solve (\ref{Incentives_AMoD}) using interior-point algorithm, and obtain the optimal solution $\xi_{(n)}^a$.
\STATE Fix $\xi^a=\xi_{(n)}^a$, solve (\ref{Incentives_PT}) by grid search methods, and obtain the globally optimal solution $\xi_{(n)}^p$.
\IF {stopping criterion satisfied}
\STATE Obtain the candidate equilibrium strategy $\xi^{a^*}=\xi_{(n)}^a$ and $\xi^{p^*}=\xi_{(n)}^p$.
\BREAK
\ENDIF
\ENDFOR
\FOR {$x=1,\dots,m$}
\STATE Solve (\ref{Incentives_AMoD_relaxed_subproblem}) given $\xi^p=\xi^{p^*}$ using primal decomposition algorithm in Section \ref{solution method}.
\STATE Obtain the optimal solution $\xi_{j,j\in\mathcal{V}_x}^{a^*}$ and the corresponding optimal value $\pi_{\mathcal{V}_x}^{a^*}$.
\ENDFOR
\STATE Calculate the upper bound $\overline{\pi}^a=\sum_{x=1}^m \pi_{\mathcal{V}_x}^{a^*}$ and the lower bound $\underline{\pi}^a=\pi^a \left(\frac{\sum_{x=1}^m \pi_{\mathcal{V}_x}^{a^*} \xi_{j,j \in \mathcal{V}_x}^{a^*}}{\sum_{x=1}^m \pi_{\mathcal{V}_x}^{a^*}}, \xi^{p^*} \right)$.
\STATE Compute the value $\epsilon=\bar{\pi}^a-\pi^a(\xi^{a^*},\xi^{p^*})$.
\ENSURE the $\epsilon$-Nash equilibrium strategy $\xi^*=(\xi^{a^*},\xi^{p^*})$, the platform profit $\pi^a(\xi^{a^*},\xi^{p^*})$ and its upper bound $\overline{\pi}^a$ and lower bound $\underline{\pi}^a$, the transit ridership $\pi^p(\xi^{a^*},\xi^{p^*})$, the value of $\epsilon$.
\end{algorithmic}
\end{algorithm}

\newpage

\subsection*{\bf{Appendix D: Correspondence between Zone Number and Postal Code}}

The following Table \ref{mytable} describes the correspondences between postal code and zone number.
    \begin{table}[!h]
    \setlength{\abovecaptionskip}{0.3cm}
    \setlength{\belowcaptionskip}{0.5cm}
        \centering
        \caption{The correspondence between zone number and postal code}
        \begin{tabular}{c|c}
            1 & 94111\tablefootnote{We merge zip code zones 94104, 94105, and 94111 as an aggregated zone with postal code 94111.} \\
            2 & 94103 \\
            3 & 94109 \\
            4 & 94115\tablefootnote{We merge zip code zones 94115 and 94123 as an aggregated zone with postal code 94115.} \\
            5 & 94118 \\
            6 & 94133\tablefootnote{We merge zip code zones 94108 and 94133 as an aggregated zone with postal code 94133.} \\
            7 & 94121 \\
            8 & 94102 \\
            9 & 94117 \\
            10 & 94122 \\
            11 & 94114 \\
            12 & 94107 \\
            13 & 94110 \\
            14 & 94131 \\
            15 & 94116 \\
            16 & 94124 \\
            17 & 94132 \\
            18 & 94112
        \end{tabular}
    \label{mytable}
    \end{table}

\subsection*{\bf{Appendix E: Impacts of economies of scales of AMoD services}}

To illustrate the impacts of economies of scale for AMoD services, we conduct a sensitivity analysis on the elasticity parameter of the passenger-vehicle Cobb-Douglas matching function (\ref{cobb_douglas_function}). Note that in the matching function (\ref{cobb_douglas_function}), $a_1$ and $a_2$ are the elasticity parameters which reflect the economies of scale of AMoD services. A larger $a_1$ and/or $a_2$ indicates a more efficient matching between waiting passengers and idle vehicles, and thereby a stronger economy of scale. We fix $a_1=1$ and vary the values of $a_2$ to investigate the impacts of scale economies. This leads to the following general passenger waiting function for AMoD services:
\begin{equation} \label{waiting_time_AMoD_general}
    w_i^a = \frac{1}{A_i \left(N_i^I\right)^{a_2}} .
\end{equation}
Note that the waiting time function (\ref{waiting_time_AMoD}) corresponds to a special case of function (\ref{waiting_time_AMoD_general}) with $a_2=0.5$. Figure \ref{fig:economy_scale} shows the spatial distribution of idle AVs and the equity results under distinct values of $a_2$. The impacts of economies of scale are summarized below:
\begin{itemize}
    \item As the economies of scale enhances (as $a_2$ increases), the vehicle supply reduces but the spatial disparity in AMoD services enlarges (Figure \ref{fig:idle_AV_scale_04}-\ref{fig:idle_AV_scale_06}). When TNC services exhibit a stronger economy of scale, the TNC platform has a higher utilization of AVs, which enables it to reduce the supply of AVs for cost reduction. On the other hand, it floors its idle vehicles into high-demand areas to leverage the benefits of enhanced economies of scale to the largest extent.
    \item An enhanced economies of scale increases spatial and social inequity simultaneously and leads to exacerbated transport equity (Figure \ref{fig:WITHIN_scale}-\ref{fig:T_scale}). A stronger economy of scale of AMoD services exacerbates the geographic concentration of idle AVs, which contributes to spatial inequity. Furthermore, the spatial and socioeconomic characteristics of a city admits strong correlations: high-demand areas are usually well-served by transit and has a large population of high-income individuals. Therefore, the spatial disparity in AMoD services also enlarges social inequity gaps.
\end{itemize}

\begin{figure}[h!]
     \centering
     \begin{subfigure}[b]{0.32\textwidth}
         \centering
         \includegraphics[width=\textwidth]{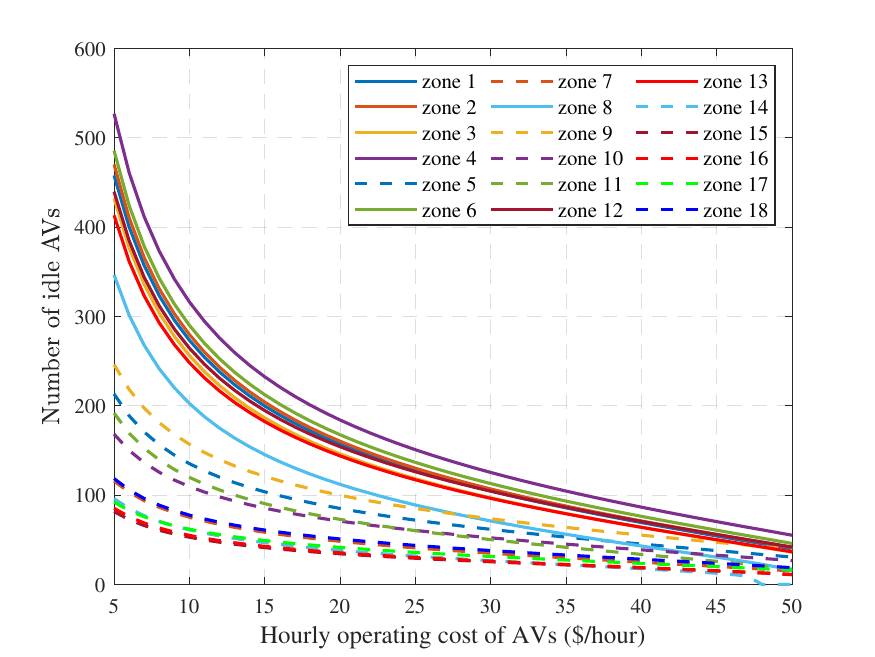}
         \caption{{Spatial distribution of idle AVs under $a_2=0.4$.}}
         \label{fig:idle_AV_scale_04}
     \end{subfigure}
     \hfill
     \begin{subfigure}[b]{0.32\textwidth}
         \centering
         \includegraphics[width=\textwidth]{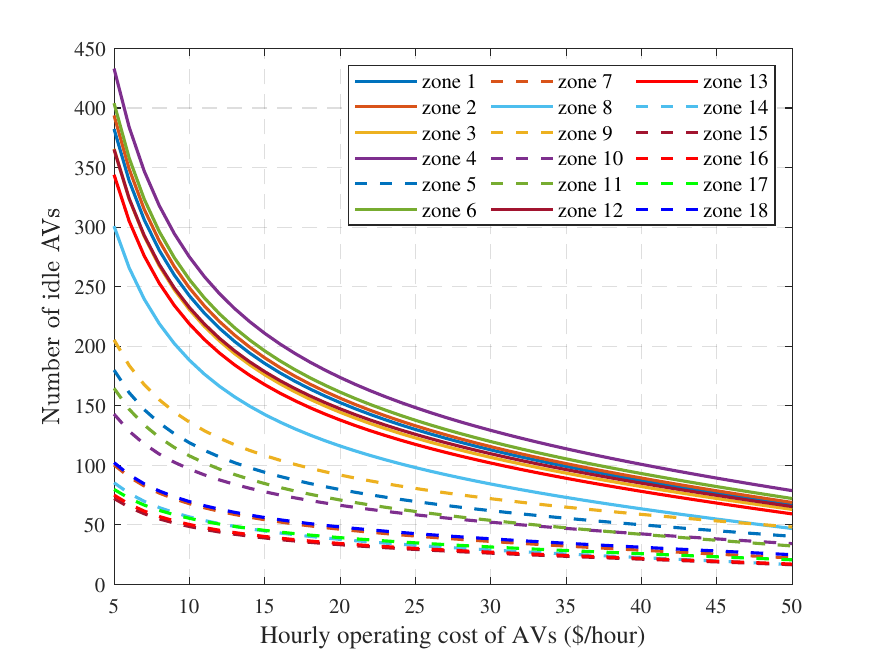}
         \caption{Spatial distribution of idle AVs under $a_2=0.5$.}
         \label{fig:idle_AV_scale_05}
     \end{subfigure}
     \hfill
    \begin{subfigure}[b]{0.32\textwidth}
         \centering
         \includegraphics[width=\textwidth]{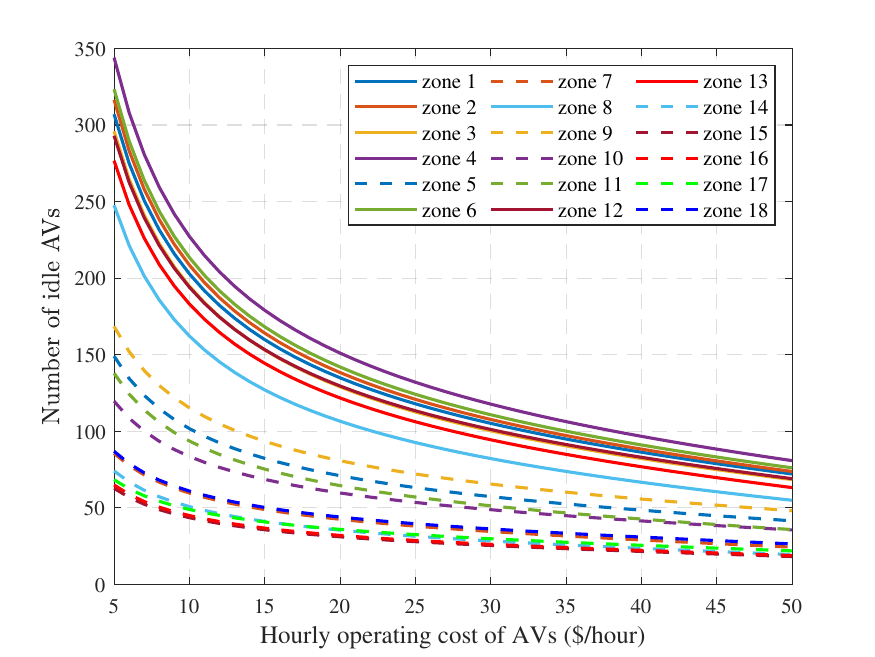}
         \caption{Spatial distribution of idle AVs under $a_2=0.6$.}
         \label{fig:idle_AV_scale_06}
     \end{subfigure}
     \begin{subfigure}[b]{0.32\textwidth}
         \centering
         \includegraphics[width=\textwidth]{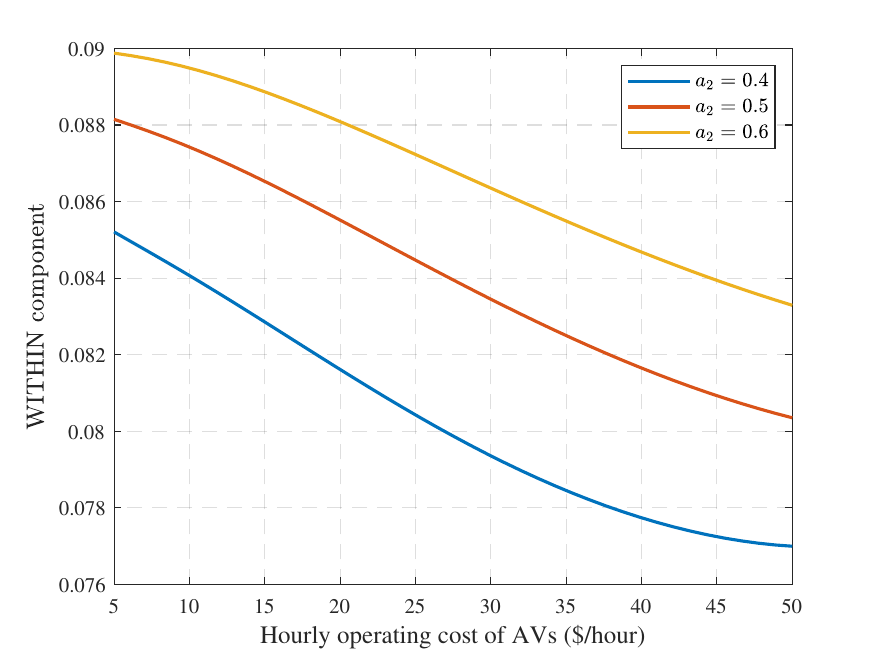}
         \caption{WITHIN component under distinct values of $a_2$.}
         \label{fig:WITHIN_scale}
     \end{subfigure}
     \hfill
     \begin{subfigure}[b]{0.32\textwidth}
     \centering
     \includegraphics[width=\textwidth]{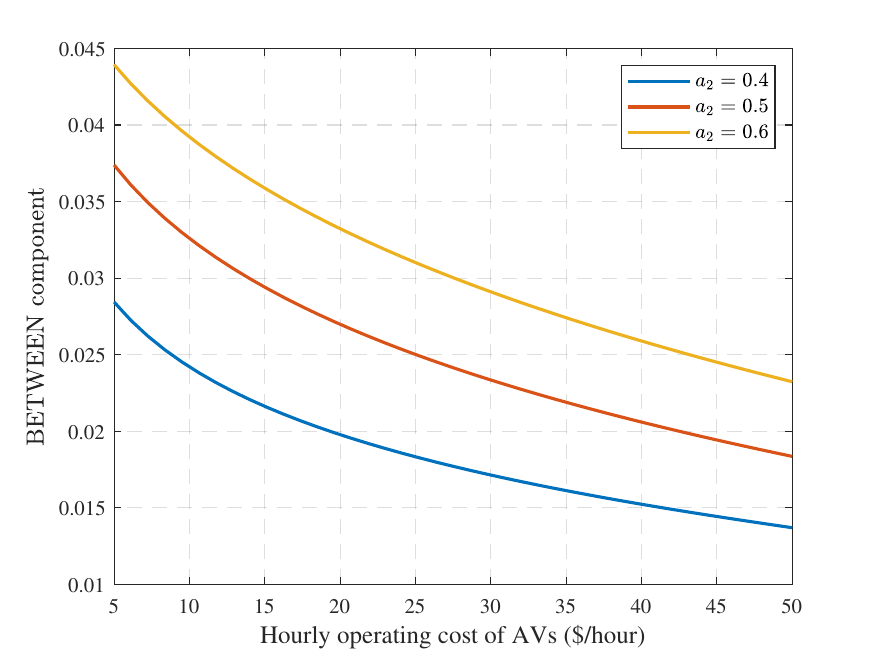}
     \caption{BETWEEN component under distinct values of $a_2$.}
     \label{fig:BETWEEEN_scale}
    \end{subfigure}
    \hfill
    \begin{subfigure}[b]{0.32\textwidth}
         \centering
         \includegraphics[width=\textwidth]{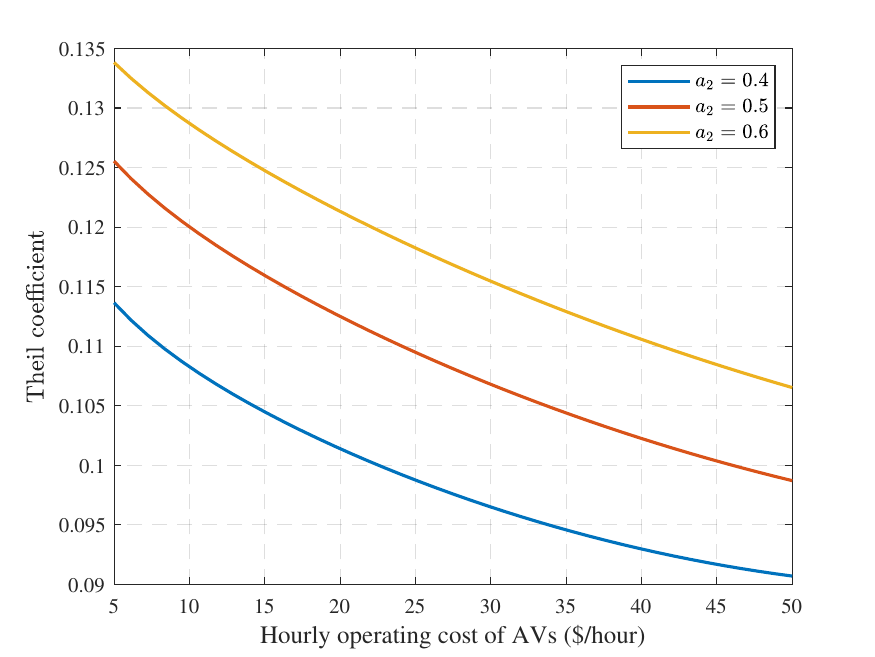}
         \caption{Theil coefficient under distinct values of $a_2$.}
         \label{fig:T_scale}
    \end{subfigure}
    \caption{The spatial distribution of idle vehicles and the equity results under distinct values of $a_2$.}
    \label{fig:economy_scale}
\end{figure}

\end{document}

%% file: format.tex
\usepackage{epsfig}
\usepackage{amssymb,amsmath,amsfonts,layout,graphicx}
\usepackage{makeidx}
\usepackage{babel}
\usepackage{tikz}
\usepackage{sublabel}
\usepackage{cases}
\usepackage{algorithm}
\usepackage{algorithmic}
\usepackage{comment}

\usepackage
[
        letterpaper,
        left=1.91cm,
        right=1.91cm,
        top=1.91cm,
        bottom=2cm,
]
{geometry}

\newtheorem{assumption}{Assumption}
\newtheorem{proposition}{Proposition}

\newtheorem{remark}{Remark}

\newcommand{\x}{{\bf x}}